\documentclass[final]{elsarticle}

\usepackage{amsmath}
\usepackage{epsfig}
\usepackage{latexsym}
\usepackage{bm}
\usepackage{verbatim}
\usepackage{epic}
\usepackage{eepic}
\usepackage{multirow,amssymb}
\usepackage{times}
\usepackage{latexsym,amssymb,amsmath,verbatim}
\usepackage{epsfig,epstopdf}
\usepackage{algorithm,algorithmicx,algpseudocode}
\usepackage{verbatim,booktabs}
\usepackage{mathptmx}
\usepackage{color}
\usepackage{slashbox}
\usepackage{amsthm}

\usepackage[pdftex, unicode=true,
            colorlinks,
            linkcolor=red,
            anchorcolor=blue,
            citecolor=green
            ]{hyperref}

\newtheorem{thm}{Theorem}[section]

\newtheorem{remark}[thm]{Remark}
\newtheorem{assumption}[thm]{Assumption}

\newcommand{\be}{\bm e}
\newcommand{\bb}{\bm b}
\newcommand{\bN}{\bm N}
\newcommand{\bff}{\bm f}
\newcommand{\bn}{\bm n}

\newcommand{\bW}{\bm W}
\newcommand{\bX}{\bm X}
\newcommand{\bx}{\bm x}

\newcommand{\bu}{\bm u}

\newcommand{\bv}{\bm v}

\newcommand{\bg}{\bm g}

\newcommand{\bS}{\bm \sigma}
\newcommand{\bD}{\bm D}

\newcommand{\bI}{\bm I}

\newcommand{\bepsilon}{\bm \varepsilon}

\graphicspath{ {./} }

\begin{document}
    \begin{frontmatter}

\title{A deep neural network/meshfree method for solving dynamic two-phase interface problems}

\author[dlu]{Xingwen Zhu}
\ead{zxw4688@126.com}
\address[dlu]{School of Mathematics and Computer Sciences, Dali University, 2 Hongsheng Street, Dali, Yunnan 671003, China}

\author[tufts]{Xiaozhe Hu\corref{cor}}
\ead{xiaozhe.hu@tufts.edu}
\address[tufts]{Department of Mathematics, Tufts University, Medford, Massachusetts 02155, USA}

\author[unlv]{Pengtao Sun\corref{cor}}
\ead{pengtao.sun@unlv.edu}
\address[unlv]{Department of Mathematical Sciences, University of Nevada Las Vegas, 4505 Maryland Parkway, Las Vegas, Nevada 89154, USA}

\cortext[cor]{Corresponding author}

\begin{abstract}
In this paper, a meshfree method using the deep neural network (DNN) approach is developed for solving two kinds of dynamic two-phase interface problems governed by different dynamic partial differential equations (PDEs) on either side of the stationary interface with the jump and high-contrast coefficients. The first type of two-phase interface problem to be studied is the fluid-fluid (two-phase flow) interface problem modeled by Navier-Stokes equations with high-contrast physical parameters across the interface. The second one belongs to fluid-structure interaction (FSI) problems modeled by Navier-Stokes equations on one side of the interface and the structural equation on the other side of the interface, both the fluid and the structure interact with each other via the kinematic- and the dynamic interface conditions across the interface, where the structural equation can be either the wave-type PDE with respect to the structural displacement or the parabolic-like PDE with respect to the structural velocity. The DNN/meshfree method is respectively developed for the above two-phase interface problems by representing solutions of PDEs using the DNNs' structure and reformulating the dynamic interface problem as a least-squares minimization problem based upon a space-time sampling point set (as the training dataset). Mathematically, the presented two kinds of two-phase interface problems own significant differences in PDEs' theories, resulting in distinct characteristics in their respective DNN/meshfree approaches whose approximation error analyses are also carried out for each kind of interface problem, which reveals an intrinsic strategy about how to efficiently build a sampling-point training dataset to obtain a more accurate DNNs' approximation. In addition, compared with traditional discretization approaches (e.g., finite element/volume/difference methods), the proposed DNN/meshfree method and its error analysis technique can be smoothly extended to many other dynamic interface problems with fixed interfaces. Numerical experiments are conducted to illustrate the accuracies of the proposed DNN/meshfree method for the presented two-phase interface problems. Theoretical results are validated to some extent through three numerical examples.
\end{abstract}

\end{frontmatter}

\section{Introduction}
Interface problems are ubiquitous in science and engineering applications. They model a variety of physical and chemical phenomena such as diffusion, electrostatics, heat transfer, fluid dynamics, elasticity, complex fluids, energy harvesting, and multi-phase flow in porous media (see, e.g., \cite{HANSBO20025537,doi:10.1137/0731054,doi:10.1137/130912700,MU2013106,peskin_2002,ZHOU20061}). Interface problems usually involve interfaces of complex geometries and strong physical interactions between two or more phases (e.g., fluid-fluid, fluid-structure, or structure-structure). Due to the complex nature of interface problems, numerical simulations are often the only possible way for scientific discovery, and numerical methods for solving such interface problems must be able to handle interfaces, interactions between different phases, and singularities of solutions due to high-contrast coefficients across interfaces accurately and efficiently.

One major approach to solve interface problems is the mesh-based numerical method that includes interface-fitted mesh methods, e.g., classical finite element methods~\cite{babuvska1970finite,bramble1996finite,Chen.Z;Zou.J1988a}, discontinuous Galerkin methods~\cite{massjung2012unfitted,cai2011discontinuous}, virtual element methods \cite{chen2017interface}, etc., and interface-unfitted mesh methods, e.g., immersed boundary methods \cite{peskin_2002}, immersed interface methods~\cite{doi:10.1137/0731054,Li.Z;Lai.M2001a}, immersed finite volume methods~\cite{OEVERMANN20095184,10.1007/978-3-540-78827-0_78}, matched interface and boundary methods \cite{ZHOU20061}, ghost fluid methods~\cite{liu2000boundary}, extended finite element methods~\cite{fries2010extended}, cut finite element methods~\cite{burman2015cutfem}, and immersed finite element methods~\cite{doi:10.1137/130912700}. The aforementioned methods have been extended to more complicated interface problems, e.g.~\cite{li2003overview, khoei2014extended,massing2015nitsche}. Although mesh-based methods have been successful in solving interface problems to a certain extent, the implementation of those numerical schemes is not a straightforward task due to jump conditions on the interface. In practice, interface problems remain quite difficult because of the complicated geometry of interfaces that oftentimes are dynamically changing, as well as the singularities introduced by interface conditions and jump coefficients.

Another approach to solve interface problems is the so-called meshfree method. It is attractive due to its ability to circumvent costly management of the complex geometry of interfaces. Meshfree methods have been applied to computational solid mechanics~\cite{liu2016overview}, fracture mechanics~\cite{daxini2014review}, and fluid-structure interaction~\cite{mazhar2021meshfree,hu2019consistent,hu2019spatially}. A comprehensive review of applications of meshfree methods can be found in~\cite{garg2018meshfree}. Recently, a special type of meshfree method, which uses the deep neural networks (DNNs) for approximations, has drawn increasing attention for solving PDEs, especially those challenging ones which cannot be handled by existing numerical methods robustly and efficiently, e,g.,~\cite{CaiChenLiuLiu2019,Dissanayake1994,EYu,HanJentzenE,HeLiXuZheng,SamaniegoGoswami,SirignanoSpilopoulos,TranHamiltonMckayQuiringVassilevski}. More recently, DNNs have been used to solve second-order elliptic interface problems, see~\cite{EYu,CaiChenLiuLiu2019,WangZhang,HeLinHu2022}. These methods are based on rewriting the second-order elliptic interface problem as a minimization problem. In particular, a specific type of least-squares (LS) formulation is used to define the total loss functional, and DNN structures are applied to approximate the solution of the interface problem in each subdomain separated by the interface.

In this work, motivated by the previous work~\cite{HeLinHu2022}, we consider studying general dynamic two-phase interface problems by developing corresponding DNN/meshfree methods. In contrast to the standard second-order elliptic interface problem, the interface problems to be studied in this paper involve different types of PDEs in each subdomain to describe interactions between different kinds of phases (e.g., fluid-fluid or fluid-structure). Thus various kinds of interface conditions are employed across the stationary interface. Following the idea from~\cite{HeLinHu2022}, we can use different NN structures in different subdomains to approximate solutions to two-phase interface problems. This approach provides flexibility to handle different types of PDEs in different subdomains and sophisticated interface conditions across interfaces that make different physical quantities interact with each other. In particular, we rewrite the dynamic interface problem into an LS formulation and then define a discrete LS problem via sampling points in the entire computational domain that is actually a space-time domain in $\mathbb{R}^{d+1}\ (d=2,3)$, where residuals of the PDEs defined in each $(d+1)$-dimensional interior subdomain, of the boundary conditions defined on all $d$-dimensional boundaries, of the interface conditions defined on all $d$-dimensional interfaces, and of the initial conditions defined in each initial subdomain of dimension $d$, are reformulated respectively through the (discrete) LS formulation, then are added up to form a total (discrete) loss functional. Piecewise NN structures are used to approximate solutions of different PDEs in each subdomain. Finally, the discrete LS problem can be solved by the stochastic gradient descent (SGD) method \cite{Robbins1951} or its variants.

Following the recent theoretical results~\cite{mishra2022estimates}, we also provide an error analysis of the proposed DNN/meshfree method for each studied interface model problem. Our error estimates decompose the approximation error into two parts. One is the quadrature error since we replace all integral terms in the LS type loss functional by quadrature rules to derive the corresponding discrete LS problem. Another part is the optimization error caused by approximately solving the discrete LS problem in practice. According to our analysis, when the optimization error part is negligible, the overall approximation error is usually dominated by the quadrature errors coming from the LS formulations of boundary and initial conditions~\cite{mishra2022estimates}, as well as from the LS formulation of interface conditions which was previously ignored in the literature since sampling points on the interface are simply taken from their adherence to the interior sampling-point training set inside the space-time domain. Thus, much fewer sampling points regarding interface conditions are usually used to solve the discrete LS type loss functional for interface problems, leading to an error dominance in the DNN's approximation. Our numerical experiments verify our theoretical findings by carrying out the proposed DNN/meshfree approach for the two-phase flow interface problem and the FSI problem in two different formulations,  where we can observe that the DNN's approximation errors decrease when we increase the number of sampling points on the interfaces, the boundaries, and in the initial subdomains while keeping the number of sampling points inside the space-time domain relatively large. Thus, the effectiveness and capability of the developed DNN/meshfree approach for solving complicated interface problems are validated.

The rest of the paper is organized as follows. In Section~\ref{sec:model} we introduce notations and definitions for a general model of dynamic interface problems to be studied and introduce two concrete examples: the two-phase flow interface problem and the FSI problem in two different formulations. The DNN/meshfree method for solving the presented interface problems is then introduced in Sections~\ref{sec:DNN-method} and \ref{sec:error}, where we also provide a framework for the error analysis of the proposed DNN/meshfree method for each studied dynamic interface problem. Finally, in Section~\ref{sec:numerics}, we present numerical experiments using three numerical examples to illustrate the effectiveness of the proposed DNN/meshfree method for solving two kinds of dynamic two-phase interface problems. Concluding remarks and a discussion of future work are given in Section~\ref{sec:conclusions}.

\section{Model description} \label{sec:model}
\subsection{Models of interface problems}
In this paper, we start our studies on the DNN/meshfree method for general dynamic interface problems with fixed interfaces, which has wide applications in different areas, such as in material sciences and fluid dynamics, while discontinuous coefficients are encountered across the interface, or, more complicated than that, different types of governing equations are defined on either side of the interface interacting with each other through interface conditions. It happens when two distinct materials/phases (fluid or structure) involve different conductivities/densities/viscosities/diffusions. One typical example is the two-phase flow interface problem, which refers to an interactive flow of two distinct phases with common interfaces in a channel/container. Each phase represents a mass or volume of matter. An immiscible, incompressible two-phase flow interface problem without a phase change can be modeled by two sets of Navier-Stokes equations with the jump and high-contrast coefficients across the interface. Another example is the fluid-structure interaction problem that couples fluid dynamics with structural mechanics through interface conditions defined on interfaces between the fluid and the structure, which is modeled by Navier-Stokes equations and the elasticity equation on either side of the interfaces, respectively, while interactional effects between the fluid and the structure are described by the kinematic (continuity of velocity) and the dynamic (balance of force) interface conditions across the interface.

For the sake of generality, we consider the following abstract interface model problem, which couples two types of PDEs (possibly nonlinear) and whose differential operators are denoted by $\mathcal{L}_1$ and $\mathcal{L}_2$, respectively,
\begin{equation}\label{eqn:interface-model}
\left\{
    \begin{array}{rcll}
        \mathcal{L}_i(\bm{u}_i) &=& \bm{0}, &  \text{in} \ \Omega_i \times (0,T], \ i = 1,2, \\
        \mathcal{L}_{\Gamma}(\bm{u}_1, \bm{u}_2) &=& \bm{0}, & \text{on} \ \Gamma \times [0, T], \\
        \mathcal{B}_i(\bm{u}_i) &=& \bm{0},  & \text{on}  \ \partial \Omega_i \backslash \Gamma \times [0,T], \ i = 1,2, \\
        \mathcal{I}_i(\bm{u}_i) &=& \bm{0}, & \text{in} \ \Omega_i,\ i=1,2,
    \end{array}
\right.
\end{equation}
where $\Omega$ is an open bounded domain in $\mathbb{R}^d\ (d=2,3)$ with a convex polygonal boundary $\partial\Omega$, and consists of two subdomains $\Omega_i\subset \Omega\ (i=1,2)$ separated by a stationary interface $\Gamma$, satisfying that ${\Omega}={\Omega_1}\cup{\Omega_2}$, $\Omega_1\cap\Omega_2=\emptyset$, $\Gamma=\partial\Omega_1\cap\partial\Omega_2$. In addition, $\mathcal{L}_{\Gamma}$ represents the differential operator of abstract interface conditions on $\Gamma\times [0,T]$ which could be nonlinear, $\mathcal{B}_{i}$ and $\mathcal{I}_{i}$ denote differential operators of abstract boundary conditions on $\partial \Omega_i \backslash \Gamma\times [0,T]\ (i=1,2)$, and of abstract initial conditions in $\Omega_i\ (i=1,2)$ at $t=0$, respectively. Possible configurations are illustrated in Figure \ref{fig:domain}.

\begin{figure}[hbt]
    \begin{center}
        \includegraphics[height=3cm]{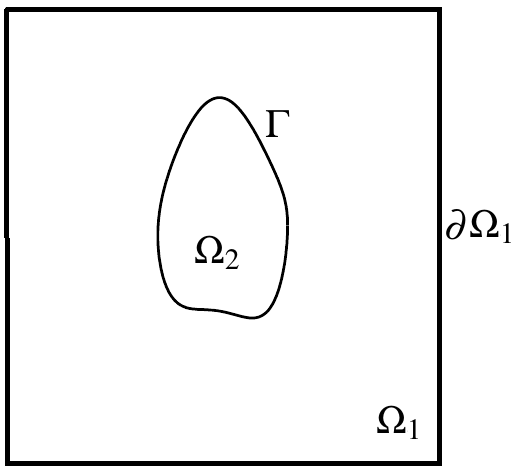}
        \hspace{2cm}
        \includegraphics[height=3cm]{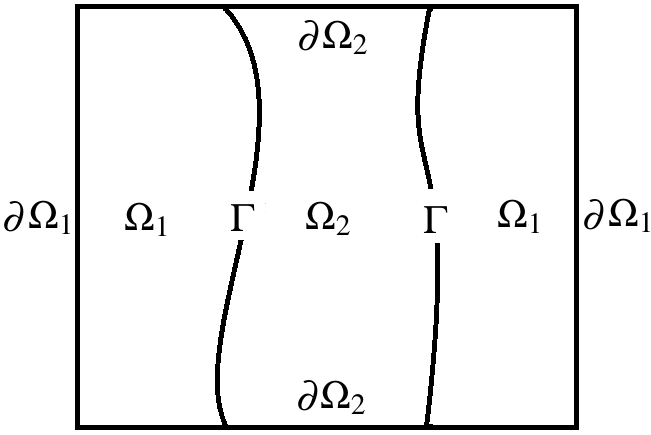}
    \end{center}
    \caption{Two schematic domain decompositions divided by the
        interface $\Gamma$: the immersed case (\textbf{left}) and the
        back-to-back case (\textbf{right}).}\label{fig:domain}
\end{figure}

Although the abstract model~\eqref{eqn:interface-model} consists of two different PDEs that are respectively defined in two subdomains through interface conditions across one interface, we want to point out that our proposed numerical approach is flexible and can be easily extended to the case where more than two  PDEs in their respective subdomains are coupled together with linear or nonlinear interface conditions across multiple interfaces. Moreover, it is also possible to consider the case where PDEs are defined on different dimensions, i.e., mixed-dimensional interface problems. For the sake of simplicity, we just choose the simple abstract model~\eqref{eqn:interface-model} to propose the numerical methodology of our DNN/meshfree method in this paper.

In what follows, we use the standard notation for Sobolev spaces~$W^{l,p}(\Psi)\ (0\leq l< \infty,1\leq p\leq\infty)$ and their associated norms $\| \cdot \|_{W^{l,p}(\Psi)}$.  For $p=2$, we use the standard notation $W^{l,2}(\Psi)=H^l(\Psi)$.  When $l=0$, $H^0(\Psi)$ coincides with the standard $L^2$ space, i.e., $L^2(\Psi)$.  Since we consider time-dependent problems, we also use the standard notation for the time-dependent Sobolev spaces $W^{m,q}([0,T];W^{l,p}(\Psi))\ (0\leq m<\infty, 1\leq q \leq \infty)$.  When $p=q=2$, we use the standard notation $H^m([0,T]; H^l(\Psi))$. Furthermore, if $m=0$, we use the notation $L^2([0,T]; H^l(\Psi))$.  In addition, we also allow $\Psi$ to be a space-time domain $\mathfrak{D}$ and the corresponding space-time Sobolev spaces' notation follow the same convention. Finally, we also consider the space of the continuously differentiable functions defined in a space-time domain $\mathfrak{D}$ and use the standard notation $C^k(\mathfrak{D}) \ (0 \leq k \leq \infty)$ with the associated norm $\| \cdot \|_{C^k(\mathfrak{D})}$.

\subsection{Examples of interface problems} \label{sec:model:subsec:example}
Now we present two concrete examples of interface problems described by the abstract model \eqref{eqn:interface-model}. The first example is a two-phase flow interface problem that couples two sets of Navier-Stokes equations in two subdomains with suitable interface conditions across the interface. The second example is the FSI problem that consists of Navier-Stokes equations on one side of the interface and the elasticity equation on the other side of the interface, and, both the fluid and the structure interact with each other through the kinematic and the dynamic interface conditions defined on the fluid-structure interface. These two kinds of problems are typical examples of dynamic two-phase interface problems and have a wide range of applications in practice (see, e.g., \cite{Wang;Sun2016,Takizawa2011b,Chakrabarti2005,Dowell2001,Seo2013,Loon2005,Turek.S;Hron.J2010a,Turek.S;Hron.J2006a}).

\subsubsection{The two-phase flow interface problem} \label{sec:NS-NS-example}
The first example we consider is the two-phase flow interface problem. In this case, two sets of Navier-Stokes equations are coupled together through proper interface conditions with jump coefficients across the interface. More precisely, the positive, jump piecewise constants $\rho|_{\Omega_i}=\rho_i,\ \mu|_{\Omega_i}=\mu_i\ (i=1,2)$ represent the density and the dynamic viscosity of two kinds of fluids, respectively, leading to different velocities, $\bm{v}_i$, and different pressures, $p_i$, in $\Omega_i\ (i=1, 2)$, respectively. The stress rate tensor of every kind of fluid is defined as follows,
$$
\bS_i:=-p_i\bI+2\mu_i\bD(\bv_i),\quad\text{where }
\bD(\bv_i):=\frac{1}{2}\left(\nabla\bv_i+(\nabla\bv_i)^T\right)
\quad \text{for }i=1,2.
$$
Introduce $\bm{u}_1 := (\bv_1, p_1)$ and $\bm{u}_2 := (\bv_2, p_2)$. $\mathcal{L}_i\ (i=1, 2)$, which represent differential operators of Navier-Stokes equations for each fluid phase in $\Omega_i\times[0,T]\ (i=1, 2)$, are defined below,
\begin{equation*}
    \mathcal{L}_i(\bm{u}_i) :=
    \begin{pmatrix}
        \rho_i \left(\frac{\partial\bv_i}{\partial
            t}+\bv_i\cdot\nabla\bv_i\right)-\nabla\cdot\bS_i - \bff_i \\
        \nabla \cdot \bv_i
    \end{pmatrix},
\qquad i = 1,2.
\end{equation*}
The differential operator of interface conditions on~$\Gamma\times[0,T]$, $\mathcal{L}_{\Gamma}$, is defined as
\begin{equation} \label{eqn:two-phase-flow-interface}
    \mathcal{L}_{\Gamma}(\bm{u}_1, \bm{u}_2) :=
    \begin{pmatrix}
        \bv_1 - \bv_2 - \bg_1 \\
        \bS_1\bn_1+\bS_2\bn_2 - \bg_2
    \end{pmatrix},
\end{equation}
where $\bn_1=-\bn_2$ is the unit outward normal vector on the interface $\Gamma$ pointing to $\Omega_2$ from $\Omega_1$ (see Figure~\ref{fig:domain}). Here, the first component of $\mathcal{L}_{\Gamma}$, which comes from the kinematic interface condition, describes the continuity ($\bm{g}_1 = \bm{0}$) or the jump ($\bm{g}_1 \neq \bm{0}$) of velocities, and, the second component of $\mathcal{L}_{\Gamma}$, coming from the dynamic interface condition,  describes the continuity ($\bm{g}_2 = \bm{0}$) or the jump ($\bm{g}_2 \neq \bm{0}$) of the flux across the interface $\Gamma$. In many realistic two-phase flow interface problems, we usually take $\bm{g}_1 = \bm{g}_2 = \bm{0}$ and then~\eqref{eqn:two-phase-flow-interface} becomes the so-called no-slip interface conditions.   As for the slip-type interface conditions, they are defined as $\bg_1\cdot\bn_1= 0$ while $\bm{g}_1 - (\bm{g}_1 \cdot \bm{n_1})\bm{n}_1 \neq \bm{0}$ and $\bg_2\neq \bm{0}$, in general.

For the simplicity, we consider the standard Dirichlet boundary conditions on the rest boundaries~$\partial\Omega_i\backslash\Gamma\times[0,T]\ (i=1,2)$, whose differential operator is defined as
\begin{equation*}
    \mathcal{B}_i(\bm{u}_i) := \bm{v}_i - \bm{v}_i^b, \quad i = 1,2,
\end{equation*}
and consider the standard initial condition in $\Omega_i\ (i=1,2)$ at $t=0$, whose differential operator is defined as
\begin{equation*}
    \mathcal{I}_i(\bm{u}_i) := \bm{v}_i(\bm{x}_i, 0) - \bm{v}_i^0(\bm{x}_i), \quad i = 1,2.
\end{equation*}
Of course, other types of boundary conditions (such as Neumann/Robin conditions) could also be specified, depending on the realistic scenarios.

\subsubsection{The fluid-structure interaction problem} \label{sec:FSI-example}
The second example we consider is the FSI problem that usually models the interactions between a free viscous fluid flow and an elastic structure, giving rise to a wide variety of physical phenomena with applications in many fields of science and engineering, such as the vibration of rotating turbine blades impacted by the fluid flow, the response of bridges and tall buildings to winds, the floating parachute wafted by the air current, the blood flow through arteries, etc. FSI problems are challenging for numerical simulations because physical parameters across the interface might be discontinuous and high-contrast, governing equations from either side of the interface are completely different, and more difficultly, fluid-structure interfaces are usually moving (deforming/translating/rotating) all the time. Nevertheless, as the first step of our study on the DNN/meshfree method for solving FSI problems, we assume the interface between the fluid and the structure is stationary only in this paper. As for the DNN/meshfree method for the FSI problem with a moving interface, it is more challenging  since the undetermined  moving interface breaks the connectivity of the sampling points set in each subdomain along the time, which introduces difficulties to the time difference scheme while DNN is used in the spatial approximation only. We will report our achievements on this topic in our subsequent work, comprehensively.

As illustrated in Figure \ref{fig:domain}, we let $\Omega_1 := \Omega_f$ and $\Omega_2 := \Omega_s$ denote the fluid subdomain and the structural subdomain, respectively. The open bounded domain ${\Omega}={\Omega_f}\cup{\Omega_s}\subset \mathbb{R}^d \ (d=2,3)$ is equipped with a convex polygonal boundary $\partial\Omega$. Moreover, $\Omega_f\cap\Omega_s=\emptyset$, and their boundaries share the common interface $\Gamma:=\partial\Omega_f\cap\partial\Omega_s$. In addition, we let $\bm{u}_1 :=\bu_f:= (\bv_f, p_f)$ denote the fluid velocity and the fluid pressure defined in $\Omega_f$, and $\bm{u}_2 := \bu_s$ denote the structural displacement defined in $\Omega_s$. Therefore, we introduce the following differential operator of the usual Navier-Stokes equations in $\Omega_f$,
\begin{equation*}
    \mathcal{L}_1(\bm{u}_1) := \mathcal{L}_f(\bu_f)  :=
    \begin{pmatrix}
        \rho_f\left(\frac{\partial\bv_f}{\partial
            t}+\bv_f\cdot\nabla\bv_f\right)-\nabla\cdot\bS_f - \bff_f\\
        \nabla\cdot\bv_f
    \end{pmatrix},
\end{equation*}
where $\bff_f$ is the external force in $\Omega_f$, the fluid stress rate tensor, $\bS_f$, is defined as
$$
\bS_f:=-p_f\bI+2\mu_f\bD(\bv_f),\quad\text{and }
\bD(\bv_f):=\frac{1}{2}\left(\nabla\bv_f+(\nabla\bv_f)^T\right),
$$
where positive constants $\rho_f$ and $\mu_f$ are the density and the dynamic viscosity of the fluid, respectively.

As for the structural equation, one possible choice is the following differential operator of wave equation with respect to the structural displacement $\bm{u}_s$,
\begin{equation}\label{structure_wave}
    \mathcal{L}_2(\bm{u}_2) := \mathcal{L}_s(\bm{u}_s) := \rho_s\frac{\partial^2 \bu_s}{\partial t^2}-\nabla\cdot{\bS}_s -
    \bff_s,
\end{equation}
where $\rho_s$ is the structural density and $\bff_s$ the external force in $\Omega_s$. The structural (Cauchy) stress tensor, $\bS_s$, is defined in terms of the structural displacement $\bu_s$, i.e.,
$$
\bS_s(\bu_s):=2\mu_s\bepsilon(\bu_s)+\lambda_s\text{tr}(\bepsilon(\bu_s))\bI,\quad\text{and
}
\bepsilon(\bu_s):=\frac{1}{2}\left(\nabla\bu_s+(\nabla\bu_s)^T\right),
$$
where $\lambda_s$ and $\mu_s$ are Lam\'e's first and second constants, and~$\text{tr}(\bepsilon(\bu_s))=\nabla\cdot\bu_s$.

In this case, the coupling of $\mathcal{L}_f$ and $\mathcal{L}_s$ essentially forms the differential operators of governing equations for a Navier-Stokes/wave interface problem. Therefore, the corresponding interface conditions' differential operator can be defined as follows,
\begin{equation*}
	\mathcal{L}_{\Gamma}(\bm{u}_1, \bm{u}_2) :=
	\begin{pmatrix}
		\bv_f-\frac{\partial\bu_s}{\partial t} - \bg_1 \\
		\bS_f\bn_f+\bS_s\bn_s - \bg_2
	\end{pmatrix}.
\end{equation*}
Note that two components of $\mathcal{L}_{\Gamma}$ define the differential operator of the kinematic and dynamic interface conditions of FSI, respectively, which describe the continuity ($\bg_1= \bm{0}$) or the jump ($\bg_1\neq \bm{0}$) of fluid and structural velocities and the balance ($\bg_2=0$) or the jump ($\bg_2\neq \bm{0}$) of traction forces across the fluid-structure interface $\Gamma$. In addition, differential operators of Dirichlet boundary conditions on the rest boundaries and of standard initial conditions in the initial subdomains are defined as follows,
\begin{align*}
    &\mathcal{B}_1(\bm{u}_1) := \mathcal{B}_f(\bm{u}_f) := \bv_f - \bv_f^b,  \quad \text{on} \ \partial \Omega_f\backslash \Gamma \times [0,T], \\
    &\mathcal{B}_2(\bm{u}_2) := \mathcal{B}_s(\bm{u}_s) := \bu_s - \bu_s^b, \quad \text{on} \ \partial \Omega_s \backslash \Gamma \times
    [0,T],\\
    &\mathcal{I}_1(\bm{u}_1) := \mathcal{I}_f(\bm{u}_f):= \bv_f(\bx_f,0) - \bv_f^0, \quad \text{in} \ \Omega_f, \\
    &\mathcal{I}_2(\bm{u}_2) := \mathcal{I}_s(\bm{u}_s) :=\begin{pmatrix}
            \mathcal{I}_{s, \bu}(\bm{u}_s) \\
            \mathcal{I}_{s, \bv}(\bm{u}_s)
        \end{pmatrix}
    :=\begin{pmatrix}
            \bu(\bx_s,0) - \bu^0_s \\
            \frac{\partial\bu_s}{\partial t} - \bv_s^0
        \end{pmatrix}, \quad \text{in} \ \Omega_s.
\end{align*}

On the other hand, instead of the structural displacement $\bu_s$, it is also possible to use the structural velocity $\bm{v}_s$ as the primary unknown to define the FSI model. Note that $\bv_s=\frac{\partial\bu_s}{\partial t}$ in $\Omega_s$, we can reformulate the differential operator of structural equation, $\mathcal{L}_s$, as follows using $\bm{u}_2 := (\bu_s, \bv_s)$,
\begin{eqnarray}\label{structure_par}
    \mathcal{L}_2(\bm{u}_2) :=
 \mathcal{L}_s(\bm{u}_2)  :=
 \begin{pmatrix}
 \rho_s\frac{\partial{\bm v}_s}{\partial
        t}-\nabla\cdot\bS_{s}(\bu_s) - {\bff}_s \\
    \frac{\partial \bu_s}{\partial t} - \bv_s
 \end{pmatrix},
\end{eqnarray}
In this case, the structural equation is actually a parabolic-like equation with respect to $\bv_s$, where $\bu_s=\int_0^t\bv_s(\tau)d\tau+\bu^0_s$~\cite{Wang;Sun2016,JLiu2016}. Accordingly, the differential operators of interface conditions, boundary conditions, and initial conditions need to be rewritten as follows,
\begin{equation*}
\begin{array}{l}
    \mathcal{L}_{\Gamma}(\bm{u}_1, \bm{u}_2) :=
    \begin{pmatrix}
        \bv_f-\bv_s - \bg_1 \\
        \bS_f\bn_f+\bS_s\bn_s - \bg_2,
    \end{pmatrix}, \
    \mathcal{B}_2(\bm{u}_2) := \mathcal{B}_s(\bm{u}_2) := \bv_s -
    \bv_s^b,\\
\
    \mathcal{I}_2(\bm{u}_2) :=
    \mathcal{I}_s(\bm{u}_2) :=\begin{pmatrix}
            \mathcal{I}_{s, \bu}(\bm{u}_s) \\
            \mathcal{I}_{s, \bv}(\bm{v}_s)
        \end{pmatrix}:=
    \begin{pmatrix}
        \bu_s(\bx_s,0) - \bu_s^0 \\
        \bv_s(\bx_s,0) - \bv_s^0. 
    \end{pmatrix}.
    \end{array}
\end{equation*}
Thus, a Navier-Stokes/parabolic-like interface problem is formed. In this paper, we consider both the Navier-Stokes/wave-type and the Navier-Stokes/parabolic-like interface problem as FSI models, which are often used to describe FSI problems. And through these two FSI models, we demonstrate that our proposed DNN/meshfree method is flexible and can easily handle different formulations of the FSI model.

\section{DNN/meshfree method for the interface problem} \label{sec:DNN-method}
This section introduces a type of DNN method for solving the abstract interface problem~\eqref{eqn:interface-model}. Similar to existing methods for solving PDEs via neural networks~\cite{karniadakis2021physics,shukla2022scalable}, our method can be viewed as a meshfree approach with more flexibility to handle interfaces and interface conditions than the standard numerical methods for solving PDEs, such as finite element methods. Following our recent work~\cite{HeLinHu2022}, the main idea is to rewrite the interface problem as a minimization problem using LS formulation and use different DNN structures to approximate solutions in different subdomains. For the abstract interface problem~\eqref{eqn:interface-model}, we use two DNN structures suitable for different interface problems with large jumps in their solutions and derivatives of their solutions. Our DNN/meshfree approach allows us to remove the computational mesh that turns out to be troubl]esome for both body-fitted and body-unfitted mesh methods in numerical simulations of interface problems and to easily handle complicated interfaces by adequately choosing the sampling points without the needs for an underlying mesh and an accurate representation of the interface respective to the mesh.

\subsection{The least-squares formulation}
To apply the DNN approach to the abstract interface problem~\eqref{eqn:interface-model}, it is natural to rewrite PDEs as a minimization problem. There are different ways to do this; see~\cite{Dissanayake1994,bochev2009least}.  In this work, we adopt the least-squares (LS) formulation \cite{Dissanayake1994}  which minimizes residuals of the abstract interface problems~\eqref{eqn:interface-model} by defining an LS functional that incorporates all equations introduced in~\eqref{eqn:interface-model} together.  This approach has been widely used in physics-informed neural networks (PINN)~\cite{karniadakis2021physics} for various PDEs, including the elliptic interface problem~\cite{HeLinHu2022}.

In particular, considering the abstract model problem~\eqref{eqn:interface-model}, we define a space-time LS formulation as follows,
\begin{equation}\label{eqn:LS}
\mathcal{R}(\bm{\tilde{u}}_1, \bm{\tilde{u}}_2) := \int_{0}^T \left(
\sum_{i=1}^2 \omega_{\mathcal{L}_i} \|
\mathcal{L}_i(\bm{\tilde{u}}_i) \|_{0,\Omega_i}^2 + \omega_{\Gamma}
\|\mathcal{L}_{\Gamma}(\bm{\tilde{u}}_1, \bm{\tilde{u}}_2) \|_{0,
\Gamma}^2 + \sum_{i=1}^{2}\omega_{\mathcal{B}_i} \|
\mathcal{B}_i(\bm{\tilde{u}}_i) \|_{0, \partial\Omega_i \backslash
\Gamma}^2 \right) \, \mathrm{d}t + \sum_{i=1}^{2}
\omega_{\mathcal{I}_i} \| \mathcal{I}_i(\bm{\tilde{u}}_i)
\|_{0,\Omega_i}^2,
\end{equation}
where $\omega_{\mathcal{L}_i}$, $\omega_{\Gamma}$, $\omega_{\mathcal{B}_i}$, and $\omega_{\mathcal{I}_i}$ are prescribed weight coefficients of each corresponding term. Then, the LS solution associated with the LS functional~\eqref{eqn:LS} is to find $\bm{u}_1 \in \mathbb{V}_1$ and $\bm{u}_2 \in \mathbb{V}_2$ such that
\begin{equation}\label{eqn:minLS}
\mathcal{R}(\bm{u}_1,\bm{u}_2)\notag =\min\limits_{\bm{\tilde{u}}_1
\in \mathbb{V}_1, \bm{\tilde{u}}_2 \in \mathbb{V}_2}
\mathcal{R}(\bm{\tilde{u}}_1,\bm{\tilde{u}}_2),
\end{equation}
where $\mathbb{V}_1$ and $\mathbb{V}_2$ are suitable spaces to which the LS solution belongs.

\begin{remark}\label{rmk:weights}
An optimal choice for all weight coefficients $\omega_i$ in~\eqref{eqn:LS} can make the LS optimization process converge faster, i.e., accelerate the minimization of each LS term in~\eqref{eqn:LS} by letting it approach zero as quickly as possible. In practice, a possible way to choose each weight coefficient associated with its $L^2$-inner product or LS term is the reciprocal of the maximum value of all parameters in this LS term, just like a normalization process. On the other hand, a unified dimension for all LS terms in the total LS functional is another important consideration for choosing these weight coefficients to accelerate the LS minimization for a coupled PDEs that bears high-contrast magnitudes for different variables. Overall, those weights play an essential role and can be predetermined or self-learned, see, e.g.~\cite{wang2021understanding,wang2022and,mcclenny2020self}.
\end{remark}

\subsection{The DNN/meshfree method}
Our meshfree method adopts DNN structures to approximate the abstract interface problem~\eqref{eqn:interface-model} in each space-time subdomain and on boundaries/interfaces/initial subdomains without physically generating any mesh. More precisely, we employ fully connected DNN structures in a $(d+1)$ dimensional space (in the space-time sense) to approximate primary variables (velocity, pressure, etc.) of the interface problem~\eqref{eqn:interface-model} and minimize the LS formulation~\eqref{eqn:LS} via sampling spatial and temporal points together in the computational space-time domain. Our approach follows the idea proposed in~\cite{HeLinHu2022} and uses two neural networks to approximate solutions in two subdomains. This subsection briefly reviews the fully connected DNN structure and then introduces the discrete LS formulation for \eqref{eqn:interface-model} based on DNN structures.

We use a simple DNN structure here which consists of multiple linear transformations and nonlinear activation functions.  A linear transformation $\mathbf{T}^l:\mathbb{R}^{n_l}\rightarrow \mathbb{R}^{n_{l+1}} \ (l=1,\cdots,L)$ is define as follows,
$$
\mathbf{T}^l(\bX^l):=\bW^l\bX^l+\bb^l, \quad\text{ for } \bX^l\in
\mathbb{R}^{n_l},
$$
where $\bW^l \in \mathbb{R}^{n_{l+1} \times n_l}$ denotes the weights and $\bb^l \in \mathbb{R}^{n_{l+1}}$ denotes the bias. Together with an activation function $\sigma:\mathbb{R}\rightarrow\mathbb{R}$, the {$l$}-th layer neural network, $\mathbf{\bN}^l:\mathbb{R}^{n_l}\rightarrow \mathbb{R}^{n_{l+1}}$, is defined as follows,
$$
\mathbf{\bN}^l(\bX^l):=\sigma(\mathbf{T}^l(\bX^l)), \quad \text{ for
} l=1,\cdots,L,
$$
where the nonlinear activation function $\sigma$ is applied in a component-wise fashion.  Then a general {$L$}-layer DNN is defined as
\begin{equation} \label{eqn:NN-structure}
\mathbf{\bN\bN}(\bX;\Theta):=\mathbf{T}^L\circ{\mathbf{\bN}}^{L-1}\circ\cdots\circ{\mathbf{\bN}}^2\circ{\mathbf{\bN}}^1(\bX),
\end{equation}
where $\bX\in \mathbb{R}^{n_1}$ denotes the input (training set) of the neural network, $\Theta$ stands for all the weights $\bW^l$ and the bias $\bb^l\ (l=1,\cdots,L)$ arising from all $L$ layers, i.e., $\Theta:=\{\bW^l,\bb^l,\ l=1,\cdots,L\}$. Thus, {$\Theta\in \mathbb{R}^N$}, and {$N=\sum\limits_{l=1}^Ln_{l+1}(n_l+1)$} for a fully connected DNN.

In our meshfree method for solving the interface problem~\eqref{eqn:interface-model}, we use the fully connected DNN structure~\eqref{eqn:NN-structure} to approximate solutions of PDEs. In particular, since we use a space-time LS formulation~\eqref{eqn:LS}, the input of the neural network should be {$\bX:=(\bx,t)\in\overline{\Omega}\times[0,T]\subset \mathbb{R}^{d+1}$}, thus we adopt $n_1=d+1$ for the DNN structure~\eqref{eqn:NN-structure}, and define the following DNN functions to approximate $\bm{u}_1$ and $\bm{u}_2$ in~\eqref{eqn:interface-model}, respectively,
\begin{equation} \label{eqn:uNNs}
    \bm{u}_1 \approx \mathcal{U}_{\mathbf{\bN\bN}}^{1} (\bX^{(1)};\Theta),
    \quad \bm{u}_2 \approx \mathcal{U}_{\mathbf{\bN\bN}}^{2}
    (\bX^{(2)};\Theta),
\end{equation}
where
$\bX^{(i)}:=(\bx_i,t)\in\overline{\Omega_i}\times[0,T]\in\mathbb{R}^{d+1}$,
$i=1,2$, and $\bX \vert_{\overline{\Omega_i}\times[0,T]} =
\bX^{(i)}$.

Naturally, based on the LS formulation~\eqref{eqn:LS} and the DNN approximation~\eqref{eqn:uNNs}, a DNN/meshfree method is defined as to find $\mathcal{U}_{\mathbf{\bN\bN}}^{1} (\bX^{(1)};\Theta)$ and~$\mathcal{U}_{\mathbf{\bN\bN}}^{2} (\bX^{(2)};\Theta)$ such that
\begin{equation}\label{minLoss-twophase}
\min\limits_{\Theta\in\Psi^N}\mathcal{R}(\bX;\Theta),\quad
\text{where }\Psi^N:=\{\Theta:\Theta|_{\overline{\Omega_i}\times
[0,T]}\in \mathbb{R}^N,\ \overline\Omega_1\cup
\overline\Omega_2:=\overline{\Omega}\},
\end{equation}
and the total loss functional, $\mathcal{R}(\bX;\Theta)$, according to the LS functional~\eqref{eqn:LS}, is defined as
\begin{equation}\label{Loss-twophase}
\begin{array}{rcl}
\mathcal{R}(\bX;\Theta)
:=\mathcal{R}(\mathcal{U}_{\mathbf{\bN\bN}}^{1}(\bm{X};
\Theta),
\mathcal{U}_{\mathbf{\bN\bN}}^{2}(\bX; \Theta)).
\end{array}
\end{equation}
Since the loss functional~\eqref{Loss-twophase} contains the space-time integral and cannot be computed exactly, we approximate the space-time integral by its discrete counterpart based on a set of sampling points, i.e., by means of the Monte Carlo integration. More precisely, we introduce the following mean square error (MSE) formulations to approximate corresponding space-time integrals in~\eqref{eqn:LS},
\begin{equation}\label{eqn:MSEs}
\begin{array}{cc}
    \mathcal{F}_{\mathcal{L}_i}(\Theta):=\frac{1}{M_{\mathcal{L}_i}}\sum_{k=1}^{M_{\mathcal{L}_i}}|\mathcal{L}_i(\bX_{k};\Theta)|^2, & \mathcal{F}_{\Gamma}(\Theta):=\frac{1}{M_{\Gamma}}\sum_{k=1}^{M_{\Gamma}}|\mathcal{L}_{\Gamma}(\bX_{k};\Theta)|^2,  \\
    \mathcal{F}_{\mathcal{B}_i}(\Theta):=\frac{1}{M_{\mathcal{B}_i}}\sum_{k=1}^{M_{\mathcal{B}_i}}|\mathcal{B}_i(\bX_{k};\Theta)|^2,
    &
    \mathcal{F}_{\mathcal{I}_i}(\Theta):=\frac{1}{M_{\mathcal{I}_i}}\sum_{k=1}^{M_{\mathcal{I}_i}}|\mathcal{I}_i(\bX_{k};\Theta)|^2,
\end{array}
\end{equation}
where $M_{\mathcal{L}_i}$, $M_{\Gamma}$, $M_{\mathcal{B}_i}$, and $M_{\mathcal{I}_i}$ are the number of sampling points used for different parts of the total loss functional $\mathcal{R}(\bX;\Theta)$, i.e., for the interior part, the interface part, the boundary part and the initial domain part of the space-time domain, $\overline{\Omega}\times[0,T]$, respectively. If one of them equals to $0$, then the corresponding loss functional or the MSE is dropped. {$M= \sum\limits_{i=1}^2 (M_{\mathcal{L}_i} + M_{\mathcal{B}_i} + M_{\mathcal{I}_i})+M_{\Gamma}$} is the total size of the sampling point set. By trying not to abuse notations, here we use $\bX_k\ (k=1,2,\cdots)$ to represent sampling points and drop the dependence of $\bX$ on the discrete level. Thus, the total discrete loss functional is defined as follows,
$$
\mathcal{F}(\Theta):= \sum_{i=1}^2 \left(\omega_{\mathcal{L}_i}
\mathcal{F}_{\mathcal{L}_i}(\Theta)  + \omega_{\mathcal{B}_i}
\mathcal{F}_{\mathcal{B}_i}(\Theta) + \omega_{\mathcal{I}_i}
\mathcal{F}_{\mathcal{I}_i}(\Theta)\right)+ \omega_{\Gamma}
\mathcal{F}_{\Gamma}(\Theta),
$$
and our DNN/meshfree method for the interface problem~\eqref{eqn:interface-model} solves the following minimization problem,
\begin{equation} \label{eqn:DNN-interface-method}
    \min\limits_{\Theta\in\Psi^N}\mathcal{F}(
    \Theta)=\min\limits_{\Theta\in\Psi^N} \left[ \sum_{i=1}^2
    \left(\omega_{\mathcal{L}_i} \mathcal{F}_{\mathcal{L}_i}(\Theta)
    + \omega_{\mathcal{B}_i} \mathcal{F}_{\mathcal{B}_i}(\Theta) +
    \omega_{\mathcal{I}_i}
    \mathcal{F}_{\mathcal{I}_i}(\Theta)\right)+ \omega_{\Gamma}
    \mathcal{F}_{\Gamma}(\Theta) \right].
\end{equation}

To solve~\eqref{eqn:DNN-interface-method}, we use standard optimization algorithms such as the SGD method \cite{Robbins1951}. Actually, we use a variant of SGD, so-called the adaptive moment estimation (ADAM)~\cite{kingma2014adam}, in practice. The gradients used in the gradient descent method are computed via backward propagation as well as automatic differentiation. When the approximated minimizer~$\Theta^\ast$ is reached, we attain the desired DNN result, $\bigg(\mathcal{U}_{\mathbf{\bN\bN}}^{1}(\bX^{(1)};\Theta^\ast)$, $\mathcal{U}_{\mathbf{\bN\bN}}^{2}(\bX^{(2)};\Theta^\ast)\bigg)$, which is the numerical solution of the abstract interface problem~\eqref{eqn:interface-model}.

\subsection{Applications to two two-phase interface problems}\label{sec:DNN4interfaceproblems}
In the previous subsection, we discuss the DNN/meshfree method for solving the general interface problem~\eqref{eqn:interface-model}. In this subsection, we give concrete examples of applying the proposed DNN/meshfree method~\eqref{eqn:DNN-interface-method} to the two-phase flow interface problem and the FSI problems introduced in Section~\ref{sec:model:subsec:example}, respectively.

\subsubsection{Example 1: The two-phase flow interface problem}
To solve the two-phase flow interface problem using the proposed DNN/meshfree method, we propose the following total LS formulation with respect to $\tilde{\bm{u}}_i := (\tilde{\bm{v}}_i, \tilde{p}_i) \in H^1([0,T]; H^2(\Omega_i)) \times L^2([0,T]; H^1(\Omega_i))$, $i=1,2$,
\begin{align}
\mathcal{R}(\tilde{\bm{u}}_1, \tilde{\bm{u}}_2) & := \sum_{i=1}^2 \omega_{\mathcal{I}_i} \int_{\Omega_i} {\bm{\rho}_i} \left|\bm{v}_i^0 - \tilde{\bm{v}}_i \right|^2 \mathrm{d} \bx \notag\\
& \quad + \int_{0}^T \left[  \sum_{i=1}^2 \omega_{\mathcal{L}_i} \int_{\Omega_i}  \left|\bm{\rho}_i \left( \frac{\partial \tilde{\bm{v}}_i}{\partial t}  + \tilde{\bm{v}}_i \cdot \nabla \tilde{\bv}_i \right) - \nabla \cdot \tilde{\bS}_i - \bm{f}_i   \right|^2 \mathrm{d} \bx    \right. \notag\\
& \qquad \qquad + \sum_{i=1}^2 \omega_{\mathcal{L}_i} \int_{\Omega_i} \left| \nabla \cdot \tilde{\bv}_i \right|^2 \mathrm{d} \bx  + \sum_{i=1}^2 \omega_{\mathcal{B}_i} \int_{\partial \Omega_i \backslash \Gamma} \left| \bm{v}_i^b - \tilde{\bv}_i \right|^2 \mathrm{d} \bx
\label{LS4twophase}\\
& \qquad \qquad \left. + \omega_{\mathcal{L}_{\Gamma}} \int_{\Gamma}
\left| \tilde{\bv}_1 - \tilde{\bv}_2 - \bg_1  \right|^2 \mathrm{d} s
+ \omega_{\mathcal{L}_{\Gamma}}\int_{\Gamma} \left| \tilde{\bS}_1
\bn_1 + \tilde{\bS}_2 \bn_2 - \bg_2 \right|^2 \mathrm{d} s \right]
\mathrm{d}t.\notag
\end{align}
So the LS formulation~\eqref{eqn:LS} becomes: find $\bm{u}_i := (\bm{v}_i, p_i) \in H^1([0,T]; H^2(\Omega_i)) \times L^2([0,T]; H^1(\Omega_i))$ , $i=1,2$, such that
\begin{equation*}
\mathcal{R}(\bu_1, \bu_2) = \min_{\tilde{\bu}_i \in H^1([0,T]; H^2(\Omega_i)) \times L^2([0,T]; H^1(\Omega_i)), \ i=1,2}
\mathcal{R}(\tilde{\bm{u}}_1, \tilde{\bm{u}}_2).
\end{equation*}
Approximating $\tilde{\bv}_i$ and $\tilde{p}_i$ by their DNN functions $\mathcal{V}^i_{\mathbf{NN}}(\bX^{(i)}; \Theta)$ and~$\mathcal{P}_{\mathbf{NN}}^i(\bX^{(i)}; \Theta)$, $i=1,2$, as suggested in~\eqref{eqn:uNNs}, respectively, i.e., letting $\mathcal{U}^i_{\mathbf{NN}}(\bX^{(i)}; \Theta):=\left(\mathcal{V}^i_{\mathbf{NN}}(\bX^{(i)}; \Theta),\mathcal{P}_{\mathbf{NN}}^i(\bX^{(i)}; \Theta)\right)$, $i=1,2$, $n_1=d+1$ and $n_{L+1}=2(d+1)$ for the DNN structure~\eqref{eqn:NN-structure}, and replacing integrals in the LS formulation~\eqref{LS4twophase} by corresponding MSEs as suggested in~\eqref{eqn:MSEs}, we obtain the following DNN/meshfree method for the two-phase flow interface problem, where we drop $\bX$ in the formulation for the sake of simplicity,
\begin{equation*}
\min_{\Theta} \mathcal{F}(\Theta) = \min_{\Theta} \left[
\sum_{i=1}^2 \left(\omega_{\mathcal{I}_i}
\mathcal{F}_{\mathcal{I}_i}(\Theta) + \omega_{\mathcal{L}_i}
\mathcal{F}_{\mathcal{L}_i}(\Theta) + \omega_{\mathcal{B}_i}
\mathcal{F}_{\mathcal{B_i}}(\Theta)\right) + \omega_{\Gamma}
\mathcal{F}_{\Gamma}(\Theta) \right],
\end{equation*}
where
\begin{align*}
    \mathcal{F}_{\mathcal{I}_i}(\Theta) & \ :=\ \frac{1}{M_{\mathcal{I}_i}} \sum_{k=1}^{M_{\mathcal{I}_i}} \bm{\rho}_i \left| \bm{v}_i^0(\bX_k) - \mathcal{V}^i_{\mathbf{NN}}(\bX_k; \Theta) \right|^2, \\
    \mathcal{F}_{\mathcal{L}_i}(\Theta) & \ := \ \frac{1}{M_{\mathcal{L}_i}} \sum_{k=1}^{M_{\mathcal{L}_i}}  \left|\bm{\rho}_i \left( \frac{\partial \mathcal{V}^i_{\mathbf{NN}}(\bX_k; \Theta)}{\partial t}  + \mathcal{V}^i_{\mathbf{NN}}(\bX_k; \Theta) \cdot \nabla \mathcal{V}^i_{\mathbf{NN}}(\bX_k; \Theta) \right) \right. \\
    & \qquad \left. - \nabla \cdot \bS(\mathcal{V}^i_{\mathbf{NN}}(\bX_k; \Theta), \mathcal{P}_{\mathbf{NN}}^i(\bX_k; \Theta)) - \bm{f}_i(\bX_k)   \right|^2 +  \frac{1}{M_{\mathcal{L}_i}} \sum_{k=1}^{M_{\mathcal{L}_i}} \left| \nabla \cdot \mathcal{V}^i_{\mathbf{NN}}(\bX_k; \Theta) \right|^2, \\
    \mathcal{F}_{\mathcal{B}_i}(\Theta) & \ := \ \frac{1}{M_{\mathcal{B}_i}} \sum_{k=1}^{M_{\mathcal{B}_i}} \left| \bm{v}_i^b(\bX_k) - \mathcal{V}^i_{\mathbf{NN}}(\bX_k; \Theta) \right|^2, \\
    \mathcal{F}_{\Gamma}(\Theta) & \ := \ \frac{1}{M_{\Gamma}} \sum_{k=1}^{M_{\Gamma}} \left| \mathcal{V}^1_{\mathbf{NN}}(\bX_k; \Theta) - \mathcal{V}^2_{\mathbf{NN}}(\bX_k; \Theta) - \bg_1(\bX_k)  \right|^2, \\
    & \ \ \ +  \frac{1}{M_{\Gamma}} \sum_{k=1}^{M_{\Gamma}} \left| \bS(\mathcal{V}^1_{\mathbf{NN}}(\bX_k; \Theta), \mathcal{P}^1_{\mathbf{NN}}(\bX_k; \Theta)) \bn_1 + \bS(\mathcal{V}^2_{\mathbf{NN}}(\bX_k; \Theta), \mathcal{P}^2_{\mathbf{NN}}(\bX_k; \Theta)) \bn_2 - \bg_2(\bX_k) \right|^2.
\end{align*}

\subsubsection{Example 2: The FSI problems}
For the fluid-structure interaction problem with a wave-type structural equation, we propose the following total LS formulation, for $\tilde{\bm{u}}_1 := (\tilde{\bm{v}}_f, \tilde{p}_f) \in H^1([0,T]; H^2(\Omega_f)) \times L^2([0,T]; H^1(\Omega_f))$ and $\tilde{\bm{u}}_2 := \tilde{\bm{u}}_s \in H^2([0,T]; H^2(\Omega_s))$, define
\begin{align}
    \mathcal{R}(\tilde{\bm{u}}_1, \tilde{\bm{u}}_2) & := \omega_{\mathcal{I}_f} \int_{\Omega_f} \bm{\rho}_f \left|\bm{v}_f^0 - \tilde{\bm{v}}_f \right|^2 \mathrm{d} \bx + \omega_{\mathcal{I}_s} \int_{\Omega_s} \bm{\rho}_s \left|\frac{\partial \bm{u}_s^0}{\partial t} - \frac{\partial \tilde{\bm{u}}_s}{\partial t} \right|^2 \mathrm{d} \bx \notag\\
    & \quad + \omega_{\mathcal{I}_s} \int_{\Omega_s} \left( 2\mu_s \left| \varepsilon(\bu_s^0 -\tilde{\bu}_s)   \right|^2 + \lambda_s \left| \nabla \cdot (\bu^0_s - \tilde{\bu}_s) \right|^2 \right) \mathrm{d} \bx + \omega_{\mathcal{I}_s} \int_{\Omega_s} \left| \bu_s^0 - \tilde{\bu}_s \right|^2 \mathrm{d} \bx  \notag\\
    & \quad + \int_{0}^T \left[ \omega_{\mathcal{L}_f} \int_{\Omega_f}  \left|\bm{\rho}_f \left( \frac{\partial \tilde{\bm{v}}_f}{\partial t}  + \tilde{\bm{v}}_f \cdot \nabla \tilde{\bv}_f \right) - \nabla \cdot \tilde{\bS}_f - \bm{f}_f   \right|^2 \mathrm{d} \bx  +  \omega_{\mathcal{L}_f} \int_{\Omega_f} \left| \nabla \cdot \tilde{\bv}_f \right|^2 \mathrm{d} \bx   \right. \notag\\
    & \qquad \qquad + \omega_{\mathcal{L}_s}\int_{\Omega_s} \left| \bm{\rho}_s \left( \frac{\partial^2 \tilde{\bu}_s}{\partial t^2} \right) - \nabla \cdot \bS(\tilde{\bu}_s) - \bm{f}_s \right|^2 \mathrm{d}\bx \label{LS4FSI}\\
    & \qquad \qquad  + \omega_{\mathcal{B}_f} \int_{\partial \Omega_f \backslash \Gamma} \left| \bm{v}_f^b - \tilde{\bv}_f \right|^2 \mathrm{d} s + \omega_{\mathcal{B}_s} \int_{\partial \Omega_s \backslash \Gamma} \left| \frac{\partial \bu_s^b}{\partial t} - \frac{\partial \tilde{\bu}_s}{\partial t} \right|^2 \mathrm{d}s + \omega_{\mathcal{B}_s} \int_{\partial{\Omega_s} \backslash \Gamma} \left| \bu_s^b - \tilde{\bu}_s \right|^2 \mathrm{d}s  \notag\\
    & \qquad \qquad \left. + \omega_{\Gamma} \int_{\Gamma} \left| \tilde{\bv}_f - \frac{\partial\tilde{\bu}_s}{\partial t} - \bg_1  \right|^2 \mathrm{d} s + \omega_{\Gamma}\int_{\Gamma} \left| \tilde{\bS}_f \bn_f + \tilde{\bS}_s \bn_s - \bg_2 \right|^2 \mathrm{d} s \right]
    \mathrm{d}t.\notag
\end{align}
We want to point out that the third and the ninth integral term on the right-hand side of the above equation do not directly come from the wave-type structural equation introduced in Section~\ref{sec:FSI-example}. For the theoretical purpose, we add those two terms based on the error analysis in Section~\ref{sec:fsi-erroranalysis}. Without those two terms, our DNN/meshfree can still be applied in practice. The overall LS optimization problem is to find $\bu_1 \in H^1([0,T]; H^2(\Omega_f)) \times L^2([0,T]; H^1(\Omega_f))$ and $\bu_2 \in H^2([0,T]; H^2(\Omega_s))$ such that
\begin{equation*}
    \mathcal{R}(\bu_1, \bu_2) = \min_{\substack{\tilde{\bu}_1 \in H^1([0,T]; H^2(\Omega_f)) \times L^2([0,T]; H^1(\Omega_f))  \\ \tilde{\bu}_2 \in H^2([0,T]; H^2(\Omega_s))}} \mathcal{R}(\tilde{\bm{u}}_1, \tilde{\bm{u}}_2).
\end{equation*}
By approximating solutions of FSI, $\bv_f$, $p_f$ and $\bu_s$, using their DNN functions, $\mathcal{V}^f_{\mathbf{NN}}(\bX^{(f)}; \Theta)$, $\mathcal{P}^f_{\mathbf{NN}}(\bX^{(f)}; \Theta)$ and $\mathcal{U}^s_{\mathbf{NN}}(\bX^{(s)}; \Theta)$ as suggested in~\eqref{eqn:uNNs}, respectively, i.e., letting $\mathcal{U}^1_{\mathbf{NN}}(\bX^{(f)}; \Theta):=\left(\mathcal{V}^f_{\mathbf{NN}}(\bX^{(f)}; \Theta), \mathcal{P}^f_{\mathbf{NN}}(\bX^{(f)}; \Theta)\right)$ and $\mathcal{U}^2_{\mathbf{NN}}(\bX^{(s)}; \Theta) := \mathcal{U}^s_{\mathbf{NN}}(\bX^{(s)}; \Theta)$, $n_1=d+1$ and $n_{L+1}=2d+1$ for the DNN structure~\eqref{eqn:NN-structure}, and replacing all integrals in the LS formulation (\ref{LS4FSI}) by MSEs as suggested in~\eqref{eqn:MSEs}, we obtain the following DNN/meshfree method for the FSI problem,
\begin{equation*}
    \min_{\Theta} \mathcal{F}(\Theta) = \min_{\Theta} \left[   \omega_{\mathcal{I}_f} \mathcal{F}_{\mathcal{I}_f}(\Theta) +
    \omega_{\mathcal{I}_s} \mathcal{F}_{\mathcal{I}_s}(\Theta) + \omega_{\mathcal{L}_f} \mathcal{F}_{\mathcal{L}_f}(\Theta) +
    \omega_{\mathcal{L}_s} \mathcal{F}_{\mathcal{L}_s}(\Theta) + \omega_{\mathcal{B}_f} \mathcal{F}_{\mathcal{B_f}}(\Theta) +
    \omega_{\mathcal{B}_s} \mathcal{F}_{\mathcal{B_s}}(\Theta) + \omega_{\Gamma} \mathcal{F}_{\Gamma}(\Theta) \right],
\end{equation*}
where
\begin{align*}
    \mathcal{F}_{\mathcal{I}_f}(\Theta) & \ :=\ \frac{1}{M_{\mathcal{I}_f}} \sum_{k=1}^{M_{\mathcal{I}_f}} \bm{\rho}_f \left| \bm{v}_f^0(\bX_k) - \mathcal{V}^f_{\mathbf{NN}}(\bX_k; \Theta) \right|^2, \\
    \mathcal{F}_{\mathcal{I}_s}(\Theta) & \ :=\ \frac{1}{M_{\mathcal{I}_s}} \sum_{k=1}^{M_{\mathcal{I}_s}} \bm{\rho}_s \left| \frac{\partial \bm{u}_s^0(\bX_k)}{\partial t} - \frac{\partial \mathcal{U}^s_{\mathbf{NN}}(\bX_k; \Theta)}{\partial t} \right|^2 \\
    & \ \ \ +  \frac{1}{M_{\mathcal{I}_s}} \sum_{k=1}^{M_{\mathcal{I}_s}}  2\mu_s \left| \varepsilon(\bu_s^0(\bX_k) -\mathcal{U}^s_{\mathbf{NN}}(\bX_k; \Theta))   \right|^2 + \lambda_s \left| \nabla \cdot (\bu^0_s(\bX_k) - \mathcal{U}^s_{\mathbf{NN}}(\bX_k; \Theta)) \right|^2 \\
    & \ \ \  + \frac{1}{M_{\mathcal{I}_s}} \sum_{k=1}^{M_{\mathcal{I}_s}} \left| \bu_{s}^0(\bX_k) - \mathcal{U}_{\mathbf{NN}}^s(\bX_k; \Theta) \right|^2,  \\
    \mathcal{F}_{\mathcal{L}_f}(\Theta) & \ := \ \frac{1}{M_{\mathcal{L}_f}} \sum_{k=1}^{M_{\mathcal{L}_f}}  \left|\bm{\rho}_f \left( \frac{\partial \mathcal{V}^f_{\mathbf{NN}}(\bX_k; \Theta)}{\partial t}  + \mathcal{V}^f_{\mathbf{NN}}(\bX_k; \Theta) \cdot \nabla \mathcal{V}^f_{\mathbf{NN}}(\bX_k; \Theta) \right) \right. \\
    &\qquad \left. - \nabla \cdot \bS(\mathcal{V}^f_{\mathbf{NN}}(\bX_k; \Theta), \mathcal{P}_{\mathbf{NN}}^f(\bX_k; \Theta)) - \bm{f}_f(\bX_k)   \right|^2 +  \frac{1}{M_{\mathcal{L}_f}} \sum_{k=1}^{M_{\mathcal{L}_f}} \left| \nabla \cdot \mathcal{V}^f_{\mathbf{NN}}(\bX_k; \Theta) \right|^2, \\
    \mathcal{F}_{\mathcal{L}_s}(\Theta) & \ := \ \frac{1}{M_{\mathcal{L}_s}} \sum_{k=1}^{M_{\mathcal{L}_s}} \left| \bm{\rho}_s \left( \frac{\partial^2 \mathcal{U}^s_{\mathbf{NN}}(\bX_k; \Theta)}{\partial t^2} \right) - \nabla \cdot \bS(\mathcal{U}^s_{\mathbf{NN}}(\bX_k; \Theta)) - \bm{f}_s(\bX_k) \right|^2,  \\
    \mathcal{F}_{\mathcal{B}_f}(\Theta) & \ := \ \frac{1}{M_{\mathcal{B}_f}} \sum_{k=1}^{M_{\mathcal{B}_f}} \left| \bm{v}_f^b(\bX_k) - \mathcal{V}^f_{\mathbf{NN}}(\bX_k; \Theta) \right|^2, \\
    \mathcal{F}_{\mathcal{B}_s}(\Theta) & \ := \ \frac{1}{M_{\mathcal{B}_s}} \sum_{k=1}^{M_{\mathcal{B}_s}} \left| \frac{ \partial  \bm{u}_s^b(\bX_k)}{\partial t} - \frac{ \partial \mathcal{U}^s_{\mathbf{NN}}(\bX_k; \Theta)}{\partial t} \right|^2 + \frac{1}{M_{\mathcal{B}_s}} \sum_{k=1}^{M_{\mathcal{B}_s}} \left| \bu_s^b - \mathcal{U}_{\mathbf{NN}}^s(\bX_k; \Theta) \right|^2, \\
    \mathcal{F}_{\Gamma}(\Theta) & \ := \ \frac{1}{M_{\Gamma}} \sum_{k=1}^{M_{\Gamma}} \left| \mathcal{V}^f_{\mathbf{NN}}(\bX_k; \Theta) - \frac{\partial \mathcal{U}^s_{\mathbf{NN}}(\bX_k; \Theta)}{\partial t} - \bg_1(\bX_k)  \right|^2 \\
    & \ \ \ +  \frac{1}{M_{\Gamma}} \sum_{k=1}^{M_{\Gamma}} \left| \bS(\mathcal{V}^f_{\mathbf{NN}}(\bX_k; \Theta), \mathcal{P}^f_{\mathbf{NN}}(\bX_k; \Theta)) \bn_f + \bS(\mathcal{U}^s_{\mathbf{NN}}(\bX_k; \Theta)) \bn_s - \bg_2(\bX_k) \right|^2.
 \end{align*}

On the other hand, for the FSI model with a parabolic-like structural equation, we define the following LS formulation, for $\tilde{\bm{u}}_1 := (\tilde{\bm{v}}_f, \tilde{p}_f) \in H^1([0,T];H^2(\Omega_f)) \times L^2([0,T]; H^1(\Omega_f))$ and $\tilde{\bm{u}}_2 := (\tilde{\bm{u}}_s, \tilde{\bm{v}}_s) \in H^1([0,T]; H^2(\Omega_s)) \times H^1([0,T]; H^1(\Omega_s)) $,
\begin{align}
    \mathcal{R}(\tilde{\bm{u}}_1, \tilde{\bm{u}}_2) & := \omega_{\mathcal{I}_f} \int_{\Omega_f} \bm{\rho}_f \left|\bm{v}_f^0 - \tilde{\bm{v}}_f \right|^2 \mathrm{d} \bx + \omega_{\mathcal{I}_s} \int_{\Omega_s} \bm{\rho}_s \left| \bm{v}_s^0 - \tilde{\bm{v}}_s \right|^2 \mathrm{d} \bx \notag\\
    & \quad + \omega_{\mathcal{I}_s}\int_{\Omega_s} \left( 2\mu_s \left| \varepsilon(\bu_s^0 -\tilde{\bu}_s)   \right|^2 + \lambda_s \left| \nabla \cdot (\bu^0_s - \tilde{\bu}_s) \right|^2 \right)  \mathrm{d} \bx + \omega_{\mathcal{I}_s} \int_{\Omega_s} |\bu_s^0 - \tilde{\bu_s}|^2 \mathrm{d} \bx  \notag\\
    & \quad + \int_{0}^T \left[ \omega_{\mathcal{L}_f} \int_{\Omega_f}  \left|\bm{\rho}_f \left( \frac{\partial \tilde{\bm{v}}_f}{\partial t}  + \tilde{\bm{v}}_f \cdot \nabla \tilde{\bv}_f \right) - \nabla \cdot \tilde{\bS}_f - \bm{f}_f   \right|^2 \mathrm{d} \bx  +  \omega_{\mathcal{L}_f} \int_{\Omega_f} \left| \nabla \cdot \tilde{\bv}_f \right|^2 \mathrm{d} \bx   \right. \nonumber \\
    & \qquad \qquad + \omega_{\mathcal{L}_s}\int_{\Omega_s} \left| \bm{\rho}_s \frac{\partial \tilde{\bv}_s}{\partial t} - \nabla \cdot \bS(\tilde{\bu}_s) - \bm{f}_s \right|^2 \mathrm{d}\bx \label{LS4FSI1} \\
    & \qquad \qquad + \omega_{\mathcal{L}_s} \int_{\Omega_s} \left( 2\mu_s \left| \varepsilon\left(\frac{\partial \tilde{\bu}_s}{\partial t} -\tilde{\bv}_s\right)   \right|^2 + \lambda_s \left| \nabla \cdot \left(\frac{\partial \tilde{\bu}_s}{\partial t} - \tilde{\bv}_s\right) \right|^2 \right) \mathrm{d} \bx   \notag\\
    & \qquad \qquad + \omega_{\mathcal{L}_s} \int_{\Omega_s} \left|\frac{\partial \tilde{\bu}_s}{\partial t} - \tilde{\bv}_s \right|^2 \mathrm{d} \bx + \omega_{\mathcal{B}_f} \int_{\partial \Omega_f \backslash \Gamma} \left| \bm{v}_f^b - \tilde{\bv}_f \right|^2 \mathrm{d}s + \omega_{\mathcal{B}_s} \int_{\partial \Omega_s \backslash \Gamma} \left|  \bv_s^b - \tilde{\bv}_s \right|^2 \mathrm{d}s  \notag\\
    & \qquad \qquad \left. + \omega_{\Gamma} \int_{\Gamma} \left| \tilde{\bv}_f - \frac{\partial\tilde{\bu}_s}{\partial t} - \bg_1  \right|^2 \mathrm{d} s + \omega_{\Gamma}\int_{\Gamma} \left| \tilde{\bS}_f \bn_f + \tilde{\bS}_s \bn_s - \bg_2 \right|^2 \mathrm{d} s \right]
    \mathrm{d}t.\notag
\end{align}
Again, the third and the eighth integral term on the right hand side of~\eqref{LS4FSI1} are added based on the error analysis given in Section~\ref{sec:fsi-erroranalysis} for the theoretical purpose, they do not directly come from the parabolic-like structural equation introduced in Section~\ref{sec:FSI-example}. Without those two terms, our DNN/meshfree can still be used in practice. The overall LS optimization problem in this case becomes to find $\bu_1\in H^1([0,T];H^2(\Omega_f)) \times L^2([0,T];H^1(\Omega_f))$ and $\bu_2 \in H^1([0,T];H^2(\Omega_s)) \times H^1([0,T];H^1(\Omega_s))$ such that
\begin{equation*}
    \mathcal{R}(\bu_1, \bu_2) = \min_{ \substack{ \tilde{\bu}_1 \in  H^1([0,T];H^2(\Omega_f)) \times L^2([0,T];H^1(\Omega_f)) \\ \tilde{\bu}_2 \in H^1([0,T]; H^2(\Omega_s)) \times H^1([0,T];H^1(\Omega_s)) } } \mathcal{R}(\tilde{\bm{u}}_1, \tilde{\bm{u}}_2).
\end{equation*}
By approximating solutions of FSI, $\bv_f$, $p_f$, $\bv_s$ and~$\bu_s$, using their DNN functions, $\mathcal{V}^f_{\mathbf{NN}}(\bX^{(f)}; \Theta)$, $\mathcal{P}^f_{\mathbf{NN}}(\bX^{(f)}; \Theta)$, $\mathcal{V}^s_{\mathbf{NN}}(\bX^{(s)}; \Theta)$, and~$\mathcal{U}^s_{\mathbf{NN}}(\bX^{(s)}; \Theta)$ as suggested in~\eqref{eqn:uNNs}, respectively, i.e., letting $\mathcal{U}^1_{\mathbf{NN}}(\bX^{(f)}; \Theta):=$ $\left(\mathcal{V}^f_{\mathbf{NN}}(\bX^{(f)}; \Theta)\right.$, $\left.\mathcal{P}^f_{\mathbf{NN}}(\bX^{(f)}; \Theta)\right)$ and~$\mathcal{U}^2_{\mathbf{NN}}(\bX^{(s)}; \Theta):=\left(\mathcal{U}^s_{\mathbf{NN}}(\bX^{(s)}; \Theta), \mathcal{V}^s_{\mathbf{NN}}(\bX^{(s)}; \Theta)\right)$, $n_1=d+1$ and $n_{L+1}=3d+1$ for the DNN structure~\eqref{eqn:NN-structure}, and replacing all integrals in the LS formulation (\ref{LS4FSI1}) by MSEs as suggested in~\eqref{eqn:MSEs}, we obtain the following DNN/meshfree method,
\begin{equation*}
    \min_{\Theta} \mathcal{F}(\Theta) = \min_{\Theta} \left[   \omega_{\mathcal{I}_f} \mathcal{F}_{\mathcal{I}_f}(\Theta) +
    \omega_{\mathcal{I}_s} \mathcal{F}_{\mathcal{I}_s}(\Theta) + \omega_{\mathcal{L}_f} \mathcal{F}_{\mathcal{L}_f}(\Theta) +
    \omega_{\mathcal{L}_s} \mathcal{F}_{\mathcal{L}_s}(\Theta) + \omega_{\mathcal{B}_f} \mathcal{F}_{\mathcal{B_f}}(\Theta) +
    \omega_{\mathcal{B}_s} \mathcal{F}_{\mathcal{B_s}}(\Theta) + \omega_{\Gamma} \mathcal{F}_{\Gamma}(\Theta) \right],
\end{equation*}
where
\begin{align*}
    \mathcal{F}_{\mathcal{I}_f}(\Theta) & \ :=\ \frac{1}{M_{\mathcal{I}_f}} \sum_{k=1}^{M_{\mathcal{I}_f}} \bm{\rho}_f \left| \bm{v}_f^0(\bX_k) - \mathcal{V}^f_{\mathbf{NN}}(\bX_k; \Theta) \right|^2, \\
    \mathcal{F}_{\mathcal{I}_s}(\Theta) & \ :=\ \frac{1}{M_{\mathcal{I}_s}} \sum_{k=1}^{M_{\mathcal{I}_s}} \bm{\rho}_s \left|  \bm{v}_s^0(\bX_k) -  \mathcal{V}^s_{\mathbf{NN}}(\bX_k; \Theta) \right|^2 \\
    & \ \ \ +  \frac{1}{M_{\mathcal{I}_s}} \sum_{k=1}^{M_{\mathcal{I}_s}}  2\mu_s \left| \varepsilon(\bu_s^0(\bX_k) -\mathcal{U}^s_{\mathbf{NN}}(\bX_k; \Theta))   \right|^2 + \lambda_s \left| \nabla \cdot (\bu^0_s(\bX_k) - \mathcal{U}^s_{\mathbf{NN}}(\bX_k; \Theta)) \right|^2 \\
    & \ \ \  + \frac{1}{M_{\mathcal{I}_s}} \sum_{k=1}^{M_{\mathcal{I}_s}} \left| \bu_s^0(\bX_k) - \mathcal{U}_{\mathbf{NN}}^s(\bX_k; \Theta) \right|^2,  \\
    \mathcal{F}_{\mathcal{L}_f}(\Theta) & \ := \ \frac{1}{M_{\mathcal{L}_f}} \sum_{k=1}^{M_{\mathcal{L}_f}}  \left|\bm{\rho}_f \left( \frac{\partial \mathcal{V}^f_{\mathbf{NN}}(\bX_k; \Theta)}{\partial t}  + \mathcal{V}^f_{\mathbf{NN}}(\bX_k; \Theta) \cdot \nabla \mathcal{V}^f_{\mathbf{NN}}(\bX_k; \Theta) \right) \right. \\
    & \qquad  \left. - \nabla \cdot \bS(\mathcal{V}^f_{\mathbf{NN}}(\bX_k; \Theta), \mathcal{P}_{\mathbf{NN}}^f(\bX_k; \Theta)) - \bm{f}_f(\bX_k)   \right|^2 +  \frac{1}{M_{\mathcal{L}_f}} \sum_{k=1}^{M_{\mathcal{L}_f}} \left| \nabla \cdot \mathcal{V}^f_{\mathbf{NN}}(\bX_k; \Theta) \right|^2, \\
    \mathcal{F}_{\mathcal{L}_s}(\Theta) & \ := \ \frac{1}{M_{\mathcal{L}_s}} \sum_{k=1}^{M_{\mathcal{L}_s}} \left| \bm{\rho}_s \left( \frac{\partial \mathcal{V}^s_{\mathbf{NN}}(\bX_k; \Theta)}{\partial t} \right) - \nabla \cdot \bS(\mathcal{U}^s_{\mathbf{NN}}(\bX_k; \Theta)) - \bm{f}_s(\bX_k) \right|^2  \\
    & \ \ \ + \frac{1}{M_{\mathcal{L}_s}} \sum_{k=1}^{M_{\mathcal{L}_s}}  2\mu_s \left| \varepsilon\left(\frac{\partial \mathcal{U}^s_{\mathbf{NN}}(\bX_k; \Theta)}{\partial t} -\mathcal{V}^s_{\mathbf{NN}}(\bX_k; \Theta)\right)   \right|^2 + \lambda_s \left| \nabla \cdot \left(\frac{\partial \mathcal{U}^s_{\mathbf{NN}}(\bX_k; \Theta)}{\partial t} -\mathcal{V}^s_{\mathbf{NN}}(\bX_k; \Theta)\right) \right|^2 \\
  & \ \ \ + \frac{1}{\mathcal{M}_{\mathcal{L}_s}} \sum_{k=1}^{M_{\mathcal{L}_s}}  \left|  \frac{\partial \mathcal{U}^s_{\mathbf{NN}}(\bX_k; \Theta)}{\partial t} -\mathcal{V}^s_{\mathbf{NN}}(\bX_k; \Theta) \right|^2, \\
    \mathcal{F}_{\mathcal{B}_f}(\Theta) & \ := \ \frac{1}{M_{\mathcal{B}_f}} \sum_{k=1}^{M_{\mathcal{B}_f}} \left| \bm{v}_f^b(\bX_k) - \mathcal{V}^f_{\mathbf{NN}}(\bX_k; \Theta) \right|^2, \\
    \mathcal{F}_{\mathcal{B}_s}(\Theta) & \ := \ \frac{1}{M_{\mathcal{B}_s}} \sum_{k=1}^{M_{\mathcal{B}_s}} \left|  \bm{v}_s^b(\bX_k) -   \mathcal{V}^s_{\mathbf{NN}}(\bX_k; \Theta)
    \right|^2,    \\
    \mathcal{F}_{\Gamma}(\Theta) & \ := \ \frac{1}{M_{\Gamma}} \sum_{k=1}^{M_{\Gamma}} \left| \mathcal{V}^f_{\mathbf{NN}}(\bX_k; \Theta) - \frac{\partial \mathcal{U}^s_{\mathbf{NN}}(\bX_k; \Theta)}{\partial t} - \bg_1(\bX_k)  \right|^2 \\
    & \ \ \ +  \frac{1}{M_{\Gamma}} \sum_{k=1}^{M_{\Gamma}} \left| \bS(\mathcal{V}^f_{\mathbf{NN}}(\bX_k; \Theta), \mathcal{P}^f_{\mathbf{NN}}(\bX_k; \Theta)) \bn_f + \bS(\mathcal{U}^s_{\mathbf{NN}}(\bX_k; \Theta)) \bn_s - \bg_2(\bX_k) \right|^2.
\end{align*}

\section{Error analysis}\label{sec:error}
In this section, based on the recent work~\cite{mishra2022estimates}, we analyze  convergence properties of our proposed DNN/meshfree method. We first present an abstract framework of the error analysis for the abstract interface problem~\eqref{eqn:interface-model} solved by our DNN/meshfree method. Then we derive the specific error estimation of the developed DNN/meshfree method for each interface model problem presented in this paper, i.e., the two-phase flow interface problem and the FSI problem with two different types of structural equations. Here we use $\mathbb{X}(\mathfrak{D})$ to denote a general normed space defined in the space-time domain $\mathfrak{D}$ equipped with norms $\| \cdot \|_{\mathbb{X}( \mathfrak{D})}$.

\subsection{Abstract framework} \label{sec:abstract-error-analysis}
Following~\cite{mishra2022estimates}, our abstract framework of the error analysis for the proposed DNN/meshfree method is based on the following two major assumptions on the abstract interface model problem~\eqref{eqn:interface-model}.

\begin{assumption}\label{assumption:pde}
For any $\bm{u}_i, \bm{v}_i \in \mathbb{X}(\overline{\Omega_i} \times [0,T])$, $i=1,2$, the differential operators in the abstract interface model problem~\eqref{eqn:interface-model} satisfies, for $0< \beta_{\mathcal{L}_i}, \ \beta_{\mathcal{B}_i}, \ \beta_{\mathcal{I}_i}, \ \beta_{\Gamma} \leq 1$,
    \begin{align} %
        & \quad \sum_{i=1}^2 \| \bm{u}_i - \bm{v}_i \|_{\mathbb{X}(\overline{\Omega_i} \times [0,T])}^2 \nonumber \\
        &\leq C_{pde} \left\{  \sum_{i=1}^2 \left[
        \left(\| \mathcal{L}_i(\bm{u_i}) - \mathcal{L}_i(\bm{v_i}) \|^2_{L^2({\Omega_i} \times [0,T])} \right)^{\beta_{\mathcal{L}_i}}
        + \left( \| \mathcal{B}_i(\bu_i) - \mathcal{B}_i(\bv_i) \|^2_{L^2(\partial{\Omega_i} \backslash \Gamma \times
        [0,T])} \right)^{\beta_{\mathcal{B}_i}} \right.\right.\notag\\
        &\left.\left.\qquad\qquad+ \left(\| \mathcal{I}_i(\bm{u}_i) - \mathcal{I}_i(\bm{v}_i) \|^2_{L^2({\Omega_i})} \right)^{\beta_{\mathcal{I}_i}} \right]
        +  \left(\| \mathcal{L}_{\Gamma}(\bm{u}_1, \bm{u}_2) - \mathcal{L}_{\Gamma}(\bm{v}_1, \bm{v}_2) \|^2_{L^2(\Gamma \times [0,T])} \right)^{\beta_{\Gamma}} \right\}, \label{ine:pde-assumption}
    \end{align}
where the constant $C_{pde}$ depends on $\| \bm{u}_i\|_{\mathbb{X}(\overline{\Omega_i}\times [0,T])}$ as well as the regularity property of the underlying interface model problem~\eqref{eqn:interface-model}.
\end{assumption}

\begin{assumption} \label{assumption:quad}
    Consider an integral $I(g) := \int_{\mathfrak{D}} g(\bm{x}) \ \mathrm{d}\bm{x}$ and a quadrature rule using $N$ quadrature points $\bm{x}_i \in \mathfrak{D}\ (1 \leq i \leq N)$ as follows,
    \begin{equation*}
        I_N(g) : = \sum_{i=1}^N w_i g(\bm{x}_i), \qquad w_i \in \mathbb{R}^+, \ 1 \leq i \leq N.
    \end{equation*}
    We assume the quadrature error is
    \begin{equation} \label{eqn:quad}
        \vert  I(g) - I_N(g) \vert \leq C_{quad} \, N^{-\alpha}, \quad \alpha > 0,
    \end{equation}
where the constant $C_{quad}$ depends on the dimension of the domain and the property of the integrand $g(\bx)$.
\end{assumption}

Based on the above two assumptions, we have the following framework of the error estimation for the proposed DNN/meshfree method~\eqref{eqn:DNN-interface-method}.
\begin{thm} \label{thm:general-error-analysis}
Let $\bm{u}_i \in \mathbb{X}(\overline{\Omega_i} \times [0,T])$ be the unique solution of the general interface problem~\eqref{eqn:interface-model} and~$\mathcal{U}_{\mathbf{NN}}^{i,\ast}$ be the numerical approximation to $\bu_i\ (i=1,2) $ that is obtained by the DNN/meshfree method~\eqref{eqn:DNN-interface-method}. Under Assumptions~\ref{assumption:pde} and~\ref{assumption:quad}, we have the following estimation on the generalization error,
\begin{align} %
    & \quad  \sum_{i=1}^2 \| \bm{u}_i - \mathcal{U}_{\mathbf{\bN\bN}}^{i, \ast} \|_{\mathbb{X}(\overline{\Omega_i} \times [0,T])}^2 \nonumber \\
    &\leq C_{pde} \left\{  \sum_{i=1}^2 \left[ \left( \mathcal{F}_{\mathcal{L}_i}(\bm{\Theta}^*) \right)^{\beta_{\mathcal{L}_i}} + \left(\mathcal{F}_{\mathcal{B}_i}(\bm{\Theta^*}) \right)^{\beta_{\mathcal{B}_i}} + \left(\mathcal{F}_{\mathcal{I}_i}(\bm{\Theta^*})\right)^{\beta_{\mathcal{I}_i}}\right]+  \left(\mathcal{F}_{\Gamma}(\bm{\Theta}^*) \right)^{\beta_{\Gamma}} \right. \nonumber \\
    & \left.  \qquad \quad + \sum_{i=1}^2 \left(C_{quad}^{\mathcal{L}_i} M_{\mathcal{L}_i}^{-\alpha_{\mathcal{L}_i}\beta_{\mathcal{L}_i}} + C_{quad}^{\mathcal{B}_i} M_{\mathcal{B}_i}^{-\alpha_{\mathcal{B}_i}\beta_{\mathcal{B}_i}} + C_{quad}^{\mathcal{I}_i} M_{\mathcal{I}_i}^{-\alpha_{\mathcal{I}_i}\beta_{\mathcal{I}_i}}\right)+ C_{quad}^{\Gamma} M_{\Gamma}^{-\alpha_{\Gamma}\beta_{\Gamma}} \right\}. \label{ine:general-error}
\end{align}
\end{thm}

\begin{proof}
Consider the error $\bm{u}_i - \mathcal{U}_{\mathbf{\bN\bN}}^{i, \ast}$, $i = 1,2$, based on Assumption~\eqref{assumption:pde}, we have
\begin{align*} %
    \sum_{i=1}^2 \| \bm{u}_i - \mathcal{U}_{\mathbf{\bN\bN}}^{i, \ast} \|_{\mathbb{X}(\overline{\Omega_i} \times [0,T])}^2 &\leq C_{pde} \left\{  \sum_{i=1}^2 \left( \|  \mathcal{L}_i(\mathcal{U}_{\mathbf{\bN\bN}}^{i, \ast}) \|^2_{L^2(\Omega_i \times [0,T])} \right)^{\beta_{\mathcal{L}_i}}  +  \left( \| \mathcal{L}_{\Gamma}(\mathcal{U}_{\mathbf{\bN\bN}}^{1, \ast},\mathcal{U}_{\mathbf{\bN\bN}}^{2, \ast}) \|^2_{L^2(\Gamma \times [0,T])} \right)^{\beta_{\Gamma}} \right. \\
    & \left.  \qquad \quad + \sum_{i=1}^2 \left( \| \mathcal{B}_i(\mathcal{U}_{\mathbf{\bN\bN}}^{i, \ast}) \|^2_{L^2(\partial\Omega_i \backslash \Gamma \times [0,T])} \right)^{\beta_{\mathcal{B}_i}}  + \sum_{i=1}^2 \left(\|\mathcal{I}_i(\mathcal{U}_{\mathbf{\bN\bN}}^{i, \ast}) \|^2_{L^2(\Omega_i)} \right)^{\beta_{\mathcal{I}_i}} \right\}.
\end{align*}
To further estimate the terms on the right hand side, we need to use Assumption~\ref{assumption:quad}.  For example, for $0 < \beta_{\mathcal{L}_i} \leq 1$, we have,
\begin{align*}
\left(\|\mathcal{L}_i(\mathcal{U}_{\mathbf{\bN\bN}}^{i, \ast}) \|_{L^2(\Omega_i \times [0,T])}^2 \right)^{\beta_{\mathcal{L}_i}} = \left( \int_{\Omega_i \times [0,T]} |\mathcal{L}_i(\mathcal{U}_{\mathbf{\bN\bN}}^{i, \ast}(\bm{X}; \bm{\Theta}^*)|^2 \, \mathrm{d} \bm{X} \right)^{\beta_{\mathcal{L}_i}} \leq \left( \mathcal{F}_{\mathcal{L}_i}(\bm{\Theta}^*) \right)^{\beta_{\mathcal{L}_i}} + C_{quad}^{\mathcal{L}_i} M_{\mathcal{L}_i}^{-\alpha_{\mathcal{L}_i}\beta_{\mathcal{L}_i}}.
\end{align*}
The other terms can be estimated in the similar fashion and we can obtain that
\begin{align*} %
    & \quad \sum_{i=1}^2 \| \bm{u}_i - \mathcal{U}_{\mathbf{\bN\bN}}^{i, \ast} \|_{\mathbb{X}(\Omega_i \times [0,T])}^2  \\
    &\leq C_{pde} \left\{  \sum_{i=1}^2 \left( \mathcal{F}_{\mathcal{L}_i}(\bm{\Theta}^*) \right)^{\beta_{\mathcal{L}_i}}  +  \left(\mathcal{F}_{\Gamma}(\bm{\Theta}^*) \right)^{\beta_{\Gamma}} + \sum_{i=1}^2 \left( \mathcal{F}_{\mathcal{B}_i}(\bm{\Theta^*}) \right)^{\beta_{\mathcal{B}_i}} + \sum_{i=1}^2 \left( \mathcal{F}_{\mathcal{I}_i}(\bm{\Theta^*}) \right)^{\beta_{\mathcal{I}_i}} \right. \\
    & \left.  \qquad \quad + \sum_{i=1}^2 C_{quad}^{\mathcal{L}_i} M_{\mathcal{L}_i}^{-\alpha_{\mathcal{L}_i}\beta_{\mathcal{L}_i}} + C_{quad}^{\Gamma} M_{\Gamma}^{-\alpha_{\Gamma}\beta_{\Gamma}}  + \sum_{i=1}^2 C_{quad}^{\mathcal{B}_i} M_{\mathcal{B}_i}^{-\alpha_{\mathcal{B}_i}\beta_{\mathcal{B}_i}} + \sum_{i=1}^2 C_{quad}^{\mathcal{I}_i} M_{\mathcal{I}_i}^{-\alpha_{\mathcal{I}_i}\beta_{\mathcal{I}_i}} \right\}.
\end{align*}
This completes the proof.
\end{proof}

\begin{remark}
When the optimization problem is solved exactly or fairly accurately, the first four terms on the right-hand side of the error estimate~\eqref{ine:general-error} are negligible compared with the last four terms coming from the quadrature error. In this case, the error estimate~\eqref{ine:general-error} serves as a prior error estimation and provides convergence order with respect to the number of sampling points. On the other hand, when the optimization is not solved accurately enough, and its error dominates the quadrature errors, the error estimate~\eqref{ine:general-error} serves as a posteriori error estimation, and one can use the computed solution to estimate how large the error is. This theoretical observation can be used as a posteriori error estimator to design adaptive algorithms.
\end{remark}

\subsection{Error analysis of the DNN method for two interface problems}
In this subsection, we apply the general error estimation, Theorem~\eqref{thm:general-error-analysis}, to two interface model problems introduced in Section~\ref{sec:model:subsec:example} and provide concrete convergence analysis for their DNN/meshfree methods developed in Section \ref{sec:DNN4interfaceproblems}.

\subsubsection{Example 1: The two-phase flow interface problem}
Consider the errors $\bm{e}_i^{\bm{v}} := \bm{v}_i -
\mathcal{V}_{\mathbf{NN}}^{i,*} $ and $e_i^p := p_i -
\mathcal{P}_{\mathbf{NN}}^{i, *} $, $i=1,2$.  By a direct
calculation, we have
\begin{align}
    \bm{\rho_i} \left( \frac{\partial \bm{e}_i^{\bm{v}}}{\partial t} + (\bv_i \cdot \nabla) \bv_i  - (\mathcal{V}_{\mathbf{NN}}^{i,*} \cdot \nabla )\mathcal{V}_{\mathbf{NN}}^{i,*}\right) - \nabla \cdot 2 \mu_i \bD(\be_i^{\bv}) + \nabla e_i^p &= \mathcal{R}^{\bv}_i, \quad \text{in} \ \Omega_i\times(0,T], \notag \\
       \nabla \cdot \bm{e}_i^{\bv} &= \mathcal{R}^{\nabla \cdot}_i, \quad \text{in} \ \Omega_i\times(0,T], \nonumber \\
    \be_{1}^{\bv} - \bm{e}_2^{\bv} &= \mathcal{R}_{\Gamma}^{\bv}, \quad \text{on} \ \Gamma\times[0,T], \nonumber  \\
    \bm{\sigma}_1(\be_1^{\bv}, e_1^p) \bn_1 + \bm{\sigma}_2(\be_2^{\bv}, e_2^p) \bn_2 &= \mathcal{R}_{\Gamma}^{\sigma}, \quad \text{on} \ \Gamma\times[0,T],  \label{eqn:err-v}  \\
    e_i^{\bm{v}} &= \mathcal{R}^{\mathcal{B}}_i, \quad \text{on} \ \partial \Omega_i\backslash \Gamma\times[0,T], \nonumber  \\
    e_i^{\bm{v}}(\bm{x}_i, 0) & = \mathcal{R}^{\mathcal{I}}_i, \quad \text{in} \ \Omega_i, \nonumber
\end{align}
where
\begin{equation*}
\begin{pmatrix}
    \mathcal{R}^{\bv}_i \\
    \mathcal{R}^{\nabla \cdot}_i
\end{pmatrix} := -\mathcal{L}_{i}(\mathcal{U}_{\mathbf{NN}}^{i,\ast}), \quad
\begin{pmatrix}
    \mathcal{R}_{\Gamma}^{\bv} \\
    \mathcal{R}_{\Gamma}^{\bm{\sigma}}
\end{pmatrix}
:= -\mathcal{T}(\mathcal{U}_{\mathbf{NN}}^{1,\ast},\mathcal{U}_{\mathbf{NN}}^{2,\ast}), \quad \mathcal{R}^{\mathcal{B}}_i := -\mathcal{B}_i(\mathcal{U}_{\mathbf{NN}}^{i,\ast}), \quad \text{and} \ \mathcal{R}_i^{\mathcal{I}} : = -\mathcal{I}_i(\mathcal{U}_{\mathbf{NN}}^{i,\ast}), \quad i = 1,2.
\end{equation*}
Using $(\bv_i \cdot \nabla) \bv_i  - (\mathcal{V}_{\mathbf{NN}}^{i,*} \cdot \nabla )\mathcal{V}_{NN}^{i,*} = (\bv_i \cdot \nabla) \be_i^{\bv} + (\be_i^{\bv} \cdot \nabla) \mathcal{V}^{i,*}_{\mathbf{NN}}$, multiplying~\eqref{eqn:err-v}$_1$ by $\be_{i}^{\bv}$, and
integrating over $\Omega_i$, $i=1,2$, we have,
\begin{align*}
 \frac{\mathrm{d}}{\mathrm{d}t} \int_{\Omega_i} \bm{\rho}_i \frac{| \be_i^{\bv} |^2}{2} \mathrm{d} \bx + \int_{\Omega} \bm{\rho}_i \left( (\bv_i \cdot \nabla) \be_i^{\bv} \right) \cdot \be_i^{\bv}  \mathrm{d} \bx + \int_{\Omega_i} \bm{\rho}_i ( (\be_i^{\bv} \cdot \nabla)  \mathcal{V}_{\mathbf{NN}}^{i,*} ) \cdot \be_i^v \mathrm{d} \bx & \\
- \int_{\Omega_i}\left( \nabla \cdot 2\mu_i \bD(\be_i^{\bv})\right) \cdot \be_i^{\bv} \mathrm{d} \bx + \int_{\Omega_i} (\be_i^{\bv} \cdot \nabla) \be_i^p \mathrm{d} \bx &= \int_{\Omega_i} \mathcal{R}_{i}^{\bv} \cdot \be_i^{\bv} \mathrm{d} \bx.
\end{align*}
Applying integration by parts and summing up $i=1,2$, we arrive at
\begin{align}
& \quad \frac{\mathrm{d}}{\mathrm{d}t} \left[ \sum_{i=1}^2 \int_{\Omega_i} \bm{\rho}_i \left( \frac{| \be_i^{\bv} |^2}{2} \right) \mathrm{d} \bx \right] \notag \\
&= -\sum_{i=1}^2 \int_{\Omega_i} \bm{\rho}_i ((\bv_i \cdot \nabla) \be_i^{\bv} ) \cdot \be_i^{\bv}  \mathrm{d} \bx - \sum_{i=1}^2 \int_{\Omega_i} \bm{\rho}_i ( (\be_i^{\bv} \cdot \nabla) \mathcal{V}_{\mathbf{NN}}^{i,*}) \cdot \be_i^{\bv} \mathrm{d} \bx
 + \sum_{i=1}^2 \int_{\Omega_i} \mathcal{R}_i^{\bv} \cdot \be_i^{\bv} \mathrm{d} \bx \notag\\
 & \quad -\sum_{i=1}^2 2 \mu_i \int_{\Omega_i} | \bD(\be_i^{\bv}) |^2 \mathrm{d} \bx + \sum_{i=1}^2 \int_{\Omega_i} \mathcal{R}^{\nabla \cdot}_{i} e_i^p \mathrm{d}\bx + \sum_{i=1}^2   \int_{\partial \Omega_i \backslash \Gamma} \mathcal{R}_i^{\mathcal{B}} \cdot \bm{\sigma}_i(\be_i^{\bv}, e_i^p) \bn_i \mathrm{d} s \label{estimate4twophase}\\
& \quad + \int_{\Gamma} \mathcal{R}_{\Gamma}^{\bv}\cdot \left(
\frac{ \bm{\sigma}_1(\be_1^{\bv}, e_1^p) \bn_1 -
\bm{\sigma}_2(\be_2^{\bv}, e_2^p)\bn_2 }{2}  \right) \mathrm{d} s +
\int_{\Gamma} \mathcal{R}_{\Gamma}^{\bm{\sigma}}\cdot \left(
\frac{\be_1^{\bv} + \be_2^{\bv}}{2} \right) \mathrm{d}s,\notag
\end{align}
where we use the algebraic identity~$ac+bd=\frac{1}{2}[(a-b)(c-d)+(a+b)(c+d)]$.

To estimate the first term on the right hand side of~\eqref{estimate4twophase}, we use Cauchy-Schwarz inequality, Young's inequality, and Korn's inequality to get
\begin{align*}
    \left|\sum_{i=1}^2 \int_{\Omega_i} \bm{\rho}_i ((\bv_i \cdot \nabla) \be_i^{\bv} ) \cdot \be_i^{\bv}  \mathrm{d} \bx\right| &\leq \theta_i \sum_{i=1}^2\int_{\Omega_i} \bm{\rho}_i |\bv_i \cdot \nabla \be_i^{\bv}|^2 \mathrm{d} \bx  + \frac{1}{4\theta_i} \sum_{i=1}^2\int_{\Omega_i} \bm{\rho}_i |\be_i^{\bv}|^2 \mathrm{d} \bx  \\
    & \leq \theta_i C_d \bm{\rho}_i \| \bv_i \|^2_{\infty} \sum_{i=1}^2\int_{\Omega_i} |\nabla \be_i^{\bv}|^2 \mathrm{d} \bx +   \frac{1}{4\theta_i} \sum_{i=1}^2\int_{\Omega_i} \bm{\rho}_i |\be_i^{\bv}|^2 \mathrm{d} \bx \\
    & \leq  \theta_i 2 C_d \bm{\rho}_i \| \bv_i \|^2_{\infty} \sum_{i=1}^2\int_{\Omega_i} |\bD(\be_i^{\bv})|^2 \mathrm{d} \bx +   \frac{1}{4\theta_i} \sum_{i=1}^2\int_{\Omega_i} \bm{\rho}_i |\be_i^{\bv}|^2 \mathrm{d} \bx \\
    & \leq 2\mu_i \sum_{i=1}^2\int_{\Omega_i} |\bD(\be_i^{\bv})|^2 \mathrm{d} \bx + \frac{C_d \rho_i \| \bv_i \|_{\infty}^2}{4 \mu_i}  \sum_{i=1}^2\int_{\Omega_i} \bm{\rho}_i |\be_i^{\bv}|^2 \mathrm{d} \bx,
\end{align*}
where $C_d$ is a positive constant, and we choose~$\theta_i = \frac{\mu_i}{C_d \rho_i \| \bv_i \|^2_{\infty}}$ in the last step. By bounding $\nabla \mathcal{V}_{\mathbf{NN}^{i,*}}$ from above, we can estimate the second term on the right hand side of (\ref{estimate4twophase}) as follows,
\begin{align*}
    \left|\sum_{i=1}^2\int_{\Omega_i} \bm{\rho}_i ( (\be_i^{\bv} \cdot \nabla) \mathcal{V}_{\mathbf{NN}}^{i,*}) \cdot \be_i^{\bv} \mathrm{d} \bx\right| \leq   \sum_{i=1}^2 \| \nabla \mathcal{V}_{\mathbf{NN}}^{i,*} \|_{\infty} \int_{\Omega_i} \bm{\rho}_i |\be_i^{\bv}|^2 \mathrm{d} \bx \leq \left(\max_{i=1,2} \| \nabla \mathcal{V}_{\mathbf{NN}}^{i,*} \|_{\infty}  \right) \int_{\Omega_i} \bm{\rho}_i |\be_i^{\bv}|^2 \mathrm{d} \bx.
\end{align*}
The third term and the last four terms on the right hand side of~\eqref{estimate4twophase} are handled by Cauchy-Schwarz inequality in the inner product space $L^2$ first, and then as a whole, in the Euclidean space $\mathbb{R}^8$, resulting in the following estimation,
\begin{align*}
& \quad \frac{\mathrm{d}}{\mathrm{d}t} \left[ \sum_{i=1}^2 \int_{\Omega_i} \bm{\rho}_i \left( | \be_i^{\bv} |^2 \right) \mathrm{d} \bx \right] \\
& \leq C_1 \left[ \sum_{i=1}^2 \int_{\Omega_i}
|\mathcal{R}_{i}^{\bv}|^2 \mathrm{d} \bx  +  \sum_{i=1}^2
\int_{\Omega_i} |\mathcal{R}_i^{\nabla \cdot}|^2 \mathrm{d} \bx +
\sum_{i=1}^2 \int_{\partial \Omega_i \backslash \Gamma}
|\mathcal{R}_i^{\mathcal{B}}|^2 \mathrm{d}s + \int_{\Gamma}
|\mathcal{R}_{\Gamma}^{\bv}|^2 \mathrm{d} s +   \int_{\Gamma}
|\mathcal{R}_{\Gamma}^{\bm{\sigma}}|^2 \mathrm{d} s
\right]^{\frac{1}{2}} \\
& \quad + C_2 \left( \sum_{i=1}^2 \int_{\Omega_i} \bm{\rho}_i |
\be_i^{\bv}|^2 \mathrm{d} \bx \right) ,
\end{align*}
where $C_1$ and $C_2$ are positive constants depending on $\| \bv_i \|_{C^2(\overline{\Omega_i} \times [0,T])}$, $\| p_i \|_{C^1(\overline{\Omega_i} \times [0,T])}$, $\| \mathcal{V}_{\mathbf{NN}}^{i,*} \|_{(C^2(\overline{\Omega_i} \times [0,T])}$, $\| \mathcal{P}_{\mathbf{NN}}^{i,*} \|_{C^1(\overline{\Omega_i} \times [0,T])}$, $\bm{\rho}_i$, $\mu_i$, $\Omega_i\ (i=1,2)$, and $d$. For any $0\leq\tau \leq T$, we integrate the above estimation over time to obtain
\begin{align*}
    & \sum_{i=1}^2 \int_{\Omega_i} \bm{\rho}_i | \be_i^{\bv}(\bx_i, \tau) |^2 \mathrm{d} \bx \leq \sum_{i=1}^2 \int_{\Omega_i} \bm{\rho}_i |
\mathcal{R}_{i}^{\mathcal{I}} |^2 \mathrm{d} \bx  \\
&+ C_1 T^{\frac{1}{2}} \left[ \int_{0}^T \left( \sum_{i=1}^2
\int_{\Omega_i}  |\mathcal{R}_{i}^{\bv}|^2 \mathrm{d} \bx  +
\sum_{i=1}^2 \int_{\Omega_i} |\mathcal{R}_i^{\nabla \cdot}|^2
\mathrm{d} \bx + \sum_{i=1}^2 \int_{\partial \Omega_i \backslash
\Gamma} |\mathcal{R}_i^{\mathcal{B}}|^2 \mathrm{d}s + \int_{\Gamma}
\left(|\mathcal{R}_{\Gamma}^{\bv}|^2 +
|\mathcal{R}_{\Gamma}^{\bm{\sigma}}|^2\right) \mathrm{d} s
    \right) \mathrm{d} t \right]^{\frac{1}{2}} \\
    & + C_2 \int_{0}^\tau \left( \sum_{i=1}^2 \int_{\Omega_i} \bm{\rho}_i | \be_i^{\bv}(\bx_i, t)|^2 \mathrm{d} \bx \right) \mathrm{d} t \\
    &   =: \mathcal{R} + C_2 \int_{0}^{\tau} \left( \sum_{i=1}^2 \int_{\Omega_i} \bm{\rho}_i | \be_i^{\bv}(\bx_i, t)|^2 \mathrm{d} \bx \right) \mathrm{d}
    t.
\end{align*}
Now apply the Gr\"{o}nwall's inequality, yields,
\begin{align*}
\sum_{i=1}^2 \int_{\Omega_i} \bm{\rho}_i | \be_i^{\bv}(\bx_i, \tau)
|^2 \mathrm{d} \bx  \leq (1 + C_2 T e^{C_2T}) \mathcal{R}.
\end{align*}
Integrate again over $[0,T]$, results in
\begin{align} \label{ine:two-phase-flow-pde-assumption}
    \int_{0}^T \sum_{i=1}^2 \int_{\Omega_i} \bm{\rho}_i | \be_i^{\bv}(\bx_i, t) |^2 \mathrm{d} \bx \mathrm{d} t \leq (T + C_2T^2 e^{C_2 T}) \mathcal{R}.
\end{align}
By the definition of $\mathcal{R}$, \eqref{ine:two-phase-flow-pde-assumption} essentially verifies Assumption~\ref{assumption:pde} for this two-phase flow interface problem with $\beta_{\mathcal{L}_i} = \beta_{\mathcal{B}_i} = \beta_{\Gamma} = \frac{1}{2}$ and $\beta_{\mathcal{I}_i} = 1$, $i=1,2$. Following the abstract analysis presented in Section~\ref{sec:abstract-error-analysis} and using Assumption~\ref{assumption:quad} to bound $\mathcal{R}$ via quadrature errors, we have the following estimation on the generalization error for the presented two-phase flow interface problem.
\begin{thm} \label{thm:error-two-phase-fluid}
Let $\bm{v}_i \in C^2(\overline{\Omega_i} \times [0,T])$ and $p_i \in C^1(\overline{\Omega_i} \times [0,T])$ be the classical solution of the two-phase flow interface problem~\eqref{eqn:two-phase-flow-interface}, correspondingly, let~$\mathcal{V}_{\mathbf{NN}}^{i,*}\in C^2(\overline{\Omega_i} \times [0,T])$ and~$\mathcal{P}_{\mathbf{NN}}^{i,*}\in C^1(\overline{\Omega_i} \times [0,T])$ be the numerical approximations obtained by the DNN/meshfree method. We have the following generalization error estimation,
\begin{align*}
&\int_{0}^T \sum_{i=1}^2 \int_{\Omega_i}  | \be_i^{\bv}(\bx_i, t) |^2 \mathrm{d} \bx \mathrm{d} t \leq
 C_1T(1 + C_2T e^{C_2 T}) \left[ \sum_{i=1}^2  \mathcal{F}_{\mathcal{I}_i}(\bm{\Theta}^*) + \sum_{i=1}^2 
C_{quad}^{\mathcal{I}_i} M_{\mathcal{I}_i}^{-\alpha_{\mathcal{I}_i}} \right] \\
& + C_1T^{\frac{3}{2}}(1 + C_2T e^{C_2 T}) \left[
\sum_{i=1}^2 \mathcal{F}_{\mathcal{L}_i}(\bm{\Theta}^*) +
\sum_{i=1}^2 \mathcal{F}_{\mathcal{B}_i}(\bm{\Theta}^*)  +
\mathcal{F}_{\Gamma}(\bm{\Theta}^*)
 + \sum_{i=1}^2 C_{quad}^{\mathcal{L}_i} M_{\mathcal{L}_i}^{-\alpha_{\mathcal{L}_i}} + \sum_{i=1}^2 C_{quad}^{\mathcal{B}_i} M_{\mathcal{B}_i}^{-\alpha_{\mathcal{B}_i}}
+ C_{quad}^{\Gamma} M_{\Gamma}^{-\alpha_{\Gamma}}
\right]^{\frac{1}{2}},
\end{align*}
where $C_1$ and $C_2$ are positive constants depending on $\| \bv_i \|_{C^2(\overline{\Omega_i} \times [0,T])}$, $\| p_i \|_{C^1(\overline{\Omega_i} \times [0,T])}$, $\| \mathcal{V}_{\mathbf{NN}}^{i,*} \|_{C^2(\overline{\Omega_i} \times [0,T])}$, $\| \mathcal{P}_{\mathbf{NN}}^{i,*} \|_{C^1(\overline{\Omega_i} \times [0,T])}$, $\bm{\rho}_i$, $\mu_i$, $\Omega_i\ (i=1,2)$, and $d$.
\end{thm}
\begin{proof}
Using Assumption~\ref{assumption:quad}, we have
\begin{align*}
\mathcal{R} &\leq \sum_{i=1}^2 \bm{\rho}_i \mathcal{F}_{\mathcal{I}_i}(\bm{\Theta}^*) + \sum_{i=1}^2 \bm{\rho_i} C_{quad}^{\mathcal{I}_i} M_{\mathcal{I}_i}^{-\alpha_{\mathcal{I}_i}} \\
& \quad  + C_1 T^{\frac{1}{2}} \left[  \sum_{i=1}^2 \mathcal{F}_{\mathcal{L}_i}(\bm{\Theta}^*) + \sum_{i=1}^2 C_{quad}^{\mathcal{L}_i} M_{\mathcal{L}_i}^{-\alpha_{\mathcal{L}_i}}      \right.  \\
& \qquad \qquad \quad + \sum_{i=1}^2 \mathcal{F}_{\mathcal{B}_i}(\bm{\Theta}^*) + \sum_{i=1}^2 C_{quad}^{\mathcal{B}_i} M_{\mathcal{B}_i}^{-\alpha_{\mathcal{B}_i}} \\
& \qquad \qquad \quad \left. + \mathcal{F}_{\Gamma}(\bm{\Theta}^*)
 + C_{quad}^{\Gamma} M_{\Gamma}^{-\alpha_{\Gamma}}
 \right]^{\frac{1}{2}}.
\end{align*}
Substituting the above estimation back into~\eqref{ine:two-phase-flow-pde-assumption}, we complete the proof.
\end{proof}

\subsubsection{Example 2: The FSI problems}\label{sec:fsi-erroranalysis}
Consider the errors~$\be_f^{\bv} := \bv_f -\mathcal{V}_{\mathbf{NN}}^{f,*}$, $e_f^p := p_f -\mathcal{P}_{\mathbf{NN}}^{f,*}$, and~$\be_{s}^{\bu} := \bu_s -\mathcal{U}_{\mathbf{NN}}^{s,*}$. By a direct calculation on the FSI model with a wave-type structural equation, we have
\begin{align}
    \bm{\rho_f} \left( \frac{\partial \bm{e}_f^{\bm{v}}}{\partial t} + (\bv_f \cdot \nabla) \bv_f  - (\mathcal{V}_{\mathbf{NN}}^{f,*} \cdot \nabla )\mathcal{V}_{\mathbf{NN}}^{f,*}\right) - \nabla \cdot 2 \mu_i \bD(\be_f^{\bv}) + \nabla e_f^p &= \mathcal{R}^{\bv}_f, \quad \text{in} \ \Omega_f\times(0,T], \nonumber \\
    \nabla \cdot \bm{e}_f^{\bv} &= \mathcal{R}^{\nabla \cdot}_f, \quad \text{in} \ \Omega_f\times(0,T], \nonumber \\
    \bm{\rho_s} \left( \frac{\partial^2 \be_{s}^{\bu}}{\partial t^2}  \right) - \nabla \cdot \bm{\sigma}_s(\be_{s}^{\bu}) & = \mathcal{R}_s^{\bu}, \quad \text{in} \ \Omega_s\times(0,T], \notag\\
    \be_{f}^{\bv} - \frac{\partial \bm{e}_s^{\bu}}{\partial t} &= \mathcal{R}_{\Gamma}^{\bv}, \quad \text{on} \ \Gamma\times[0,T], \nonumber  \\
    \bm{\sigma}_f(\be_f^{\bv}, e_f^p) \bn_f + \bm{\sigma}_s(\be_s^{\bu}) \bn_s &= \mathcal{R}_{\Gamma}^{\sigma}, \quad \text{on} \ \Gamma\times[0,T], \label{eqn:err-u}  \\
    e_f^{\bm{v}} &= \mathcal{R}^{\mathcal{B}}_f, \quad \text{on} \ \partial \Omega_f\backslash \Gamma\times[0,T],\nonumber  \\
    e_s^{\bm{u}} &= \mathcal{R}^{\mathcal{B}}_s, \quad \text{on} \ \partial \Omega_s\backslash \Gamma\times[0,T], \notag  \\
    e_f^{\bm{v}}(\bm{x}_f, 0) & = \mathcal{R}^{\mathcal{I}}_f, \quad \text{in} \ \Omega_f, \nonumber \\
    e_s^{\bm{u}}(\bm{x}_s, 0) & = \mathcal{R}^{\mathcal{I}, \bu}_s, \quad \text{in} \ \Omega_s, \nonumber  \\
    \frac{\partial e_s^{\bm{u}}}{\partial t}(\bm{x}_s, 0) & = \mathcal{R}^{\mathcal{I}, \bv}_s, \quad \text{in} \ \Omega_s, \nonumber
\end{align}
where $\mathcal{U}^{1,*}_{\mathbf{NN}}=(\mathcal{V}^{f,*}_{\mathbf{NN}}, \mathcal{P}^{f,*}_{\mathbf{NN}})$, $\mathcal{U}^{2,*}_{\mathbf{NN}} = \mathcal{U}^{s,*}_{\mathbf{NN}}$,
\begin{align*}
    \begin{pmatrix}
        \mathcal{R}^{\bv}_f \\
        \mathcal{R}^{\nabla \cdot}_f
    \end{pmatrix} &:= -\mathcal{L}_{f}(\mathcal{U}_{\mathbf{NN}}^{1,\ast}), \
    \mathcal{R}_s^{\bu} := - \mathcal{L}_s(\mathcal{U}_{\mathbf{NN}}^{2,*}), \
    \begin{pmatrix}
        \mathcal{R}_{\Gamma}^{\bv} \\
        \mathcal{R}_{\Gamma}^{\bm{\sigma}}
    \end{pmatrix}
    := -\mathcal{T}(\mathcal{U}_{\mathbf{NN}}^{1,\ast},\mathcal{U}_{\mathbf{NN}}^{2,\ast}), \\
    \mathcal{R}^{\mathcal{B}}_f &:= -\mathcal{B}_f(\mathcal{U}_{\mathbf{NN}}^{1,\ast}), \
     \mathcal{R}^{\mathcal{B}}_s := -\mathcal{B}_s(\mathcal{U}_{\mathbf{NN}}^{2,\ast}), \
    \mathcal{R}_f^{\mathcal{I}} : = -\mathcal{I}_f(\mathcal{U}_{\mathbf{NN}}^{1,\ast}), \
    \mathcal{R}_s^{\mathcal{I},\bu} : = -\mathcal{I}_{s, \bu}(\mathcal{U}_{\mathbf{NN}}^{2,\ast}), \
    \mathcal{R}_s^{\mathcal{I}, \bv} : = -\mathcal{I}_{s, \bv}(\mathcal{U}_{\mathbf{NN}}^{2,\ast}).
\end{align*}
For the fluid part, we follow a similar derivation for the two-phase flow interface problem and obtain,
\begin{align*}
    \frac{\mathrm{d}}{\mathrm{d}t} \int_{\Omega_f} \bm{\rho}_f \frac{| \be_f^{\bv} |^2}{2} \mathrm{d} \bx + \int_{\Omega_f} \bm{\rho}_f \left( (\bv_f \cdot \nabla) \be_f^{\bv} \right) \cdot \be_f^{\bv}  \mathrm{d} \bx + \int_{\Omega_f} \bm{\rho}_f ( (\be_f^{\bv} \cdot \nabla)  \mathcal{V}_{\mathbf{NN}}^{f,*} ) \cdot \be_f^v \mathrm{d} \bx & \\
    - \int_{\Omega_f}\left( \nabla \cdot 2\mu_f \bD(\be_f^{\bv})\right) \cdot \be_f^{\bv} \mathrm{d} \bx + \int_{\Omega_f} (\be_f^{\bv} \cdot \nabla) \be_f^p \mathrm{d} \bx &= \int_{\Omega_f} \mathcal{R}_{f}^{\bv} \cdot \be_f^{\bv} \mathrm{d} \bx.
\end{align*}
For the structural part, we multiply~\eqref{eqn:err-u}$_3$ by~$\frac{\partial \be_s^{\bu}}{\partial t}$, integrate over $\Omega_s$ and conduct integration by parts to obtain,
\begin{align*}
\frac{1}{2}\frac{\mathrm{d}}{\mathrm{d}t} \int_{\Omega_s}
\bm{\rho_s} \left|\frac{\partial \be_s^{\bu}}{\partial t}\right|^2
\mathrm{d} \bx + \frac{1}{2} \frac{\mathrm{d}}{\mathrm{d}t}
\int_{\Omega_s} 2 \mu_s |\varepsilon(\be_{s}^{\bu})|^2 + \lambda_s
|\nabla \cdot \be_{s}^{\bu}|^2 \mathrm{d}\bx = \int_{\Omega_s}
\mathcal{R}_s^{\bu} \cdot \frac{\partial \be_s^{\bu}}{\partial t}
\mathrm{d} \bx + \int_{\partial \Omega_s}
\bm{\sigma}(\be_s^{\bu})\bn_s \cdot \frac{\partial
\be_s^{\bu}}{\partial t}\mathrm{d}s.
\end{align*}
Apply integration by parts to the fluid part as well, and sum up both the fluid and the structural parts, yield
\begin{align}
& \qquad  \frac{1}{2} \left[\frac{\mathrm{d}}{\mathrm{d}t} \int_{\Omega_f} \bm{\rho}_f | \be_f^{\bv} |^2 \mathrm{d} \bx + \frac{\mathrm{d}}{\mathrm{d}t} \int_{\Omega_s} \bm{\rho_s} \left|\frac{\partial \be_s^{\bu}}{\partial t}\right|^2 \mathrm{d} \bx +  \frac{\mathrm{d}}{\mathrm{d}t} \int_{\Omega_s} 2 \mu_s |\varepsilon(\be_{s}^{\bu})|^2 + \lambda_s |\nabla \cdot \be_{s}^{\bu}|^2 \mathrm{d}\bx \right] \notag\\
&   =- \int_{\Omega_f} \bm{\rho}_f \left( (\bv_f \cdot \nabla) \be_f^{\bv} \right) \cdot \be_f^{\bv}  \mathrm{d} \bx  - \int_{\Omega_f} \bm{\rho}_f ( (\be_f^{\bv} \cdot \nabla)  \mathcal{V}_{\mathbf{NN}}^{f,*} ) \cdot \be_f^{\bv} \mathrm{d} \bx  - 2\mu_f \int_{\Omega_f} | \bD(\be_f^{\bv}) |^2 \mathrm{d} \bx \notag\\
& \quad + \int_{\Omega_f} \mathcal{R}_f^{\bv} \cdot \be_f^{\bv} \mathrm{d} \bx +  \int_{\Omega_f} \mathcal{R}^{\nabla \cdot}_{f} e_f^p \mathrm{d}\bx + \int_{\Omega_s} \mathcal{R}_s^{\bu} \cdot \frac{\partial \be_s^{\bu}}{\partial t} \mathrm{d} \bx   \notag\\
& \quad +   \int_{\partial \Omega_f \backslash \Gamma} \mathcal{R}_f^{\mathcal{B}} \cdot \bm{\sigma}_f(\be_f^{\bv}, e_f^p) \bn_f \mathrm{d} s + \int_{\partial \Omega_s \backslash \Gamma} \bm{\sigma}(\be_s^{\bu}) \bn_s \cdot \frac{\partial \mathcal{R}_s^{\mathcal{B}}}{\partial t} \mathrm{d}s  \notag\\
& \quad + \int_{\Gamma} \mathcal{R}_{\Gamma}^{\bv} \left( \frac{
\bm{\sigma}_f(\be_f^{\bv}, e_f^p) \bn_f -
\bm{\sigma}_s(\be_s^{\bu})\bn_s }{2}  \right) \mathrm{d} s +
\int_{\Gamma} \mathcal{R}_{\Gamma}^{\bm{\sigma}} \left[
\frac{1}{2}\left( \be_f^{\bv} + \frac{ \partial
\be_s^{\bu}}{\partial t} \right) \right]
\mathrm{d}s.\label{error_FSI1}
\end{align}
The first two terms on the right hand side of (\ref{error_FSI1}) can be estimated as before and other terms can be handled by Cauchy-Schwarz inequality and Young's inequality. This, therefore, leads to the following error estimation,
\begin{align}
& \qquad  \left[\frac{\mathrm{d}}{\mathrm{d}t} \int_{\Omega_f} \bm{\rho}_f | \be_f^{\bv} |^2 \mathrm{d} \bx + \frac{\mathrm{d}}{\mathrm{d}t} \int_{\Omega_s} \bm{\rho_s} \left|\frac{\partial \be_s^{\bu}}{\partial t}\right|^2 \mathrm{d} \bx +  \frac{\mathrm{d}}{\mathrm{d}t} \int_{\Omega_s} 2 \mu_s |\varepsilon(\be_{s}^{\bu})|^2 + \lambda_s |\nabla \cdot \be_{s}^{\bu}|^2 \mathrm{d}\bx \right]  \notag\\
& \leq C_1 \left[ \int_{\Omega_f} |\mathcal{R}_f^{\bv}|^2 \mathrm{d} \bx + \int_{\Omega_f} |\mathcal{R}_f^{\nabla \cdot}|^2 \mathrm{d} \bx  + \int_{\Omega_s} |\mathcal{R}_s^{\bu} |^2 \mathrm{d} \bx + \int_{\partial \Omega_f\backslash \Gamma} |\mathcal{R}_f^{\mathcal{B}}|^2 \mathrm{d} s + \int_{\partial \Omega_s \backslash \Gamma} \left| \frac{\partial \mathcal{R}_s^{\mathcal{B}}}{\partial t} \right|^2 \mathrm{d}s +  \int_{\Gamma} |\mathcal{R}_{\Gamma}^{\bv}|^2 \mathrm{d} s +   \int_{\Gamma} |\mathcal{R}_{\Gamma}^{\bm{\sigma}}|^2 \mathrm{d} s   \right]^{\frac{1}{2}} \notag\\
& \quad + C_2 \int_{\Omega_f} \bm{\rho_f} |\be_f^{\bv}|^2 \mathrm{d}
\bx,\label{error_FSI1_1}
\end{align}
where $C_1$ and $C_2$ are positive constants depending on $\| \bv_f
\|_{C^2(\overline{\Omega_i} \times [0,T])}$, $\| p_f
\|_{C^1(\overline{\Omega_f} \times [0,T])}$, $\|
\mathcal{V}_{\mathbf{NN}}^{f,*} \|_{C^2(\overline{\Omega_f} \times
[0,T])}$, $\| \mathcal{P}_{\mathbf{NN}}^{f,*}
\|_{C^1(\overline{\Omega_f} \times [0,T])}$, $\| \bu_s
\|_{C^2(\overline{\Omega_s} \times [0,T])}$, $\|
\mathcal{U}_{\mathbf{NN}}^{s,*} \|_{C^2(\overline{\Omega_s} \times
[0,T])}$, $d$, $\bm{\rho}_f$, $\bm{\rho}_s$, $\mu_f$, $\mu_s$,
$\lambda_s$, $\Omega_f$, and~$\Omega_s$.  For any $0\leq \tau \leq
T$, we integrate (\ref{error_FSI1_1}) over time to obtain
\begin{align}
& \qquad  \int_{\Omega_f} \bm{\rho}_f | \be_f^{\bv}(\bx_f, \tau) |^2 \mathrm{d} \bx +  \int_{\Omega_s} \bm{\rho_s} \left|\frac{\partial \be_s^{\bu}}{\partial t}(\bx_s, \tau)\right|^2 \mathrm{d} \bx + \int_{\Omega_s} 2 \mu_s |\varepsilon(\be_{s}^{\bu}(\bx_s,\tau))|^2 + \lambda_s |\nabla \cdot \be_{s}^{\bu}(\bx_s,\tau)|^2 \mathrm{d}\bx   \notag\\
& \leq \int_{\Omega_f} \bm{\rho}_f |\mathcal{R}_f^{\mathcal{I}}|^2 \mathrm{d} \bx + \int_{\Omega_s} \bm{\rho}_s |\mathcal{R}_s^{\mathcal{I},\bv}|^2 \mathrm{d}\bx + \int_{\Omega_s} 2 \mu_s |\varepsilon(\mathcal{R}_{s}^{\mathcal{I},\bu})|^2 + \lambda_s |\nabla \cdot \mathcal{R}_s^{\mathcal{I},\bu}|^2 \mathrm{d} \bx + \int_{\Omega_s} |\mathcal{R}_s^{\mathcal{I},\bu}|^2 \mathrm{d} \bx \notag\\
& \quad + C_1 T^{\frac{1}{2}} \left[ \int_0^T \left( \int_{\Omega_f} |\mathcal{R}_f^{\bv}|^2 \mathrm{d} \bx + \int_{\Omega_f} |\mathcal{R}_f^{\nabla \cdot}|^2 \mathrm{d} \bx  + \int_{\Omega_s} |\mathcal{R}_s^{\bu} |^2 \mathrm{d} \bx \right. \right. \notag\\
& \qquad \qquad \  \left. \left. + \int_{\partial \Omega_f\backslash \Gamma} |\mathcal{R}_f^{\mathcal{B}}|^2 \mathrm{d} s +  \int_{\partial \Omega_s \backslash \Gamma} \left| \frac{\partial \mathcal{R}_s^{\mathcal{B}}}{\partial t} \right|^2 \mathrm{d}s + \int_{\partial_{\Omega_s} \backslash \Gamma} \left| \mathcal{R}_s^{\mathcal{B}} \right|^2 \mathrm{d}s  +  \int_{\Gamma} |\mathcal{R}_{\Gamma}^{\bv}|^2 \mathrm{d}s +   \int_{\Gamma} |\mathcal{R}_{\Gamma}^{\bm{\sigma}}|^2 \mathrm{d} s \right) \mathrm{d} t  \right]^{\frac{1}{2}} \notag\\
& \quad + C_2 \int_0^\tau \left( \int_{\Omega_f} \bm{\rho}_f |\be_f^{\bv}(\bx_f, t)|^2 \mathrm{d} \bx \right) \mathrm{d}t \label{errorestimate-FSI1}\\
& =: \mathcal{R} + C_2 \int_0^\tau \left( \int_{\Omega_f}
\bm{\rho}_f |\be_f^{\bv}(\bx_f, t)|^2 \mathrm{d} \bx \right)
\mathrm{d}t.\notag
\end{align}
Note that the fourth and the tenth integral terms on the right-hand side of (\ref{errorestimate-FSI1}), i.e., $\int_{\Omega_s} |\mathcal{R}_s^{\mathcal{I},\bu}|^2 \mathrm{d} \bx$ and $\int_{\partial_{\Omega_s} \backslash \Gamma} \left| \mathcal{R}_s^{\mathcal{B}} \right|^2 \mathrm{d}s$, are not directly derived from the time integration. We add them based on our practical experiences since they provide better accuracy in our numerical experiments. On the other hand, these two terms also directly come from the LS formulation of the initial and the boundary condition of the structural equation, respectively. Without these two terms, then only the derivatives of $\mathcal{R}_s^{\mathcal{I}, \bu}$ and $\mathcal{R}_s^{\mathcal{B}}$ (i.e., the third and the ninth integral terms on the right-hand side) are involved, which can only be determined up to a constant. Adding the fourth and the tenth integral terms helps to determine the correct constants that match the given initial and boundary conditions of the structural part. 

Now applying the Gr\"{o}nwall's inequality and integrating again over $[0,T]$, we have
\begin{align*} 
& \quad 	\int_{0}^T \left(
\int_{\Omega_f} \bm{\rho}_f | \be_f^{\bv}(\bx_f, t) |^2 \mathrm{d}
\bx +  \int_{\Omega_s} \bm{\rho_s} \left|\frac{\partial
\be_s^{\bu}}{\partial t}(\bx_s, t)\right|^2 \mathrm{d} \bx +
\int_{\Omega_s} 2 \mu_s |\varepsilon(\be_{s}^{\bu}(\bx_s,t))|^2 +
\lambda_s |\nabla \cdot \be_{s}^{\bu}(\bx_s,t)|^2 \mathrm{d}\bx
\right) \mathrm{d}t \\
&\leq (T + C_2T^2 e^{C_2T}) \mathcal{R}.
\end{align*}
Following the similar steps in the proof for
Theorem~\ref{thm:error-two-phase-fluid}, we can derive the following
generalization error bound for the FSI problem with the wave-type
structural equation.
\begin{thm}\label{thm:error-FSI}
Let~$\bv_f \in C^2(\overline{\Omega_f} \times [0,T])$, $p_f \in C^1(\overline{\Omega_f} \times [0,T])$, and~$\bu_s \in C^2(\overline{\Omega_s} \times [0,T])$ be the classical solutions of the FSI problem with the wave-type structural equation, correspondingly, let~$\mathcal{V}_{\mathbf{NN}}^{f,*} \in C^2(\overline{\Omega_f} \times [0,T])$, $\mathcal{P}_{\mathbf{NN}}^{f,*} \in C^1(\overline{\Omega_f} \times [0,T]) $, and~$\mathcal{U}_{\mathbf{NN}}^{s,*} \in C^2(\overline{\Omega_s} \times [0,T])$ be the numerical approximations obtained by the DNN/meshfree method. We have the following generalization error estimation,
\begin{align*}
& \quad \int_{0}^T \left( \int_{\Omega_f} \bm{\rho}_f | \be_f^{\bv}(\bx_f, s) |^2 \mathrm{d} \bx +  \int_{\Omega_s} \bm{\rho_s} \left|\frac{\partial \be_s^{\bu}}{\partial t}(\bx_s, s)\right|^2 \mathrm{d} \bx + \int_{\Omega_s} 2 \mu_s |\varepsilon(\be_{s}^{\bu}(\bx_s,s))|^2 + \lambda_s |\nabla \cdot \be_{s}^{\bu}(\bx_s,s)|^2 \mathrm{d}\bx \right) \mathrm{d}t      \\
& \leq
C_1T(1 + C_2T e^{C_2T}) \left[  \mathcal{F}_{\mathcal{I}_f}(\bm{\Theta}^*) +  \mathcal{F}_{\mathcal{I}_{s}}(\bm{\Theta}^*) +  C_{quad}^{\mathcal{I}_f} M_{\mathcal{I}_f}^{-\alpha_{\mathcal{I}_f}} + C_{quad}^{\mathcal{I}_{s}} M_{\mathcal{I}_{s}}^{-\alpha_{\mathcal{I}_{s}}} \right] \\
& \quad + C_1 T^{\frac{3}{2}}(1 + C_2T e^{C_2T}) \left[
\mathcal{F}_{\mathcal{L}_f}(\bm{\Theta}^*) +
\mathcal{F}_{\mathcal{L}_s}(\bm{\Theta}^*) +
\mathcal{F}_{\mathcal{B}_f}(\bm{\Theta}^*) +
\mathcal{F}_{\mathcal{B}_s}(\bm{\Theta}^*) +
\mathcal{F}_{\Gamma}(\bm{\Theta}^*)
 \right. \\
& \qquad \qquad \qquad \qquad \quad \ \qquad \left. +
C_{quad}^{\mathcal{F}_f} M_{\mathcal{F}_f}^{-\alpha_{\mathcal{F}_f}}
+ C_{quad}^{\mathcal{F}_s}
M_{\mathcal{F}_s}^{-\alpha_{\mathcal{F}_s}} +
C_{quad}^{\mathcal{B}_f} M_{\mathcal{B}_f}^{-\alpha_{\mathcal{B}_f}}
+ C_{quad}^{\mathcal{B}_s}
M_{\mathcal{B}_s}^{-\alpha_{\mathcal{B}_s}} +
C_{quad}^{\Gamma}M_{\Gamma}^{-\alpha_{\Gamma}}
\right]^{\frac{1}{2}},
\end{align*}
where~$C_1$ and~$C_2$ are positive constants depending on~$\| \bv_f
\|_{C^2(\overline{\Omega_f} \times [0,T])}$, $\| p_f
\|_{C^1(\overline{\Omega_f} \times [0,T])}$, $\| \bu_s
\|_{C^2(\overline{\Omega_s} \times [0,T])}$, $\|
\mathcal{V}_{\mathbf{NN}}^{f,*} \|_{C^2(\overline{\Omega_f} \times
[0,T])}$, $\| \mathcal{P}_{\mathbf{NN}}^{f,*}
\|_{C^1(\overline{\Omega_f} \times [0,T])}$, $\|
\mathcal{U}_{\mathbf{NN}}^{s,*} \|_{C^2(\overline{\Omega_s} \times
[0,T])}$, $\bm{\rho}_f$, $\bm{\rho}_s$, $\mu_f$, $\mu_s$,
$\lambda_s$, $\Omega_f$, $\Omega_s$, and~$d$.
\end{thm}

\begin{proof}
The result is obtained by using the following bound on~$\mathcal{R}$, which is a direct consequence of Assumption~\ref{assumption:quad},
\begin{align*}
\mathcal{R} & \leq
C_1\left(\mathcal{F}_{\mathcal{I}_f}(\bm{\Theta}^*) +
C_{quad}^{\mathcal{I}_f} M_{\mathcal{I}_f}^{-\alpha_{\mathcal{I}_f}}
+  \mathcal{F}_{\mathcal{I}_{s}}(\bm{\Theta}^*) +  C_{quad}^{\mathcal{I}_{s}} M_{\mathcal{I}_{s}}^{-\alpha_{\mathcal{I}_{s}}} \right) \\
& \quad + C_1 T^{\frac{1}{2}} \left[ \mathcal{F}_{\mathcal{L}_f}(\bm{\Theta}^*) + C_{quad}^{\mathcal{F}_f} M_{\mathcal{F}_f}^{-\alpha_{\mathcal{F}_f}} + \mathcal{F}_{\mathcal{L}_s}(\bm{\Theta}^*) + C_{quad}^{\mathcal{F}_s} M_{\mathcal{F}_s}^{-\alpha_{\mathcal{F}_s}} + \mathcal{F}_{\mathcal{B}_f}(\bm{\Theta}^*) + C_{quad}^{\mathcal{B}_f} M_{\mathcal{B}_f}^{-\alpha_{\mathcal{B}_f}}  \right. \\
& \qquad \qquad \ \ \ \left. + \mathcal{F}_{\mathcal{B}_s}(\bm{\Theta}^*) + C_{quad}^{\mathcal{B}_s} M_{\mathcal{B}_s}^{-\alpha_{\mathcal{B}_s}} +
\mathcal{F}_{\mathcal{L}_f}(\bm{\Theta}^*) + C_{quad}^{\Gamma} M_{\Gamma}^{-\alpha_{\Gamma}}  \right]^{\frac{1}{2}}.
\end{align*}
\end{proof}

\begin{remark}
The derivation above shows that two extra terms, i.e., the third and the ninth integral terms on the right-hand side of \eqref{LS4FSI} that we added to our DNN/meshfree method for solving the FSI problem with a wave-type structural equation naturally arise from the time integration and integration by parts shown in \eqref{errorestimate-FSI1}. For the theory we developed here, those two terms are needed. We also use those two terms in our implementation to verify the theoretical results we developed to solve the FSI problem with a wave-type structural equation.
\end{remark}

As discussed before, by bringing back the structural velocity~$\bv_s$, we can reformulate the FSI problem using a parabolic-like structural equation. In this case, let us consider the errors~$\be_f^{\bv} := \bv_f -
\mathcal{V}_{\mathbf{NN}}^{f,*}$, $e_f^p := p_f -
\mathcal{P}_{\mathbf{NN}}^{f,*}$, $\be_s^{\bu} := \bu_s -
\mathcal{U}_{\mathbf{NN}}^{s,*}$, and $\be_s^{\bv} := \bv_s -
\mathcal{V}_{\mathbf{NN}}^{s,*}$. Then, we have
\begin{align}
    \bm{\rho_f} \left( \frac{\partial \bm{e}_f^{\bm{v}}}{\partial t} + (\bv_f \cdot \nabla) \bv_f  - (\mathcal{V}_{\mathbf{NN}}^{f,*} \cdot \nabla )\mathcal{V}_{\mathbf{NN}}^{f,*}\right) - \nabla \cdot 2 \mu_i \bD(\be_f^{\bv}) + \nabla e_f^p &= \mathcal{R}^{\bv}_f, \quad \text{in} \ \Omega_f\times(0,T], \nonumber \\
    \nabla \cdot \bm{e}_f^{\bv} &= \mathcal{R}^{\nabla \cdot}_f, \quad \text{in} \ \Omega_f\times(0,T], \nonumber \\
    \bm{\rho_s} \left( \frac{\partial \be_{s}^{\bv}}{\partial t}  \right) - \nabla \cdot \bm{\sigma}_s(\be_{s}^{\bu}) & = \mathcal{R}_s^{\bv}, \quad \text{in} \ \Omega_s\times(0,T], \notag \\
    \frac{\partial \be_s^{\bu}}{\partial t} - \be_s^{\bv} &= \mathcal{R}_s^{\bu}, \quad \text{in} \ \Omega_s\times(0,T], \notag \\
    \be_{f}^{\bv} -  \bm{e}_s^{\bv} &= \mathcal{R}_{\Gamma}^{\bv}, \quad \text{on} \ \Gamma\times[0,T], \nonumber  \\
    \bm{\sigma}_f(\be_f^{\bv}, e_f^p) \bn_f + \bm{\sigma}_s(\be_s^{\bu}) \bn_s &= \mathcal{R}_{\Gamma}^{\sigma}, \quad \text{on} \ \Gamma\times[0,T],  \label{eqn:FSI-err-v}  \\
    e_f^{\bm{v}} &= \mathcal{R}^{\mathcal{B}}_f, \quad \text{on} \ \partial \Omega_f\backslash \Gamma\times[0,T], \nonumber  \\
    e_s^{\bm{v}} &= \mathcal{R}^{\mathcal{B}}_s, \quad \text{on} \ \partial \Omega_s\backslash \Gamma\times[0,T], \nonumber  \\
    e_f^{\bm{v}}(\bm{x}_f, 0) & = \mathcal{R}^{\mathcal{I}}_f, \quad \text{in} \ \Omega_f, \nonumber \\
    e_s^{\bm{u}}(\bm{x}_s, 0) & = \mathcal{R}^{\mathcal{I}, \bu}_s, \quad \text{in} \ \Omega_s, \nonumber  \\
    e_s^{\bm{v}}(\bm{x}_s, 0) & = \mathcal{R}^{\mathcal{I}, \bv}_s, \quad \text{in} \ \Omega_s, \nonumber
\end{align}
where $\mathcal{U}^{1,*}_{\mathbf{NN}} = (\mathcal{V}^{f,*}_{\mathbf{NN}}, \mathcal{P}^{f,*}_{\mathbf{NN}})$, $\mathcal{U}^{2,*}_{\mathbf{NN}} = (\mathcal{U}^{s,*}_{\mathbf{NN}}, \mathcal{V}^{s,*}_{\mathbf{NN}})$,
\begin{align*}
    \begin{pmatrix}
        \mathcal{R}^{\bv}_f \\
        \mathcal{R}^{\nabla \cdot}_f
    \end{pmatrix} &:= -\mathcal{L}_{f}(\mathcal{U}_{\mathbf{NN}}^{1,\ast}), \
    \begin{pmatrix}
        \mathcal{R}^{\bv}_s \\
        \mathcal{R}^{\bu}_s
    \end{pmatrix} := - \mathcal{L}_s(\mathcal{U}_{\mathbf{NN}}^{2,*}), \
    \begin{pmatrix}
        \mathcal{R}_{\Gamma}^{\bv} \\
        \mathcal{R}_{\Gamma}^{\bm{\sigma}}
    \end{pmatrix}
    := -\mathcal{T}(\mathcal{U}_{\mathbf{NN}}^{1,\ast},\mathcal{U}_{\mathbf{NN}}^{2,\ast}), \\
    \mathcal{R}^{\mathcal{B}}_f &:= -\mathcal{B}_f(\mathcal{U}_{\mathbf{NN}}^{1,\ast}), \
    \mathcal{R}^{\mathcal{B}}_s := -\mathcal{B}_s(\mathcal{U}_{\mathbf{NN}}^{2,\ast}), \
    \mathcal{R}_f^{\mathcal{I}} : = -\mathcal{I}_f(\mathcal{U}_{\mathbf{NN}}^{1,\ast}), \
    \begin{pmatrix}
    \mathcal{R}_s^{\mathcal{I},\bu}  \\
    \mathcal{R}_s^{\mathcal{I},\bv}
    \end{pmatrix}
: = -\mathcal{I}_s(\mathcal{U}_{\mathbf{NN}}^{1,\ast}).
\end{align*}

For the fluid part, we handle it as before. For the structural part, we multiply~$\be_s^{\bv}$ to~\eqref{eqn:FSI-err-v}$_3$ and integrate over $\Omega_s$, leading to
\begin{align*}
\frac{1}{2}\frac{\mathrm{d}}{\mathrm{d}t} \int_{\Omega_s}
\bm{\rho}_s |\be_s^{\bv}|^2 \mathrm{d} \bx & = \int_{\Omega_s}
(\nabla \cdot \bS_s(\be_s^{\bu})) \cdot \be_s^{\bv} \mathrm{d} \bx +
\int_{\Omega_s} \mathcal{R}_s^{\bv} \cdot \be_s^{\bv} \mathrm{d}\bx.
\end{align*}
Apply integration by parts, use the fact that $\be_s^{\bv} = \frac{\partial \be_s^{\bu}}{\partial t} - \mathcal{R}_s^{\bu}$, and sum up the fluid and the structural parts, yield
\begin{align*}
& \quad \frac{1}{2} \left[ \frac{\mathrm{d}}{\mathrm{d}t} \int_{\Omega_f} \bm{\rho}_f |\be_f^{\bv}|^2 \mathrm{d} \bx + \frac{\mathrm{d}}{\mathrm{d}t} \int_{\Omega_s} \bm{\rho}_s |\be_s^{\bv}|^2 \mathrm{d} \bx + \frac{\mathrm{d}}{\mathrm{d}t} \int_{\Omega_s} 2 \mu_s |\varepsilon(\be_{s}^{\bu})|^2 + \lambda_s |\nabla \cdot \be_{s}^{\bu}|^2 \mathrm{d}\bx \right] \\
& = - \int_{\Omega_f} \bm{\rho}_f \left( (\bv_f \cdot \nabla) \be_f^{\bv} \right) \cdot \be_f^{\bv}  \mathrm{d} \bx  - \int_{\Omega_f} \bm{\rho}_f ( (\be_f^{\bv} \cdot \nabla)  \mathcal{V}_{\mathbf{NN}}^{f,*} ) \cdot \be_f^{\bv} \mathrm{d} \bx  - 2\mu_f \int_{\Omega_f} | \bD(\be_f^{\bv}) |^2 \mathrm{d} \bx \\
& \quad + \int_{\Omega_f} \mathcal{R}_f^{\bv} \cdot \be_f^{\bv} \mathrm{d} \bx +  \int_{\Omega_f} \mathcal{R}^{\nabla \cdot}_{f} e_f^p \mathrm{d}\bx + \int_{\Omega_s} \mathcal{R}_s^{\bv} \cdot \be_s^{\bv} \mathrm{d}\bx
+ \int_{\Omega_s} 2\mu_s \varepsilon(\be_s^{\bu}) \varepsilon(\mathcal{R}_s^{\bu}) + \lambda_s (\nabla \cdot \be_s^{\bu}) (\nabla \cdot \mathcal{R}_s^{\bu}) \mathrm{d}\bx 
\\
& \quad + \int_{\partial \Omega_s \backslash \Gamma} \mathcal{R}_f^{\mathcal{B}} \cdot \bS(\be_f^{\bv}, e_f^p) \bn_f \mathrm{d} s + \int_{\partial \Omega_s \backslash \Gamma} \mathcal{R}_s^{\mathcal{B}} \cdot \bS(\be_s^{\bu})\bn_s \mathrm{d}s \\
& \quad + \int_{\Gamma} \mathcal{R}_{\Gamma}^{\bv} \left( \frac{
\bm{\sigma}_f(\be_f^{\bv}, e_f^p) \bn_f -
\bm{\sigma}_s(\be_s^{\bu})\bn_s }{2}  \right) \mathrm{d} s +
\int_{\Gamma} \mathcal{R}_{\Gamma}^{\bm{\sigma}} \left(
\frac{\be_f^{\bv} + \be_s^{\bv}}{2} \right) \mathrm{d}s.
\end{align*}
Following the same estimation method we carried out before, we obtain
\begin{align}
    & \qquad  \left[\frac{\mathrm{d}}{\mathrm{d}t} \int_{\Omega_f} \bm{\rho}_f | \be_f^{\bv} |^2 \mathrm{d} \bx + \frac{\mathrm{d}}{\mathrm{d}t} \int_{\Omega_s} \bm{\rho_s} | \be_s^{\bv}|^2 \mathrm{d} \bx +  \frac{\mathrm{d}}{\mathrm{d}t} \int_{\Omega_s} 2 \mu_s |\varepsilon(\be_{s}^{\bu})|^2 + \lambda_s |\nabla \cdot \be_{s}^{\bu}|^2 \mathrm{d}\bx \right]  \notag\\
    & \leq C_1 \left[ \int_{\Omega_f} |\mathcal{R}_f^{\bv}|^2 \mathrm{d} \bx + \int_{\Omega_f} |\mathcal{R}_f^{\nabla \cdot}|^2 \mathrm{d} \bx  + \int_{\Omega_s} |\mathcal{R}_s^{\bv} |^2 \mathrm{d} \bx + \int_{\Omega_s} 2\mu_s |\varepsilon(\mathcal{R}_{s}^{\bu})|^2 + \lambda_s |\nabla \cdot \mathcal{R}_{s}^{\bu}|^2 \mathrm{d}\bx \right.\notag\\
    & \qquad \ \  + \left. \int_{\partial \Omega_f\backslash \Gamma} |\mathcal{R}_f^{\mathcal{B}}|^2 \mathrm{d} s + \int_{\partial \Omega_s \backslash \Gamma} |\mathcal{R}_s^{\mathcal{B}} |^2 \mathrm{d}s +  \int_{\Gamma} |\mathcal{R}_{\Gamma}^{\bv}|^2 \mathrm{d} s +   \int_{\Gamma} |\mathcal{R}_{\Gamma}^{\bm{\sigma}}|^2 \mathrm{d} s   \right]^{\frac{1}{2}} \notag\\
    & \quad + C_2 \int_{\Omega_f} \bm{\rho}_f |\be_f^{\bv}|^2 \mathrm{d}
    \bx,\label{error_FSI2}
\end{align}
where $C_1$ and $C_2$ are positive constants depending on~$\| \bv_f
\|_{C^2(\overline{\Omega_f} \times [0,T])}$, $\| p_f
\|_{C^1(\overline{\Omega_f} \times [0,T])}$, $\|
\mathcal{V}_{\mathbf{NN}}^{f,*} \|_{C^2(\overline{\Omega_f} \times
[0,T])}$, $\| \mathcal{P}_{\mathbf{NN}}^{f,*}
\|_{C^1(\overline{\Omega_f} \times [0,T])}$, $\| \bu_s
\|_{C^2(\overline{\Omega_s} \times [0,T])}$, $\| \bv_s
\|_{C^1(\overline{\Omega_s} \times [0,T])}$, $\|
\mathcal{U}_{\mathbf{NN}}^{s,*} \|_{C^2(\overline{\Omega_s} \times
[0,T])}$, $\| \mathcal{V}_{\mathbf{NN}}^{s,*}
\|_{C^1(\overline{\Omega_s} \times [0,T])}$, $d$, $\bm{\rho}_f$,
$\bm{\rho}_s$, $\mu_f$, $\mu_s$, $\lambda_s$, $\Omega_f$, and
$\Omega_s$. For any $0\leq \tau \leq T$, we
integrate~\eqref{error_FSI2} over time to obtain
\begin{align*}
    & \qquad  \int_{\Omega_f} \bm{\rho}_f | \be_f^{\bv}(\bx_f,\tau) |^2 \mathrm{d} \bx +  \int_{\Omega_s} \bm{\rho_s} \left|\frac{\partial \be_s^{\bu}}{\partial t}(\bx_s,\tau)\right|^2 \mathrm{d} \bx + \int_{\Omega_s} 2 \mu_s |\varepsilon(\be_{s}^{\bu}(\bx_s,\tau))|^2 + \lambda_s |\nabla \cdot \be_{s}^{\bu}(\bx_s,\tau)|^2 \mathrm{d}\bx   \\
    & \leq \int_{\Omega_f} \bm{\rho}_f |\mathcal{R}_f^{\mathcal{I}}|^2 \mathrm{d} \bx + \int_{\Omega_s} \bm{\rho}_s |\mathcal{R}_s^{\mathcal{I},\bv}|^2 \mathrm{d}\bx + \int_{\Omega_s} 2 \mu_s |\varepsilon(\mathcal{R}_{s}^{\mathcal{I},\bu})|^2 + \lambda_s |\nabla \cdot \mathcal{R}_s^{\mathcal{I},\bu}|^2 \mathrm{d} \bx +  \int_{\Omega_s} \left| \mathcal{R}_{s}^{\mathcal{I},\bu} \right|^2 \mathrm{d} \bx   \\
    & \quad + C_1 T^{\frac{1}{2}} \left[ \int_0^T \left( \int_{\Omega_f} |\mathcal{R}_f^{\bv}|^2 \mathrm{d} \bx + \int_{\Omega_f} |\mathcal{R}_f^{\nabla \cdot}|^2 \mathrm{d} \bx  + \int_{\Omega_s} |\mathcal{R}_s^{\bv} |^2 \mathrm{d} \bx  +  \int_{\Omega_s} 2\mu_s |\varepsilon(\mathcal{R}_{s}^{\bu})|^2 + \lambda_s |\nabla \cdot \mathcal{R}_{s}^{\bu}|^2 \mathrm{d}\bx    \right. \right. \\
    & \qquad \qquad \  \left. \left. +  \int_{\Omega_s} \left| \mathcal{R}_s^{\bu} \right|^2 \mathrm{d} \bx  + \int_{\partial \Omega_f\backslash \Gamma} |\mathcal{R}_f^{\mathcal{B}}|^2 \mathrm{d} s +  \int_{\partial \Omega_s \backslash \Gamma} | \mathcal{R}_s^{\mathcal{B}} |^2 \mathrm{d}s +  \int_{\Gamma} |\mathcal{R}_{\Gamma}^{\bv}|^2 \mathrm{d} s +   \int_{\Gamma} |\mathcal{R}_{\Gamma}^{\bm{\sigma}}|^2 \mathrm{d} s \right) \mathrm{d} t  \right]^{\frac{1}{2}} \\
    & \quad + C_2 \int_0^\tau \left( \int_{\Omega_f} \bm{\rho}_f |\be_f^{\bv}(\bx_f, t)|^2 \mathrm{d} \bx \right)
    \mathrm{d}t.
\end{align*}
On the right-hand side of the above inequality, we add the fourth and the ninth integral terms, i.e., $\int_{\Omega_s} \left| \mathcal{R}_{s}^{\mathcal{I},\bu} \right|^2 \mathrm{d} \bx$ and $\int_{\Omega_s} \left| \mathcal{R}_s^{\bu} \right|^2 \mathrm{d} \bx$, as two extra terms due to the same reason as shown for the case of FSI model with a wave-type structural equation. On the other hand, these two terms also directly come from the LS formulation of the parabolic-like structural equation defined in Section~\ref{sec:FSI-example}. Without those two terms, only the derivatives of $\mathcal{R}_s^{\mathcal{I}, \bu}$ and $\mathcal{R}_s^{\bu}$ (i.e., the third and the eighth integral terms on the right-hand side) are involved, which can only be determined up to a constant. Adding the fourth and the ninth integral terms helps to determine the correct constants. Applying the Gr\"{o}nwall's inequality, integrating again over~$[0,T]$, and employing Assumption~\eqref{assumption:quad} on the residuals, we can similarly derive the following generalization error bound for the FSI model with a parabolic-like structural equation.
\begin{thm}\label{thm:error-FSI-parabolic}
Let~$\bv_f \in C^2(\overline{\Omega_f} \times [0,T])$, $p_f \in
C^1(\overline{\Omega_f} \times [0,T])$, $\bu_s \in
C^2(\overline{\Omega_s} \times [0,T])$, and $\bv_s \in
C^1(\overline{\Omega_s} \times [0,T])$ be the classical solutions of the FSI problem with a parabolic-like structural equation, correspondingly, let $\mathcal{V}_{\mathbf{NN}}^{f,*} \in
C^2(\overline{\Omega_f} \times [0,T])$,
$\mathcal{P}_{\mathbf{NN}}^{f,*} \in C^1(\overline{\Omega_f} \times
[0,T])$, $ \mathcal{U}_{\mathbf{NN}}^{s,*} \in
C^2(\overline{\Omega_s} \times [0,T])$, $
\mathcal{V}_{\mathbf{NN}}^{s,*} \in C^1(\overline{\Omega_s} \times
[0,T])$ be the numerical approximations obtained by the DNN/meshfree
method. We have the following generalization error estimation,
\begin{align*}
        & \quad \int_{0}^T \left( \int_{\Omega_f} \bm{\rho}_f | \be_f^{\bv}(\bx_f, t) |^2 \mathrm{d} \bx +  \int_{\Omega_s} \bm{\rho_s} |\be_s^{\bv}(\bx_s, t)|^2 \mathrm{d} \bx + \int_{\Omega_s} 2 \mu_s |\varepsilon(\be_{s}^{\bu}(\bx_s,t))|^2 + \lambda_s |\nabla \cdot \be_{s}^{\bu}(\bx_s,t)|^2 \mathrm{d}\bx \right) \mathrm{d}t     \\
        & \leq C_1T(1 + C_2T e^{C_2T}) \left[  \mathcal{F}_{\mathcal{I}_f}(\bm{\Theta}^*) +  \mathcal{F}_{\mathcal{I}_{s}}(\bm{\Theta}^*) +  C_{quad}^{\mathcal{I}_f} M_{\mathcal{I}_f}^{-\alpha_{\mathcal{I}_f}}  + C_{quad}^{\mathcal{I}_{s}} M_{\mathcal{I}_{s}}^{-\alpha_{\mathcal{I}_{s}}}
         \right] \\
        & \quad + C_1 T^{\frac{3}{2}}(1 + C_2T e^{C_2T}) \left[ \mathcal{F}_{\mathcal{L}_f}(\bm{\Theta}^*) + \mathcal{F}_{\mathcal{L}_{s}}(\bm{\Theta}^*) +
        \mathcal{F}_{\mathcal{B}_f}(\bm{\Theta}^*) + \mathcal{F}_{\mathcal{B}_s}(\bm{\Theta}^*) +
        \mathcal{F}_{\Gamma}(\bm{\Theta}^*)
        \right. \\
        & \qquad \qquad \qquad \qquad \quad \ \qquad \left. + C_{quad}^{\mathcal{F}_f} M_{\mathcal{F}_f}^{-\alpha_{\mathcal{F}_f}}  + C_{quad}^{\mathcal{F}_{s, \bv}} M_{\mathcal{F}_{s, \bv}}^{-\alpha_{\mathcal{F}_{s, \bv}}} +
        C_{quad}^{\mathcal{F}_{s, \bu}} M_{\mathcal{F}_{s, \bu}}^{-\alpha_{\mathcal{F}_{s, \bu}}}  \right. \\
        & \qquad \qquad \qquad \qquad \quad \ \qquad \left. + C_{quad}^{\mathcal{B}_f} M_{\mathcal{B}_f}^{-\alpha_{\mathcal{B}_f}}  + C_{quad}^{\mathcal{B}_s} M_{\mathcal{B}_s}^{-\alpha_{\mathcal{B}_s}} + C_{quad}^{\Gamma} M_{\Gamma}^{-\alpha_{\Gamma}}
        \right]^{\frac{1}{2}},
    \end{align*}
where $C_1$ and $C_2$ are positive constants depending on $\| \bv_f
\|_{C^2(\overline{\Omega_f} \times [0,T])}$, $\| p_f
\|_{C^1(\overline{\Omega_f} \times [0,T])}$, $\| \bu_s
\|_{C^2(\overline{\Omega_s} \times [0,T])}$, $\| \bv_s
\|_{C^1(\overline{\Omega_s} \times [0,T])}$, $\|
\mathcal{V}_{\mathbf{NN}}^{f,*} \|_{C^2(\overline{\Omega_f} \times
[0,T])}$, $\| \mathcal{P}_{\mathbf{NN}}^{f,*}
\|_{C^1(\overline{\Omega_f} \times [0,T])}$, $\|
\mathcal{U}_{\mathbf{NN}}^{s,*} \|_{C^2(\overline{\Omega_s} \times
[0,T])}$, $\| \mathcal{V}_{\mathbf{NN}}^{s,*}
\|_{C^1(\overline{\Omega_s} \times [0,T])}$, $\bm{\rho}_f$,
$\bm{\rho}_s$, $\mu_f$, $\mu_s$, $\lambda_s$, $\Omega_f$,
$\Omega_s$, and $d$.
\end{thm}

\begin{remark}
Similarly, from the derivation above, one can see that two extra terms, i.e., the third and the eighth integral term on the right-hand side of \eqref{LS4FSI1} that we added to our DNN/meshfree method for solving the FSI problem with a parabolic-like structural equation also come from the time integration and integration by parts. We also use those two terms in our implementation to verify the theoretical results we developed to solve the FSI problem with a parabolic-like structural equation.
\end{remark}

\section{Numerical experiments} \label{sec:numerics}
In this section, we conduct numerical experiments for the two-phase flow interface problem and the FSI problem in two forms. The purpose is to demonstrate the effectiveness and the capacity of the proposed DNN/meshfree method and illustrate the theoretical convergence results obtained in the previous section.

All three numerical examples are implemented using DeepXDE~\cite{lu2021deepxde} which is developed based on the methodology of PINN~\cite{karniadakis2021physics} and the platform of TensorFlow~\cite{tensorflow2015-whitepaper} and PyTorch~\cite{NEURIPS2019_9015}. The fully connected DNN structures used in our numerical experiments, otherwise specified, have~$3$ hidden layers and~$50$ neurons in each hidden layer, and~$\tanh$ function is used as the activation function on each layer. As mentioned before, ADAM, a variant of the SGD method, is applied as the optimizer with an initial learning rate~$0.001$ and~$50000$ epochs performed to obtain numerical solutions of the developed DNN/meshfree method. The way to choose prescribed weight coefficients~$\omega_{\mathcal{L}_i}$, $\omega_{\Gamma}$, $\omega_{\mathcal{B}_i}$, and~$\omega_{\mathcal{I}_i}$ in the space-time LS formulation $\mathcal{R}(\bm{\tilde{u}}_1, \bm{\tilde{u}}_2)$ is referred to Remark \ref{rmk:weights}.

\subsection{Example 1: The two-phase flow interface problem}\label{num:example1}
We first consider the two-phase flow interface problem, i.e., the example introduced in Section~\ref{sec:NS-NS-example}.  Here, we choose $\Omega := [0,3] \times [0,3]$, $\Omega_2 := \{(x,y) \in \mathbb{R}^2 \ | \ (x-1.5)^2 + (y-1.5)^2 < 1 \}$ and $\Omega_1=\Omega\backslash\overline{\Omega_2}$ as the computational domain, and $[0,1]$ as the time interval, which implies that the interface $\Gamma$ is just a circle defined as $\Gamma:=\{(x,y) \in \mathbb{R}^2 \ | \ (x-1.5)^2 + (y-1.5)^2 = 1 \}$. We appropriately choose $\bm{f}_i,\ \bg_i,\ \bv_i^b,\ \bv_i^0$, $i=1,2$, to make sure the following function $(\bm{v}_i,\ p_i)$, $i=1,2$, are the exact solution to this example:
\begin{align*}
&\bm{v}_1 =
\begin{pmatrix}
e^t \, \sin x \, \cos y \\
e^t \, \cos x \, \sin y
\end{pmatrix}, \ \quad\quad
p_1 =   e^t \, \sin x \, \sin y, \\
&\bm{v}_2 =
\begin{pmatrix}
    \cos t \, \cos x \, \cos y \\
    \cos t \, \sin x \, \sin y
\end{pmatrix}, \quad
p_2 =   \cos t \, \cos(x+y),
\end{align*}
which induces $\bg_i\neq 0,\ i=1,2$, leading to the jump-type interface conditions across $\Gamma$. Furthermore, in our numerical experiments, we choose different densities~$\bm{\rho}_i\ (i=1,2)$ and dynamic viscosities $\bm{\mu}_i\ (i=1,2)$ for the fluid, which presents two different fluid phases in this example to investigate numerical performances of the developed DNN/meshfree method for a two-phase flow interface problem with high-contrast coefficients and jump-type interface conditions.

\begin{table}[h!]
\centering \caption{The two-phase flow interface problem:
$\rho_1=\rho_2 = 1$ and $\mu_1 = \mu_2 = 1$} \label{tab:NSNS-1vs1}
\begin{tabular}{ c c c c c c}
\hline \hline
$M_{\mathcal{L}_i}$ & $M_{\mathcal{B}_i}$ & $M_{\Gamma}$&  $M_{\mathcal{I}_i}$  &  Approx. Error & Loss Error \\ \hline
$10 \times 10 \times 5$ & $4 \times 4 \times 5$ & $4 \times 5$  & $4 \times 4$  & 7.276e-02  & 9.025e-06 \\
$10 \times 10 \times 5$ & $8 \times 4 \times 5$ & $8 \times 5$  & $8 \times 8$  & 5.875e-02  & 1.425e-05 \\
$10 \times 10 \times 5$ & $16 \times 4 \times 5$ & $16 \times 5$  & $16 \times 16$  & 3.673e-02  & 1.115e-05 \\
$10 \times 10 \times 5$ & $32 \times 4 \times 5$ & $32 \times 5$  & $32 \times 32$  & 5.577e-02  & 1.316e-05 \\ \hline
$20 \times 20 \times 10$ & $4 \times 4 \times 10$ & $4 \times 10$  & $4 \times 4$  & 2.766e-02  & 1.239e-05 \\
$20 \times 20 \times 10$ & $8 \times 4 \times 10$ & $8 \times 10$  & $8 \times 8$  & 2.524e-02  & 1.602e-05 \\
$20 \times 20 \times 10$ & $16 \times 4 \times 10$ & $16 \times 10$  & $16 \times 16$  & 4.358e-02  & 1.361e-05 \\
$20 \times 20 \times 10$ & $32 \times 4 \times 10$ & $32 \times 10$  & $32 \times 32$  & 3.151e-02  & 1.493e-05 \\ \hline
$40 \times 40 \times 20$ & $4 \times 4 \times 20$ & $4 \times 20$  & $4 \times 4$  & 3.384e-02  & 1.160e-05 \\
$40 \times 40 \times 20$ & $8 \times 4 \times 20$ & $8 \times 20$  & $8 \times 8$  & 1.574e-02  & 1.358e-05 \\
$40 \times 40 \times 20$ & $16 \times 4 \times 20$ & $16 \times 20$  & $16 \times 16$  & 2.514e-02  & 1.507e-05 \\
$40 \times 40 \times 20$ & $32 \times 4 \times 20$ & $32 \times 20$  & $32 \times 32$  & 2.355e-02  & 1.138e-05 \\
\hline \hline
\end{tabular}
\end{table}

\begin{figure}[h!]
\centering	
\includegraphics[width=.7\textwidth]{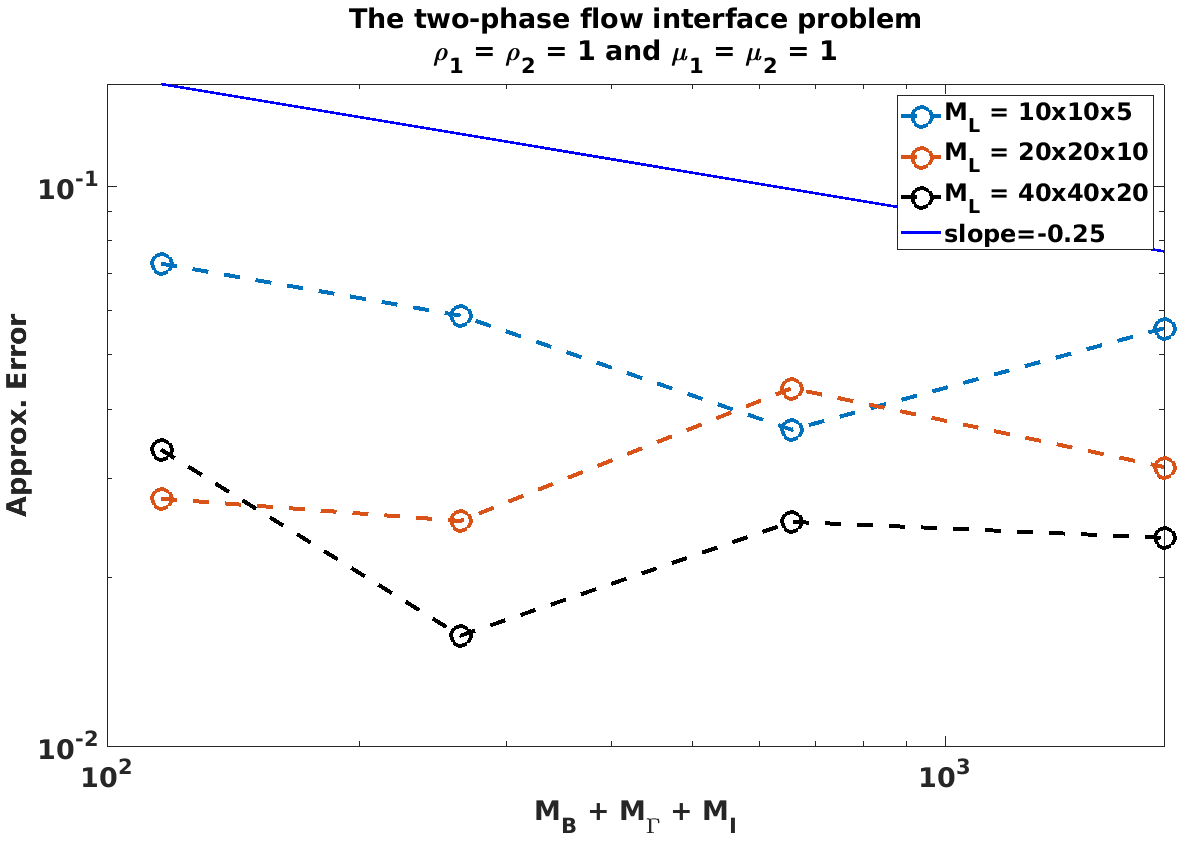}
\caption{Approximation errors of the two-phase flow interface problem
	($\rho_1=\rho_2 = 1$ and $\mu_1 = \mu_2 = 1$).} \label{fig:table1}
\end{figure}

\begin{figure}[h!]
    \centering
\includegraphics[width=.45\textwidth]{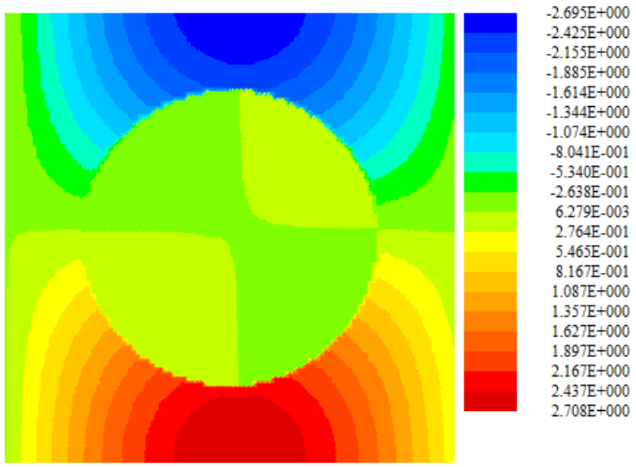}
\includegraphics[width=.45\textwidth]{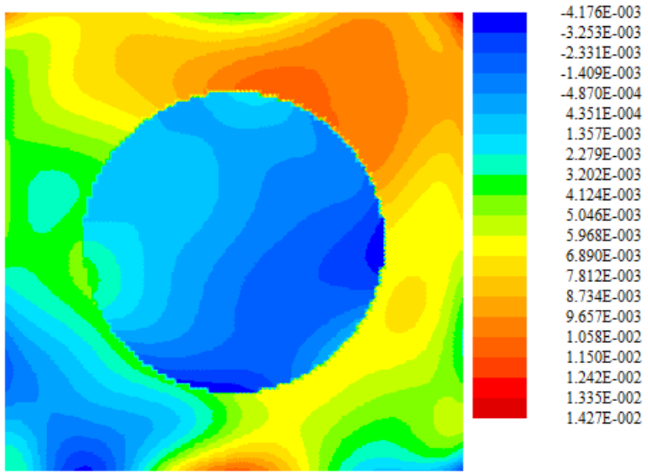} \\
\includegraphics[width=.45\textwidth]{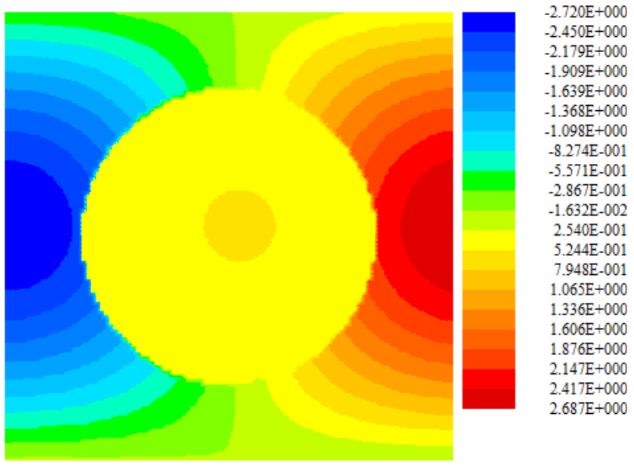}
\includegraphics[width=.45\textwidth]{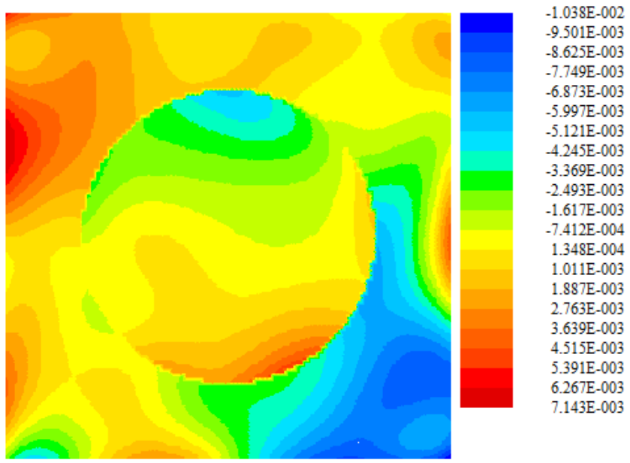} \\
\includegraphics[width=.45\textwidth]{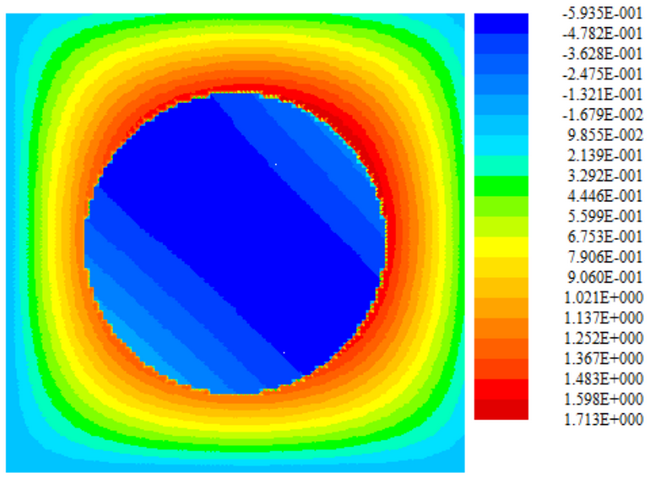}
\includegraphics[width=.45\textwidth]{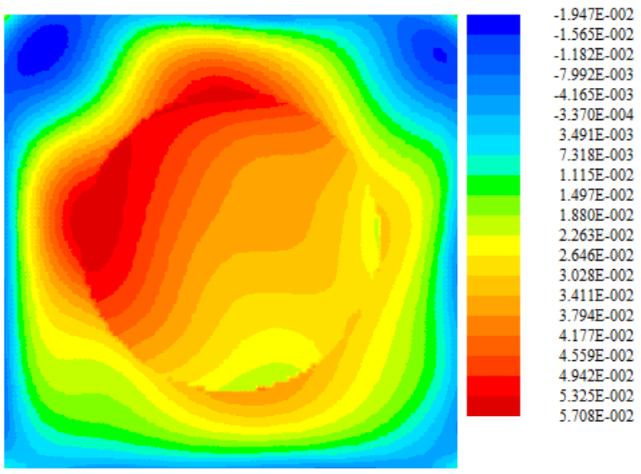}
\caption{Numerical solutions of the two-phase flow interface problem
($\rho_1=\rho_2 = 1$ and $\mu_1 = \mu_2 = 1$). Top left: the
horizontal velocity; Top right: error of the horizontal velocity;
Middle left: the vertical velocity;  Middle right: error of the
vertical velocity; Bottom left: the pressure; Bottom right: error of
the pressure.} \label{fig:NSNS-1vs1}
\end{figure}

Numerical error results are presented in Table~\ref{tab:NSNS-1vs1} where we set the physical parameters to be $\rho_1=\rho_2 = 1$ and~$\mu_1 = \mu_2 = 1$.  From the last two columns, we can see that errors of the loss function are much smaller than numerical approximation errors, which means we are at the region where the optimization problem is solved fairly accurately, and the quadrature error dominates the error. In this case, our error estimation derived in Section~\ref{sec:error} mainly serves as a prior error estimation. Furthermore, we expect the quadrature errors coming from the interface, boundary, and initial condition terms to play a more important role than the quadrature errors of the interior space-time domain terms because dimensions of the interface, the boundary, and the initial subdomains are lower than that of the interior space-time subdomains. In our experiments, as shown in Table~\ref{tab:NSNS-1vs1} and Figure~\ref{fig:table1}, when we keep the number of interior sampling point $M_{\mathcal{L}_i}$ to be a constant and increase the number of sampling points on the interface, $M_{\Gamma}$, on the boundaries, $M_{\mathcal{B}_i}$, and in the initial subdomains, $M_{\mathcal{I}_i}$, then the overall approximation errors decrease first as expected but then stagnate or even increase, which may indicate that the quadrature errors of the interior domain terms take over the dominance again when the number of sampling points on the interface/boundaries and in the initial subdomains are more than enough. In Figure~\ref{fig:table1}, we also plot a reference line with the absolute value of the slope equaling $0.25$, which should be the expected convergence order since using sampling points is equivalent to using a Monte-Carlo type numerical quadrature rule. When approximation errors decrease, we notice that they roughly follow the expected convergence order. Those observations are consistent with our theoretical result.

\begin{table}[h!]
    \centering
    \caption{The two-phase flow interface problem: $\rho_1=1$, $\rho_2 = 10$, $\mu_1 = 1$, and $\mu_2 = 10$} \label{tab:NSNS-1vs10}
    \begin{tabular}{ c c c c c c}
        \hline \hline
        $M_{\mathcal{L}_i}$ & $M_{\mathcal{B}_i}$ & $M_{\Gamma}$&  $M_{\mathcal{I}_i}$  &  Approx. Error & Loss Error \\ \hline
        $10 \times 10 \times 5$ & $4 \times 4 \times 5$ & $4 \times 5$  & $4 \times 4$  & 1.577e-01  & 1.312e-05 \\
        $10 \times 10 \times 5$ & $8 \times 4 \times 5$ & $8 \times 5$  & $8 \times 8$  & 7.793e-02  & 1.399e-05 \\
        $10 \times 10 \times 5$ & $16 \times 4 \times 5$ & $16 \times 5$  & $16 \times 16$  & 6.548e-02  & 1.166e-05 \\
        $10 \times 10 \times 5$ & $32 \times 4 \times 5$ & $32 \times 5$  & $32 \times 32$  & 8.312e-02  & 1.387e-05 \\ \hline
        $20 \times 20 \times 10$ & $4 \times 4 \times 10$ & $4 \times 10$  & $4 \times 4$  & 1.181e-01  & 1.120e-05 \\
        $20 \times 20 \times 10$ & $8 \times 4 \times 10$ & $8 \times 10$  & $8 \times 8$  & 7.081e-02  & 1.422e-05 \\
        $20 \times 20 \times 10$ & $16 \times 4 \times 10$ & $16 \times 10$  & $16 \times 16$  & 5.197e-02  & 1.065e-05 \\
        $20 \times 20 \times 10$ & $32 \times 4 \times 10$ & $32 \times 10$  & $32 \times 32$  & 7.916e-02  & 1.919e-05 \\ \hline
        $40 \times 40 \times 20$ & $4 \times 4 \times 20$ & $4 \times 20$  & $4 \times 4$  & 8.072e-02  & 1.161e-05 \\
        $40 \times 40 \times 20$ & $8 \times 4 \times 20$ & $8 \times 20$  & $8 \times 8$  & 4.400e-02  & 1.301e-05 \\
        $40 \times 40 \times 20$ & $16 \times 4 \times 20$ & $16 \times 20$  & $16 \times 16$  & 4.776e-02  & 1.199e-05 \\
        $40 \times 40 \times 20$ & $32 \times 4 \times 20$ & $32 \times 20$  & $32 \times 32$  & 4.426e-02  & 1.231e-05 \\
        \hline \hline
    \end{tabular}
\end{table}

\begin{table}[h!]
    \centering
    \caption{The two-phase flow interface problem: $\rho_1=1$, $\rho_2 = 100$, $\mu_1 = 1$, and $\mu_2 = 100$} \label{tab:NSNS-1vs100}
    \begin{tabular}{ c c c c c c}
        \hline \hline
        $M_{\mathcal{L}_i}$ & $M_{\mathcal{B}_i}$ & $M_{\Gamma}$&  $M_{\mathcal{I}_i}$  &  Approx. Error & Loss Error \\ \hline
        $10 \times 10 \times 5$ & $4 \times 4 \times 5$ & $4 \times 5$  & $4 \times 4$  & 4.662e-01  & 9.318e-05 \\
        $10 \times 10 \times 5$ & $8 \times 4 \times 5$ & $8 \times 5$  & $8 \times 8$  & 1.599e-02  & 1.131e-05 \\
        $10 \times 10 \times 5$ & $16 \times 4 \times 5$ & $16 \times 5$  & $16 \times 16$  & 2.971e-02  & 1.274e-05 \\
        $10 \times 10 \times 5$ & $32 \times 4 \times 5$ & $32 \times 5$  & $32 \times 32$  & 2.082e-02  & 1.350e-05 \\ \hline
        $20 \times 20 \times 10$ & $4 \times 4 \times 10$ & $4 \times 10$  & $4 \times 4$  & 3.704e-01  & 1.211e-05 \\
        $20 \times 20 \times 10$ & $8 \times 4 \times 10$ & $8 \times 10$  & $8 \times 8$  & 3.651e-02  & 1.347e-05 \\
        $20 \times 20 \times 10$ & $16 \times 4 \times 10$ & $16 \times 10$  & $16 \times 16$  & 8.521e-02  & 1.615e-05 \\
        $20 \times 20 \times 10$ & $32 \times 4 \times 10$ & $32 \times 10$  & $32 \times 32$  & 2.236e-01  & 1.457e-05 \\ \hline
        $40 \times 40 \times 20$ & $4 \times 4 \times 20$ & $4 \times 20$  & $4 \times 4$  & 3.933e-01  & 1.392e-05 \\
        $40 \times 40 \times 20$ & $8 \times 4 \times 20$ & $8 \times 20$  & $8 \times 8$  & 3.455e-01  & 1.466e-05 \\
        $40 \times 40 \times 20$ & $16 \times 4 \times 20$ & $16 \times 20$  & $16 \times 16$  & 1.537e-01  & 1.437e-05 \\
        $40 \times 40 \times 20$ & $32 \times 4 \times 20$ & $32 \times 20$  & $32 \times 32$  & 1.010e-02  & 1.711e-05 \\
        \hline \hline
    \end{tabular}
\end{table}

\begin{table}[h!]
    \centering
    \caption{The two-phase flow interface problem: $\rho_1=1$, $\rho_2 = 1000$, $\mu_1 = 1$, and $\mu_2 = 1000$} \label{tab:NSNS-1vs1000}
    \begin{tabular}{ c c c c c c}
        \hline \hline
        $M_{\mathcal{L}_i}$ & $M_{\mathcal{B}_i}$ & $M_{\Gamma}$&  $M_{\mathcal{I}_i}$  &  Approx. Error & Loss Error \\ \hline
        $10 \times 10 \times 5$ & $4 \times 4 \times 5$ & $4 \times 5$  & $4 \times 4$  & 3.606e00  & 1.068e-05 \\
        $10 \times 10 \times 5$ & $8 \times 4 \times 5$ & $8 \times 5$  & $8 \times 8$  & 1.445e00  & 1.261e-05 \\
        $10 \times 10 \times 5$ & $16 \times 4 \times 5$ & $16 \times 5$  & $16 \times 16$  & 1.525e00  & 1.125e-05 \\
        $10 \times 10 \times 5$ & $32 \times 4 \times 5$ & $32 \times 5$  & $32 \times 32$  & 1.248e00 & 1.357e-05 \\ \hline
        $20 \times 20 \times 10$ & $4 \times 4 \times 10$ & $4 \times 10$  & $4 \times 4$  & 3.461e00  & 1.250e-05 \\
        $20 \times 20 \times 10$ & $8 \times 4 \times 10$ & $8 \times 10$  & $8 \times 8$  & 1.327e00  & 1.338e-05 \\
        $20 \times 20 \times 10$ & $16 \times 4 \times 10$ & $16 \times 10$  & $16 \times 16$  & 6.171e-01  & 1.338e-05 \\
        $20 \times 20 \times 10$ & $32 \times 4 \times 10$ & $32 \times 10$  & $32 \times 32$  & 9.599e-01  & 1.434e-05 \\ \hline
        $40 \times 40 \times 20$ & $4 \times 4 \times 20$ & $4 \times 20$  & $4 \times 4$  & 1.976e00  & 1.038e-05 \\
        $40 \times 40 \times 20$ & $8 \times 4 \times 20$ & $8 \times 20$  & $8 \times 8$  & 1.615e00  & 1.349e-05 \\
        $40 \times 40 \times 20$ & $16 \times 4 \times 20$ & $16 \times 20$  & $16 \times 16$  & 8.797e-01  & 1.631e-05 \\
        $40 \times 40 \times 20$ & $32 \times 4 \times 20$ & $32 \times 20$  & $32 \times 32$  & 2.320e-01  & 1.409e-05 \\
        \hline \hline
    \end{tabular}
\end{table}

\begin{figure}[h!]
    \centering
    \includegraphics[width=.45\textwidth]{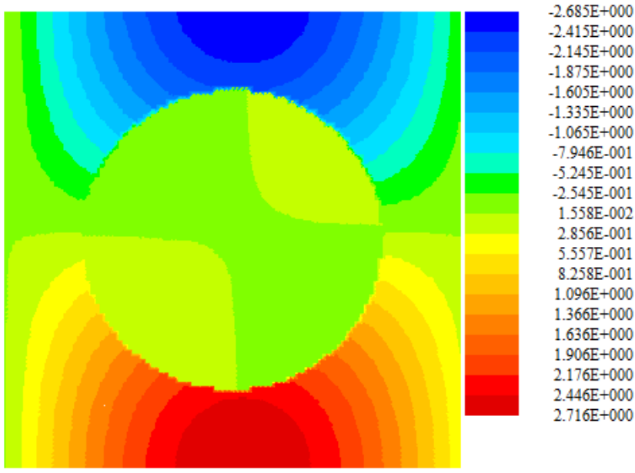}
    \includegraphics[width=.45\textwidth]{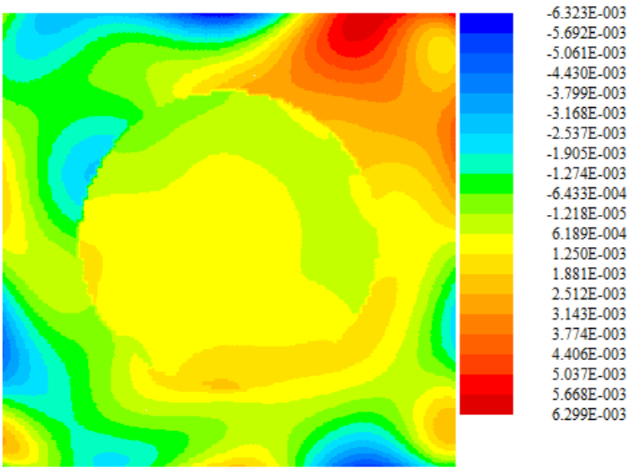} \\
    \includegraphics[width=.45\textwidth]{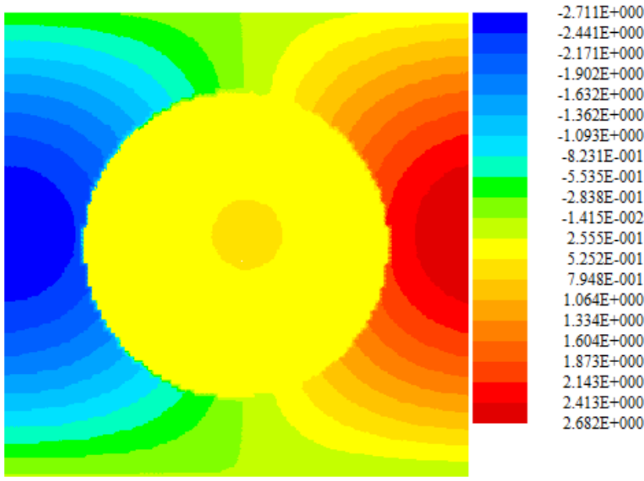}
    \includegraphics[width=.45\textwidth]{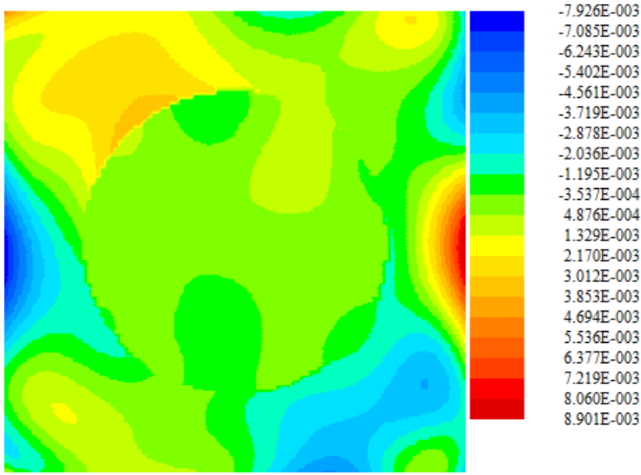} \\
    \includegraphics[width=.45\textwidth]{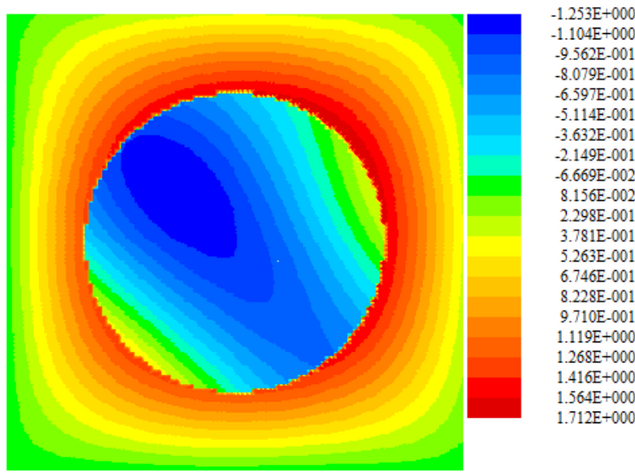}
    \includegraphics[width=.45\textwidth]{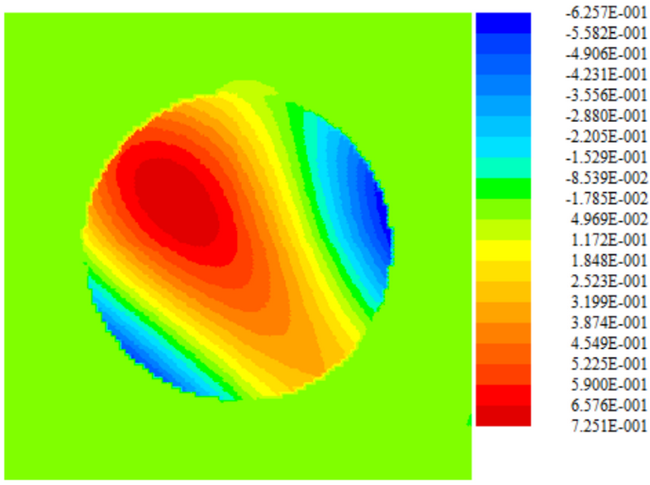}
\caption{Numerical solutions of the two-phase flow interface problem
($\rho_1=1$, $\rho_2 = 1000$, $\mu_1 = 1$, and $\mu_2 = 1000$). Top
left: the horizontal velocity; Top right: error of the horizontal
velocity; Middle left: the vertical velocity;  Middle right: error
of the vertical velocity; Bottom left: the pressure; Bottom right:
error of the pressure.} \label{fig:NSNS-1vs1000}
\end{figure}

In Tables~\ref{tab:NSNS-1vs10}, \ref{tab:NSNS-1vs100}, and~\ref{tab:NSNS-1vs1000}, we report numerical results along with increasing jumps of $\rho_1$ and $\rho_2$, and of $\mu_1$ and $\mu_2$, i.e., we gradually increase the value of physical parameters $\rho_2 = 10, 100, 1000$ and $\mu_2 = 10, 100, 1000$, respectively, while keeping~$\rho_1=\mu_1=1$, to represent a series of high-contrast coefficients. The last column of those tables shows that we solve the optimization fairly accurately, so we can ignore the error caused by the optimization problem. We still see a similar convergence behavior comparing with the case $\rho_1=\rho_2=1$ and $\mu_1=\mu_2 =1$, i.e., the approximation error decreases first when we increase $M_{\mathcal{B}_i}$, $M_{\Gamma}$, and~$M_{\mathcal{I}_i}$ while keeping $M_{\mathcal{L}_i}$ the same, and then it might stagnate or even increase, which implies that the quadrature error of the interior domain terms becomes dominating again when $M_{\mathcal{B}_i}$, $M_{\Gamma}$, and~$M_{\mathcal{I}_i}$ are large enough. In addition, we also observe that the approximation errors get worse when the ratios~$\rho_2/\rho_1$ and $\mu_2/\mu_1$ get larger. This is also expected since, with large jumps, the solution's regularity gets worse, making the interface problem more difficult to be numerically solved accurately. In the recent work~\cite{krishnapriyan2021characterizing}, how the regularity of the solutions affects the optimization problem has been studied for using PINN to solve PDEs. When the regularity worsens, the optimization problem becomes ill-conditioned, affecting the overall performance of the PINN method. As we can see from Figures~\ref{fig:NSNS-1vs1} and~\ref{fig:NSNS-1vs1000}, approximation errors of the velocity mainly stay the same when we increase the ratios~$\rho_2/\rho_1$ and $\mu_2/\mu_1$.  But the pressure's approximation error inside the circle ($\Omega_2$) deteriorates significantly, likely due to a worse approximation of the dynamic interface condition when the jump coefficients turn out higher. In the traditional finite element or finite volume methods, weak formulations are usually formed for the studied interface problems, where the Neumann-type dynamic interface condition is handled naturally through integration by parts. In other words, it can be canceled out, and thus no more dynamic interface condition exists in the weak formulation of the interface problem. Therefore, a better way to treat the dynamic interface condition using DNN/meshfree method may be provided by the weak formulation of the two-phase flow interface problem, especially for the case of large jumps, where there is no dynamic interface condition to be dealt with by the DNN's approximation. Hence, one promising approach to improving the pressure's approximation is to develop a new DNN/meshfree method based on the weak formulation of the two-phase flow interface problem, which is our ongoing work now. But due to the page limitation here, we will report a detailed investigation of this new approach in our future publications. Instead, we adopt a simpler way to improve the pressure's approximation for the developed DNN/meshfree approach in this paper, as shown below.

\begin{table}[h!]
    \centering
    \caption{The two-phase flow interface problem: $\rho_1=1$, $\rho_2 = 1000$, $\mu_1 = 1$, and $\mu_2 = 1000$ with $5$ observation points} \label{tab:NSNS-1vs1000-5p}
    \begin{tabular}{ c c c c c c}
        \hline \hline
        $M_{\mathcal{L}_i}$ & $M_{\mathcal{B}_i}$ & $M_{\Gamma}$&  $M_{\mathcal{I}_i}$  &  Approx. Error & Loss Error \\ \hline
        $10 \times 10 \times 5$ & $4 \times 4 \times 5$ & $4 \times 5$  & $4 \times 4$  & 1.151e-01  & 3.443e-05 \\
        $10 \times 10 \times 5$ & $8 \times 4 \times 5$ & $8 \times 5$  & $8 \times 8$  & 8.764e-02  & 5.086e-05 \\
        $10 \times 10 \times 5$ & $16 \times 4 \times 5$ & $16 \times 5$  & $16 \times 16$  & 1.154e-01  & 5.135e-05 \\
        $10 \times 10 \times 5$ & $32 \times 4 \times 5$ & $32 \times 5$  & $32 \times 32$  & 1.538e-01 & 5.111e-05 \\ \hline
        $20 \times 20 \times 10$ & $4 \times 4 \times 10$ & $4 \times 10$  & $4 \times 4$  & 7.920e-02  & 4.550e-05 \\
        $20 \times 20 \times 10$ & $8 \times 4 \times 10$ & $8 \times 10$  & $8 \times 8$  & 6.730e-02  & 5.554e-05 \\
        $20 \times 20 \times 10$ & $16 \times 4 \times 10$ & $16 \times 10$  & $16 \times 16$  & 7.411e-01  & 5.660e-05 \\
        $20 \times 20 \times 10$ & $32 \times 4 \times 10$ & $32 \times 10$  & $32 \times 32$  & 1.279e-01  & 6.769e-05 \\ \hline
        $40 \times 40 \times 20$ & $4 \times 4 \times 20$ & $4 \times 20$  & $4 \times 4$  & 7.871e-02  & 5.016e-05 \\
        $40 \times 40 \times 20$ & $8 \times 4 \times 20$ & $8 \times 20$  & $8 \times 8$  & 6.549e00  & 6.072e-05 \\
        $40 \times 40 \times 20$ & $16 \times 4 \times 20$ & $16 \times 20$  & $16 \times 16$  & 1.225e-01  & 4.495e-05 \\
        $40 \times 40 \times 20$ & $32 \times 4 \times 20$ & $32 \times 20$  & $32 \times 32$  & 8.506e-02  & 3.833e-05 \\
        \hline \hline
    \end{tabular}
\end{table}

\begin{figure}[h!]
    \centering
    \includegraphics[width=.45\textwidth]{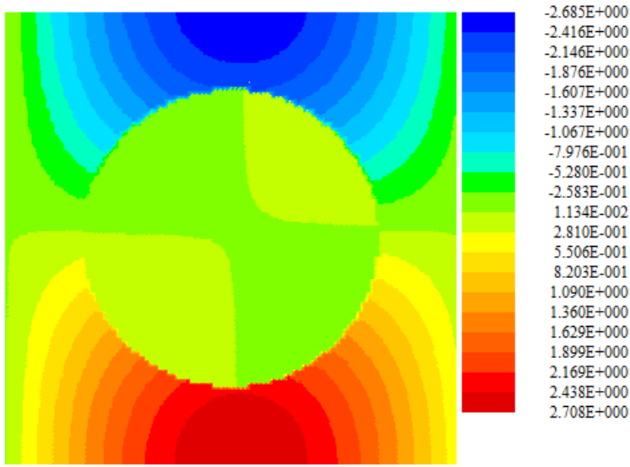}
    \includegraphics[width=.45\textwidth]{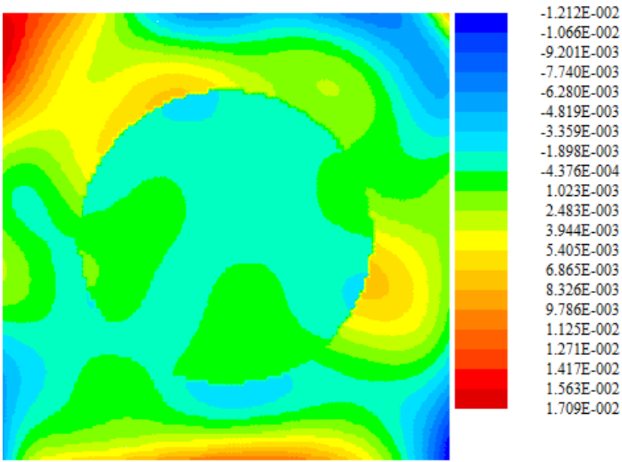} \\
    \includegraphics[width=.45\textwidth]{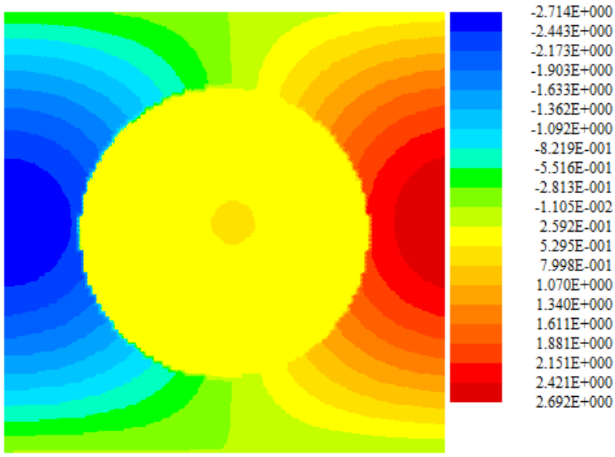}
    \includegraphics[width=.45\textwidth]{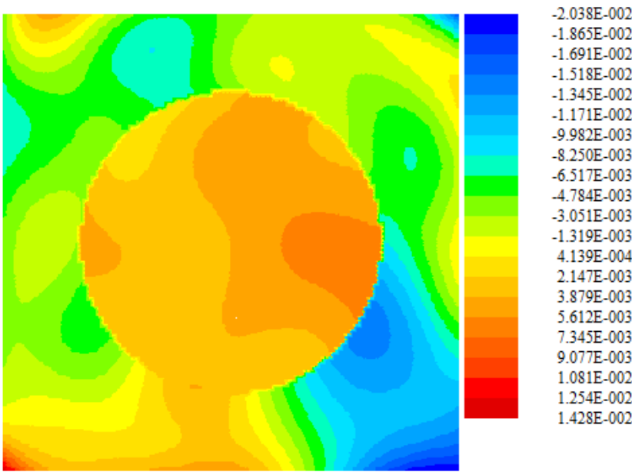} \\
    \includegraphics[width=.45\textwidth]{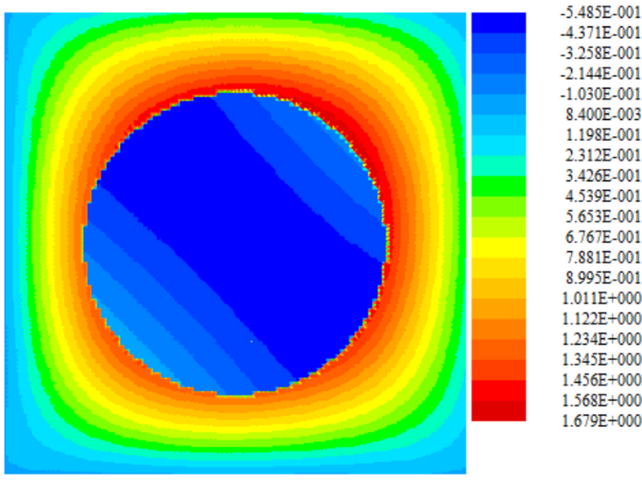}
    \includegraphics[width=.45\textwidth]{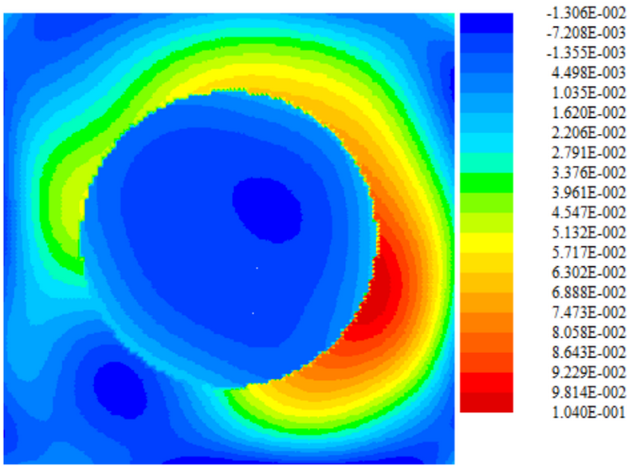}
\caption{Numerical solutions of the two-phase flow interface problem
($\rho_1=1$, $\rho_2 = 1000$, $\mu_1 = 1$, and $\mu_2 = 1000$ with $5$
observation points. Top left: the horizontal velocity; Top right:
error of the horizontal velocity; Middle left: the vertical
velocity;  Middle right: error of the vertical velocity; Bottom
left: the pressure; Bottom right: error of the pressure.}
\label{fig:NSNS-1vs1000-5p}
\end{figure}

\begin{figure}[h!]
	\centering
	\includegraphics[width=.33\textwidth]{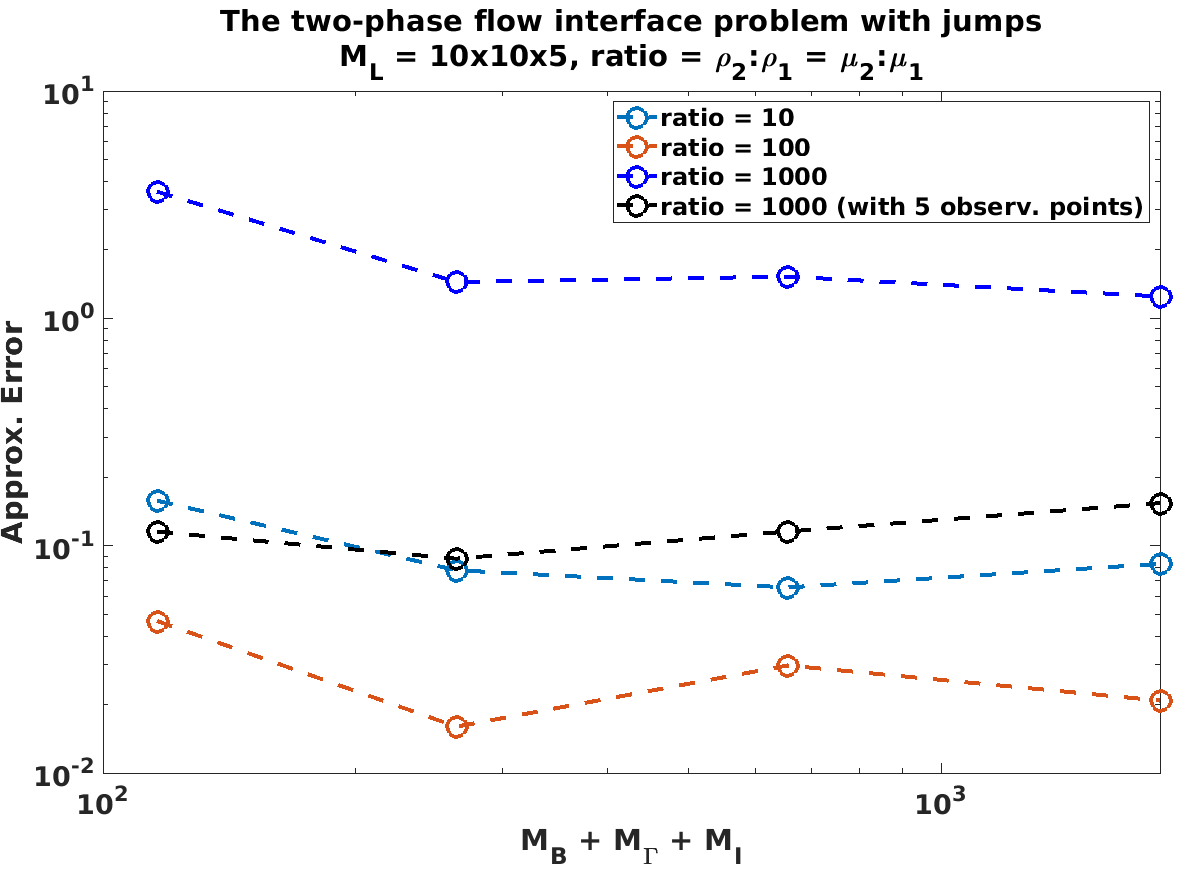} 
	\includegraphics[width=.33\textwidth]{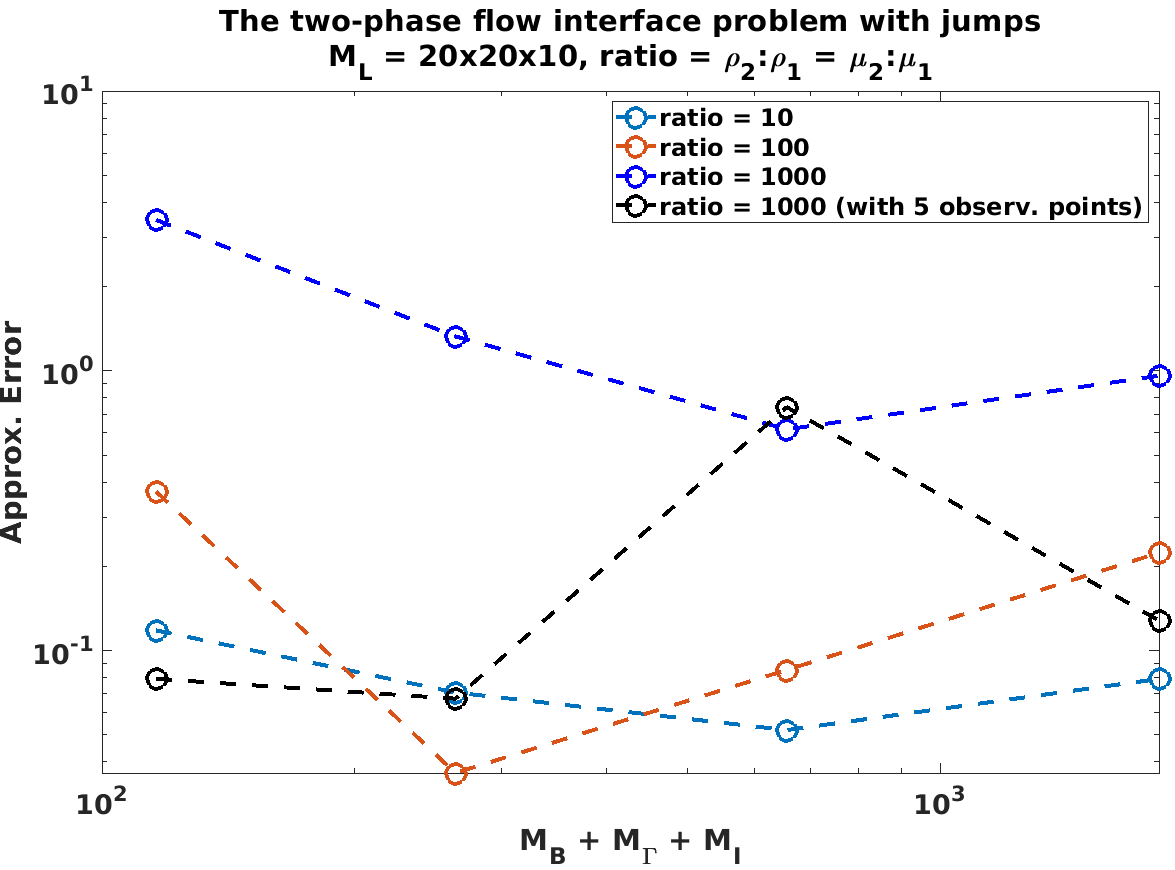} 
	\includegraphics[width=.33\textwidth]{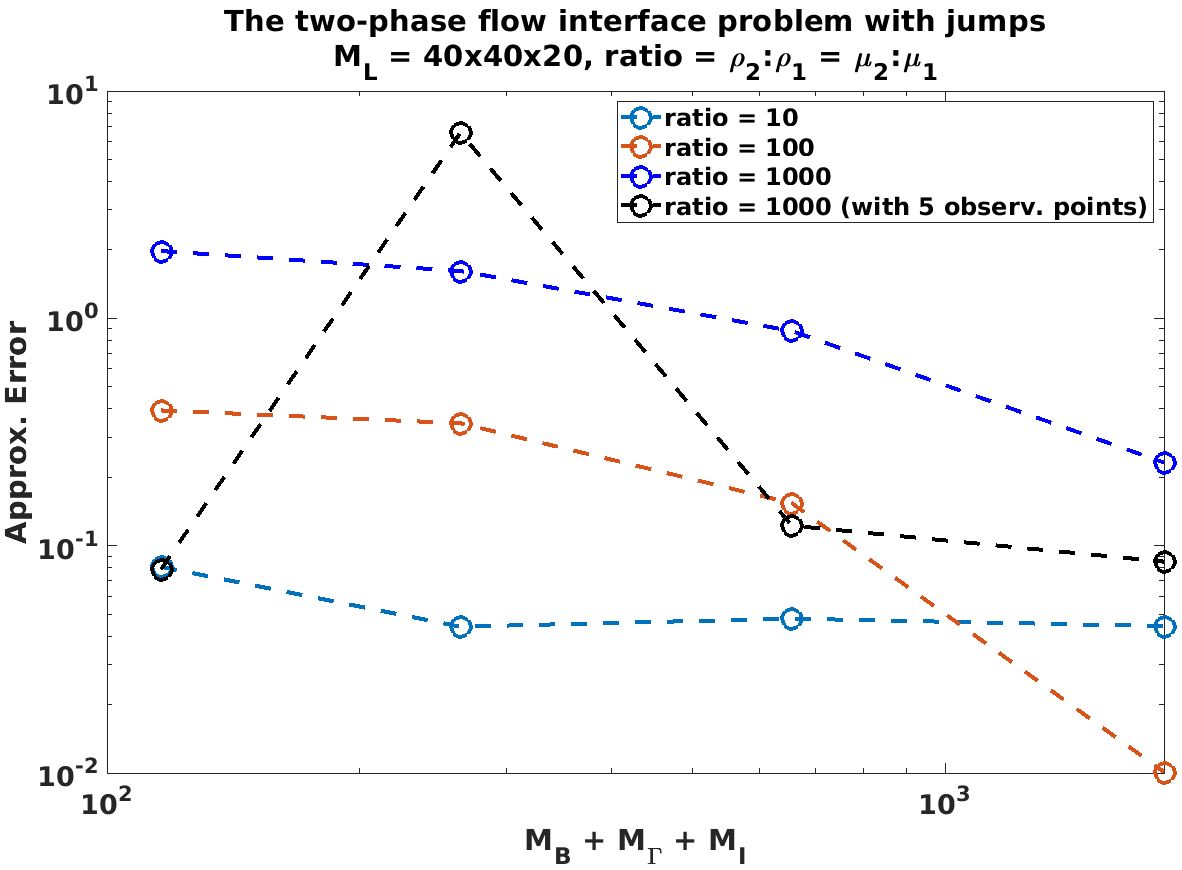}
	\caption{Approximations to the two-phase flow interface problems on different sampling points sets. Left: $M_{\mathcal{L}_i} = 10 \times 10 \times 5$; Middle: $M_{\mathcal{L}_i} = 20 \times 20 \times 10$; Right: $M_{\mathcal{L}_i} = 40 \times 40 \times 20$.} \label{fig:two-phase-flow-jump}
\end{figure}

In~\cite{krishnapriyan2021characterizing}, the authors suggest some possible approaches to make the PINN method more robust. Here, we use a simple approach suggested in~\cite{karniadakis2021physics}, i.e., adopt a few observation points in our experiments. In particular, we add $5$ observation points to $\Omega_2$ that are symmetrically located at~$(1.5+0.5\cos45^\circ,1.5+0.5\sin45^\circ)$, $(1.5-0.5\cos45^\circ,1.5+0.5\sin45^\circ)$, $(1.5-0.5\cos45^\circ,1.5-0.5\sin45^\circ)$, $(1.5+0.5\cos45^\circ,1.5-0.5\sin45^\circ)$, and~$(1.5,1.5)$, considering that the approximation error of the pressure gets worse in $\Omega_2$. In other words, we add LS formulations of a pointwise Dirichlet condition of the pressure adhering to those five observation points in $\Omega_2$ to the total loss functional and then train them together with other LS formulations within the same DNN structures. Corresponding numerical results are presented in Table~\ref{tab:NSNS-1vs1000-5p} and Figure~\ref{fig:NSNS-1vs1000-5p}. By adding those five observation points to the case of $\rho_2/\rho_1=\mu_2/\mu_1=1000$, we can significantly reduce the pressure's approximation errors, which overall decreases total approximation errors. In fact, the magnitude of those approximation errors are comparable with the no-jump case of~$\rho_1=\rho_2=1$ and~$\mu_1=\mu_2=1$. Figure~\ref{fig:two-phase-flow-jump} further illustrates the corresponding DNN's approximation trends for four scenarios of jump coefficients based on different sampling points sets, where we clearly observe that the DNN's approximation errors increase along with the increase of the jump coefficients' ratios while decreasing along with the increasing of the sampling points sets' sizes in the interior space-time domain. Additionally, the DNN's approximation errors also closely respond to the increase in the number of sampling points on the interface, boundary, and initial subdomains by showing a decreasing trend, to some extent.

Since the main purpose of this work is to develop and investigate the DNN/meshfree method for solving two-phase flow interface problems in the initial stage, a more detailed study about how to make the method more robust, especially when the solution is singular, is the subject of our ongoing work and will be reported in the future.

\subsection{Example 2: The FSI problems}
Next, we investigate the numerical performance of our developed DNN/meshfree method for solving the FSI problems in two forms, i.e., the example introduced in Section~\ref{sec:FSI-example}. We use the same computational space-time domain as defined in Section~\ref{num:example1}. For the original form of the FSI problem, we properly choose~$\bm{f}_f$, $\bm{f}_s$, $\bm{g}_1$, $\bm{g}_2$, $\bm{v}_f^b$, $\bm{u}_s^b$, $\bm{v}_s^b$, $\bm{v}_f^0$, $\bm{u}_s^0$ and $\bm{v}_s^0$ to guarantee the following functions are exact solutions to both forms of the presented FSI problem,
\begin{align*}
    &\bm{v}_f =
    \begin{pmatrix}
        e^t \, \sin x \, \cos y \\
        -e^t \, \cos x \, \sin y
    \end{pmatrix}, \quad
    p_f =   e^t \, \sin x \, \sin y, \\
    &\bm{u}_s =
    \begin{pmatrix}
        \cos t \, \cos x \, \cos y \\
        \sin t \, \sin x \, \sin y
    \end{pmatrix}.
\end{align*}
In addition, for the FSI problem with a parabolic-like structural equation, beside the above exact solutions, the structural velocity is also needed, given by
\begin{equation*}
    \bm{v}_s =
    \begin{pmatrix}
        -\sin t \, \cos x \, \cos y \\
        \cos t \, \sin x \, \sin y
    \end{pmatrix},
\end{equation*}
which is the time derivative of $\bu_s$. Again, the above exact solution leads to the jump-type kinematic- and dynamic interface conditions with $\bg_i\neq 0,\ i=1,2$. As for the physical parameters, we choose~$\rho_f = 1$, $\mu_f = 1$, $\rho_s=10^3$, Young's modulus~$E=10^6$, and Poisson ratio~$\nu = 0.3$ $\left(\text{thus, } \mu_s=\frac{E}{2(1+\nu)},\ \lambda_s=\frac{E\nu}{(1+\nu)(1-2\nu)}\right)$, as an example to illustrate the performance of the developed DNN/meshfree method, which shows a high-contrast coefficients scenario.

\begin{table}[h!]
    \centering
    \caption{The FSI problem with a wave-type structural equation} \label{tab:FSI-wave}
    \begin{tabular}{ c c c c c c}
        \hline \hline
        $M_{\mathcal{L}_i}$ & $M_{\mathcal{B}_i}$ & $M_{\Gamma}$&  $M_{\mathcal{I}_i}$  &  Approx. Error & Loss Error \\ \hline
        $10 \times 10 \times 5$ & $4 \times 4 \times 5$ & $4 \times 5$  & $4 \times 4$  & 2.009e-02  & 4.162e-05 \\
        $10 \times 10 \times 5$ & $8 \times 4 \times 5$ & $8 \times 5$  & $8 \times 8$  & 1.159e-02  & 7.789e-05 \\
        $10 \times 10 \times 5$ & $16 \times 4 \times 5$ & $16 \times 5$  & $16 \times 16$  & 9.136e-03  & 5.764e-05 \\
        $10 \times 10 \times 5$ & $32 \times 4 \times 5$ & $32 \times 5$  & $32 \times 32$  & 1.187e-02  & 6.383e-05 \\ \hline
        $20 \times 20 \times 10$ & $4 \times 4 \times 10$ & $4 \times 10$  & $4 \times 4$  & 1.347e-02  & 4.427e-05 \\
        $20 \times 20 \times 10$ & $8 \times 4 \times 10$ & $8 \times 10$  & $8 \times 8$  & 1.647e-02  & 5.643e-05 \\
        $20 \times 20 \times 10$ & $16 \times 4 \times 10$ & $16 \times 10$  & $16 \times 16$  & 1.097e-02  & 7.966e-05 \\
        $20 \times 20 \times 10$ & $32 \times 4 \times 10$ & $32 \times 10$  & $32 \times 32$  & 9.871e-02  & 6.993e-05 \\ \hline
        $40 \times 40 \times 20$ & $4 \times 4 \times 20$ & $4 \times 20$  & $4 \times 4$  & 1.611e-02  & 4.970e-05 \\
        $40 \times 40 \times 20$ & $8 \times 4 \times 20$ & $8 \times 20$  & $8 \times 8$  & 1.135e-02  & 6.213e-05 \\
        $40 \times 40 \times 20$ & $16 \times 4 \times 20$ & $16 \times 20$  & $16 \times 16$  & 1.028e-02  & 6.829e-05 \\
        $40 \times 40 \times 20$ & $32 \times 4 \times 20$ & $32 \times 20$  & $32 \times 32$  & 1.079e-02  & 6.812e-05 \\
        \hline \hline
    \end{tabular}
\end{table}

\begin{figure}[h!]
	\centering	
	\includegraphics[width=.7\textwidth]{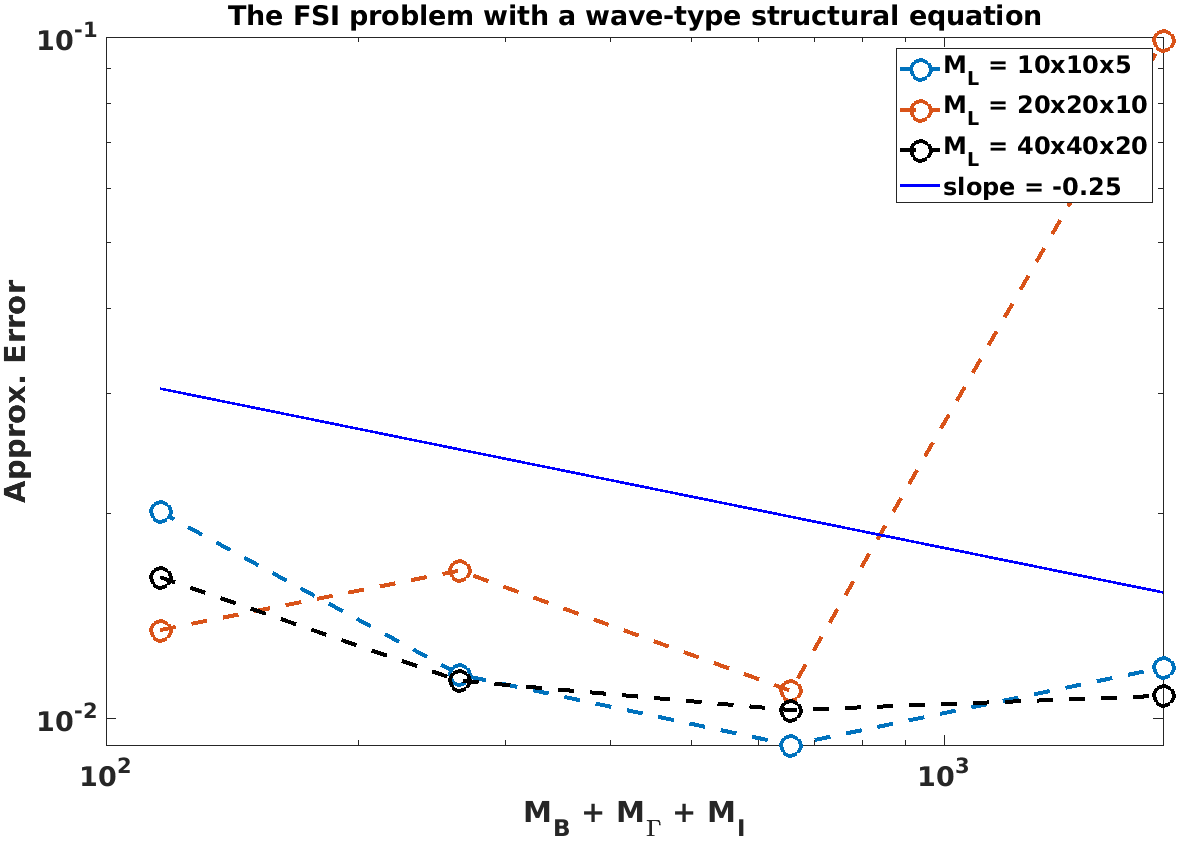}
	\caption{Approximation errors of the FSI problem with a wave-type structural equation.} \label{fig:table6}
\end{figure}

\begin{figure}[h!]
    \centering
    \includegraphics[width=.45\textwidth]{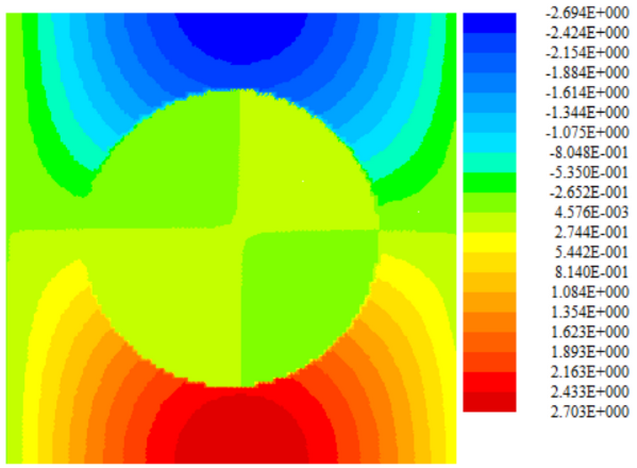}
    \includegraphics[width=.45\textwidth]{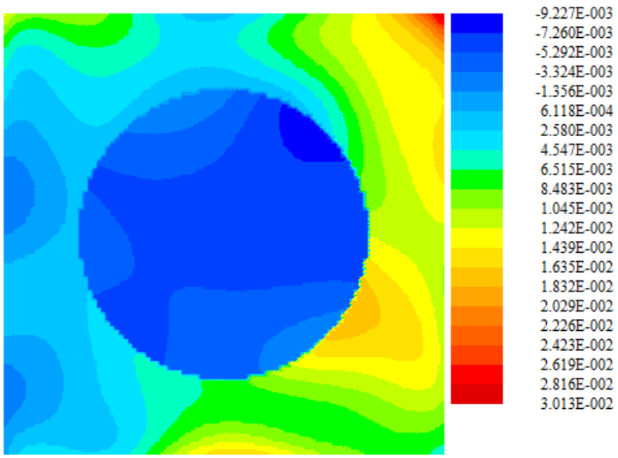} \\
    \includegraphics[width=.45\textwidth]{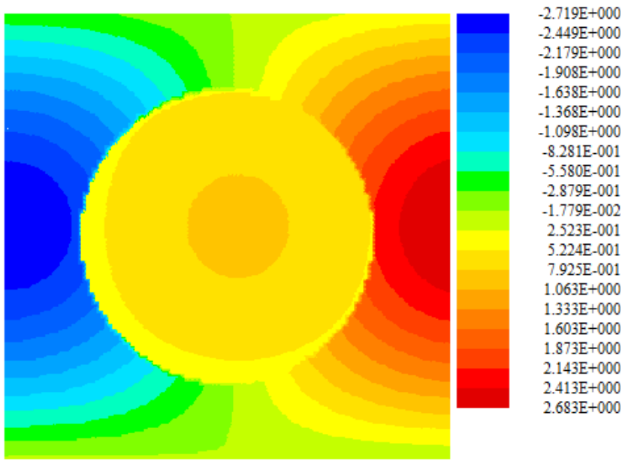}
    \includegraphics[width=.45\textwidth]{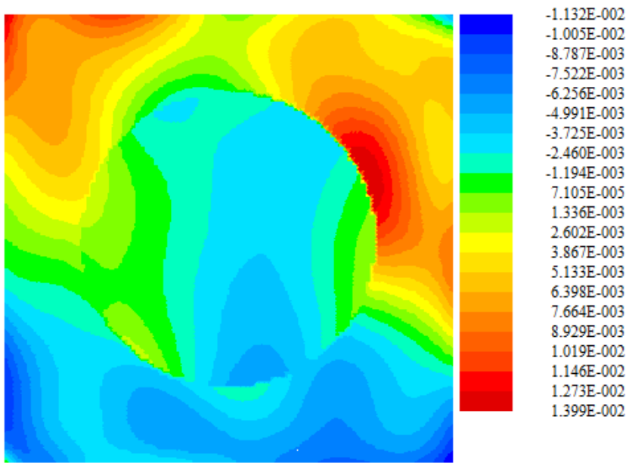} \\
    \includegraphics[width=.45\textwidth]{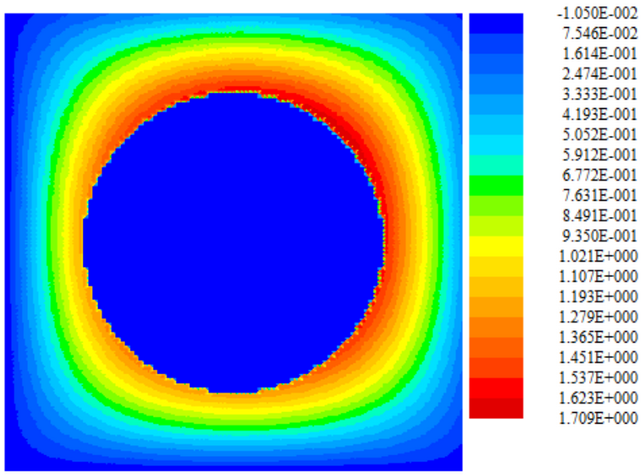}
    \includegraphics[width=.45\textwidth]{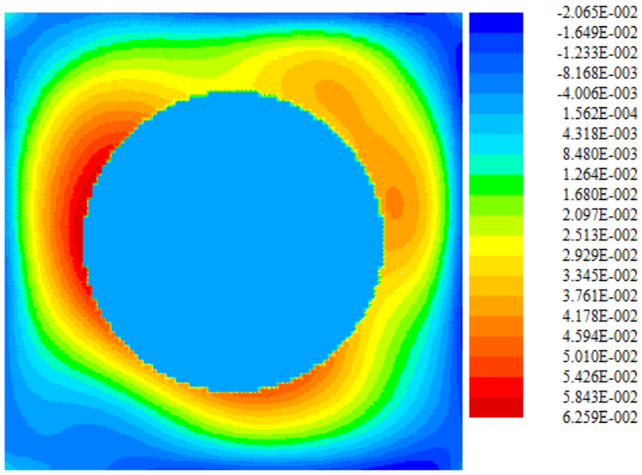}
\caption{Numerical solutions of the FSI problem with a wave-type
structural equation ($\rho_f = 1$, $\mu_f = 1$, $ \rho_s=10^3$,
$E=10^6$, and $\nu = 0.3$). Top left: the horizontal velocity; Top
right: error of the horizontal velocity; Middle left: the vertical
velocity;  Middle right: error of the vertical velocity; Bottom
left: the pressure; Bottom right: error of the pressure.}
\label{fig:FSI-wave}
\end{figure}

\begin{table}[h!]
    \centering
    \caption{The FSI problem with a parabolic-like structural equation} \label{tab:FSI-parabolic}
    \begin{tabular}{ c c c c c c}
        \hline \hline
        $M_{\mathcal{L}_i}$ & $M_{\mathcal{B}_i}$ & $M_{\Gamma}$&  $M_{\mathcal{I}_i}$  &  Approx. Error & Loss Error \\ \hline
        $10 \times 10 \times 5$ & $4 \times 4 \times 5$ & $4 \times 5$  & $4 \times 4$  & 1.893e-02  & 4.786e-05 \\
        $10 \times 10 \times 5$ & $8 \times 4 \times 5$ & $8 \times 5$  & $8 \times 8$  & 1.453e-02  & 8.117e-05 \\
        $10 \times 10 \times 5$ & $16 \times 4 \times 5$ & $16 \times 5$  & $16 \times 16$  & 9.986e-03  & 4.922e-05 \\
        $10 \times 10 \times 5$ & $32 \times 4 \times 5$ & $32 \times 5$  & $32 \times 32$  & 7.991e-03  & 6.658e-05 \\ \hline
        $20 \times 20 \times 10$ & $4 \times 4 \times 10$ & $4 \times 10$  & $4 \times 4$  & 1.761e-02  & 4.988e-05 \\
        $20 \times 20 \times 10$ & $8 \times 4 \times 10$ & $8 \times 10$  & $8 \times 8$  & 9.067e-03  & 7.050e-05 \\
        $20 \times 20 \times 10$ & $16 \times 4 \times 10$ & $16 \times 10$  & $16 \times 16$  & 7.414e-03  & 7.212e-05 \\
        $20 \times 20 \times 10$ & $32 \times 4 \times 10$ & $32 \times 10$  & $32 \times 32$  & 7.103e-03  & 5.078e-05 \\ \hline
        $40 \times 40 \times 20$ & $4 \times 4 \times 20$ & $4 \times 20$  & $4 \times 4$  & 1.877e-02  & 4.399e-05 \\
        $40 \times 40 \times 20$ & $8 \times 4 \times 20$ & $8 \times 20$  & $8 \times 8$  & 8.362e-03  & 7.138e-05 \\
        $40 \times 40 \times 20$ & $16 \times 4 \times 20$ & $16 \times 20$  & $16 \times 16$  & 1.283e-02  & 6.055e-05 \\
        $40 \times 40 \times 20$ & $32 \times 4 \times 20$ & $32 \times 20$  & $32 \times 32$  & 8.833e-03  & 5.842e-05 \\
        \hline \hline
    \end{tabular}
\end{table}

\begin{figure}[h!]
	\centering	
	\includegraphics[width=.7\textwidth]{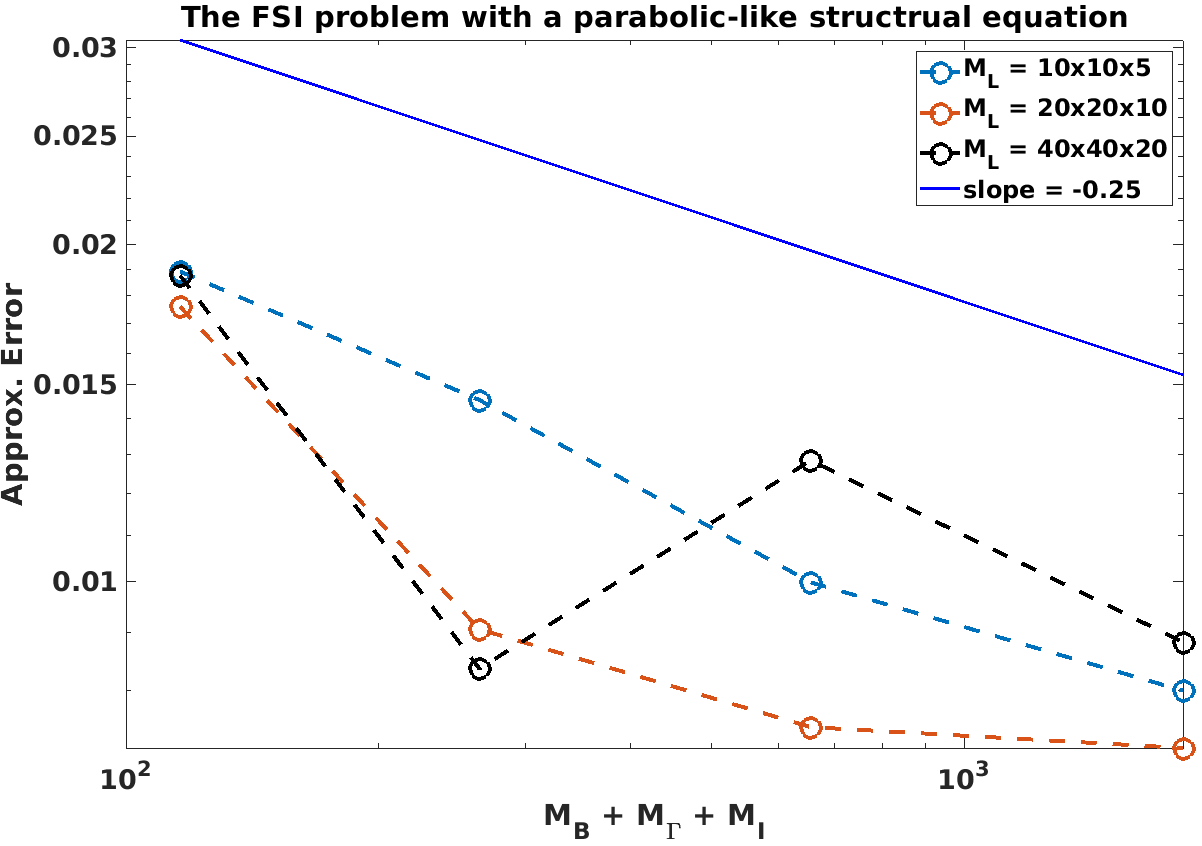}
	\caption{Approximation errors of the FSI problem with a parabolic-like structural equation.} \label{fig:table7}
\end{figure}

\begin{figure}[h!]
    \centering
    \includegraphics[width=.45\textwidth]{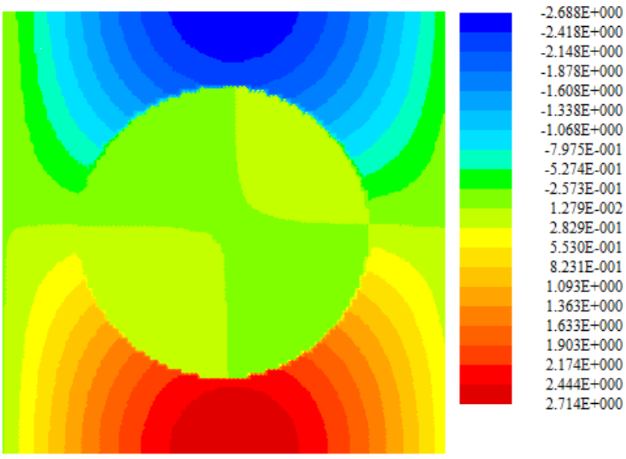}
    \includegraphics[width=.45\textwidth]{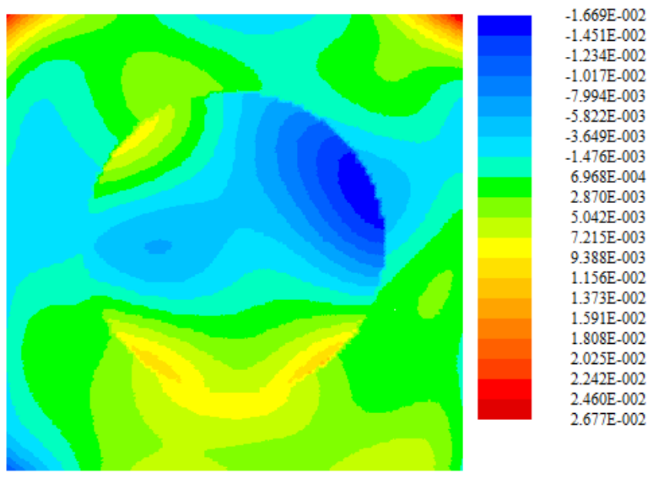} \\
    \includegraphics[width=.45\textwidth]{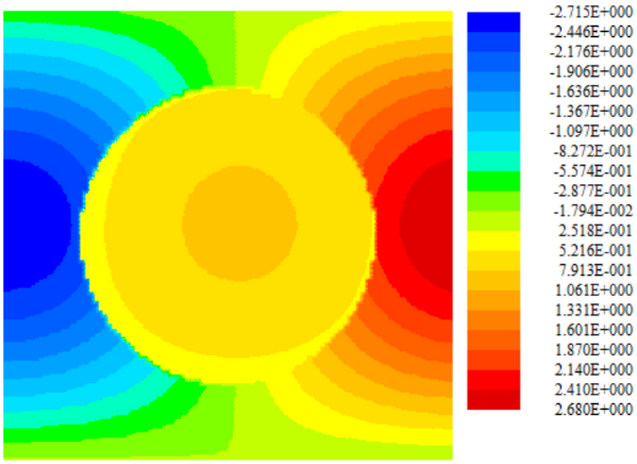}
    \includegraphics[width=.45\textwidth]{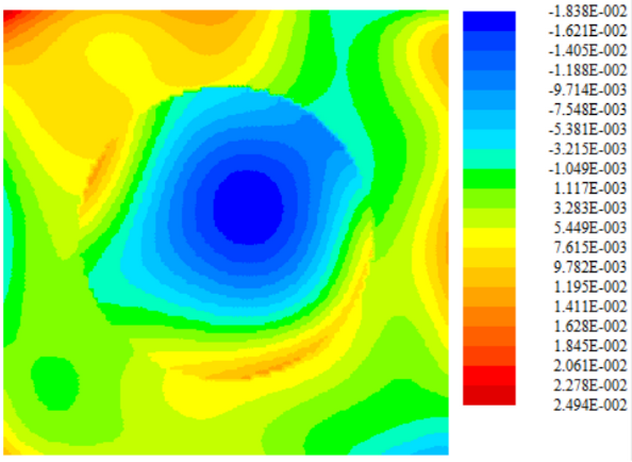} \\
    \includegraphics[width=.45\textwidth]{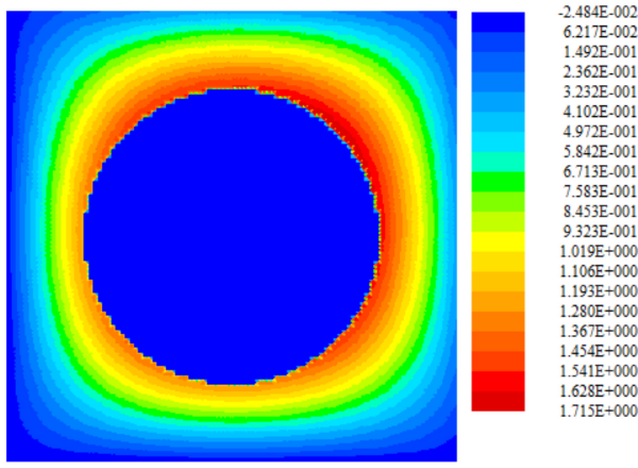}
    \includegraphics[width=.45\textwidth]{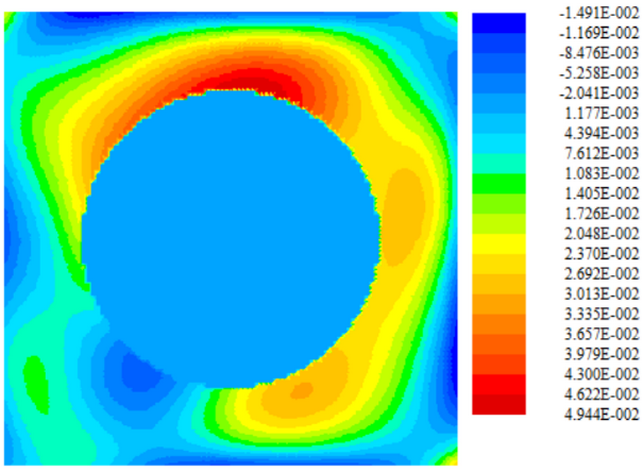}
\caption{Numerical solutions of the FSI problem with a
parabolic-like structural equation ($\rho_f = 1$, $\mu_f = 1$,
$ \rho_s=10^3$, $E=10^6$, and $\nu = 0.3$). Top left: the horizontal
velocity; Top right: error of the horizontal velocity; Middle left:
the vertical velocity;  Middle right: error of the vertical
velocity; Bottom left: the pressure; Bottom right: error of the
pressure.}\label{fig:FSI-parabolic}
\end{figure}

In Tables~\ref{tab:FSI-wave} and \ref{tab:FSI-parabolic}, we report numerical results of the developed DNN/meshfree method for the FSI problem with a wave-type and a parabolic-like structural equation, respectively, where we still see a similar convergence behavior compared with that of the two-phase flow interface problem, i.e., the approximation error decreases first when we increase the number of sampling points,~$M_{\Gamma}$, $M_{\mathcal{B}_i}$ and~$M_{\mathcal{I}_i}\ (i=f,s)$ that are associated with the interface, the boundaries, and the initial subdomains, respectively, while keeping the number of interior sampling points, $M_{\mathcal{L}_i}\ (i=f,s)$, the same, then it might stagnate or even increase, which implies that the quadrature error of the interior space-time domain terms retakes the dominance along with large enough $M_{\Gamma}$, $M_{\mathcal{B}_i}$ and~$M_{\mathcal{I}_i}\ (i=f,s)$. In Figure~\ref{fig:table6} and Figure~\ref{fig:table7}, we again plot the reference line with the absolute value of the slope equaling $0.25$ to demonstrate the approximation errors roughly follow the expected convergence order for both cases based upon different sizes of the sampling points sets, where a decreasing trend of the DNN's approximation errors is also observed, to some extent, along with the increasing number of sampling points on the interface, the boundary, and initial subdomains. Numerical solutions and corresponding approximation errors are plotted in Figures~\ref{fig:FSI-wave} and~\ref{fig:FSI-parabolic}, respectively, where we can observe that numerical solutions are relatively accurate since the approximation errors are small. As expected, the largest part of approximation errors usually lies near the interface and the boundary, which is consistent with our theoretical results derived in Section~\ref{sec:error}, and an enrichment of sampling points on the interface and boundaries can reduce approximation errors significantly.

\section{Conclusions} \label{sec:conclusions}
Overall, in this paper, we consider two kinds of dynamic two-phase interface problems that consist of two different dynamic PDEs with the jump and high-contrast coefficients: (1) two-phase flow interface problems and (2) fluid-structure interaction (FSI) problems with a wave-type and a parabolic-like structural equation, where both PDEs are defined on either side of the interface and are interacted with each other through interface conditions across a stationary interface. We develop a meshfree method using the deep neural networks (DNN) approach to tackle two kinds of dynamic interface problems. Similar to the PINN method, we first reformulate the interface problem as a least-squares (LS) problem and approximate the solution functions via DNN functions, but for two-phase interface problems, in particular, we must adopt different DNN structures in different subdomains separated by interfaces. Sampling points are used to numerically compute integrals arising from LS formulations and replace the entire loss functional with its discrete counterpart. Then the gradient-based method is used to solve the discrete LS problem numerically. We first derive an error analysis for an abstract two-phase interface model problem, then for two kinds of two-phase interface problems. The developed approximation theory provides an interesting strategy for efficiently choosing the sampling points to improve the overall approximation accuracy of the developed DNN/meshfree method for solving the presented two-phase interface problems. That is, more data points shall be sampled on the interfaces, the boundaries, and in the initial subdomains besides the interior space-time domain in order to reduce the entire approximation error. Numerical experiments demonstrate the effectiveness of the developed DNN/meshfree methods for two kinds of two-phase interface problems, and their theoretical results are also verified to some extent.

In this work, we only consider two-phase interface problems with stationary interfaces as the first step in this field. Our new research on developing the DNN/meshfree method for solving moving interface problems such as the realistic FSI problems has been ongoing. In addition, to further improve the approximation accuracy of the DNN/meshfree approach for two-phase flow interface problems, we are considering to develop a new DNN/meshfree method based upon the weak formulation of the original interface problem so that the Neumann-type boundary- and interface conditions can be handled naturally. Finally, we are also considering to sample the data points adaptively in order to achieve more efficient and more accurate algorithms in practice, especially for low-regularity and large-scale simulations in our future work.

\section*{Acknowledgments}
X.Zhu was partially supported by the Young Scientists Fund of National Natural Science Foundation of China (Grant No. 51809026 and No. 61902049) and Joint Special Fund for Basic Research of Local Undergraduate Universities (Partially) in Yunnan Province (Grant No. 202001BA070001-082). P. Sun was supported by a grant from the Simons Foundation (MPS-706640, PS).

\bibliographystyle{plain}
\bibliography{DNN4TwoPhase,FSI,DLS}
\end{document}